\documentclass{Andrea}
\usepackage[latin1]{inputenc}
\usepackage{amsmath}
\usepackage{amscd}
\usepackage{mathrsfs}
\usepackage[all]{xy}
\usepackage{graphicx}
\usepackage{pstricks} 
\usepackage{amsfonts}   
\usepackage{pict2e}   
\usepackage{makeidx}
\usepackage[colorlinks=true,linkcolor=red,citecolor=red]{hyperref}

\usepackage{amssymb}
\usepackage{latexsym}
\usepackage{bold-extra}

\makeindex

\renewcommand{\a}{\mathcal{A}}
\renewcommand{\xi}{x}
\newcommand{\bs}[1]{\boldsymbol{#1}}
\newcommand{\C}{\mathscr{C}}

\newcommand{\D}{\mathrm{D}}

\newcommand{\F}{F}

\newcommand{\FM}{\mathcal{R}^{\M}}
\newcommand{\FY}{\mathcal{R}^{Y}}
\newcommand{\RR}{\bs{\mathcal{R}}}

\newcommand{\RM}{\mathcal{R}^{\Fs}}

\renewcommand{\H}{\mathscr{H}}
\newcommand{\HH}{\bs{H}}
\newcommand{\hh}{\bs{h}}

\newcommand{\Ka}{\widehat{K^{\mathrm{alg}}}}
\newcommand{\M}{\mathrm{M}}
\renewcommand{\L}{\mathcal{L}}

\newcommand{\N}{\mathrm{N}}
\renewcommand{\O}{\mathscr{O}}

\newcommand{\R}{\mathcal{R}}

\renewcommand{\ss}{\bs{s}}
\newcommand{\Sol}{\mathrm{Sol}}
\renewcommand{\S}{\Gamma}

\newcommand{\sm}[1]{\begin{smallmatrix}#1\end{smallmatrix}}

\newcommand{\simto}{\xrightarrow{\sim}}

\newcommand{\X}{X}

\newcommand{\Fs}{\ensuremath{\mathscr{F}}}

\usepackage{amsthm}
\swapnumbers

\makeatletter
\def\swappedhead#1#2#3{%
  \thmname{#1}\;%
  \thmnumber{\@upn{\the\thm@headfont#2\@ifnotempty{#1}}}%
  \thmnote{\,{\the\thm@notefont(#3)}}{.~}}
\makeatother

\newtheoremstyle{dotless-thm}
  {10pt}
  {10pt}
  {\itshape}
  {}
  {\bfseries}
  {}
  {.0em}
  {}
\theoremstyle{dotless-thm}

\newtheorem{theorem}{\textbf{Theorem}}[subsection]
\newtheorem{thm-intro}{\textbf{\textsc{Theorem}}}
\newtheorem{Def-intro}[thm-intro]{\textbf{\textsc{Definition}}}
\newtheorem{rk-intro}[thm-intro]{\textbf{\textsc{Remark}}}
\newtheorem{cor-intro}[thm-intro]{\textbf{\textsc{Corollary}}}
\newtheorem{Crit-intro}[thm-intro]{\textbf{\textsc{Criterion}}}
\newtheorem{proposition}[theorem]{\textbf{Proposition}}
\newtheorem{lemma}[theorem]{\textbf{Lemma}}
\newtheorem{corollary}[theorem]{\textbf{Corollary}}
\newtheorem{definition}[theorem]{\textbf{Definition}}
\newtheorem{remark}[theorem]{\textbf{Remark}}
\newtheorem{example}[theorem]{\textbf{Example}}
\newtheorem{hypothesis}[theorem]{\textbf{Hypothesis}}

\newtheorem{notation}[theorem]{\textbf{Notation}}
\numberwithin{equation}{section}

\title[Convergence Newton 
polygon I : the affine line]{
The convergence Newton polygon of a $p$-adic 
differential equation I : Affinoid domains of the Berkovich affine 
line}

\author{Andrea Pulita}
\email{pulita@math.univ-montp2.fr}
\address{Département de Mathématiques,
Université de Montpellier II, CC051,
Place Eugène Bataillon,
F-34 095, Montpellier CEDEX 5.}

\date{\today}



\begin{abstract}
\if{
We prove the finiteness of the convergence Newton polygon of a 
$p$-adic differential equation $\Fs$ over an affinoid domain $X$ 
of the Berkovich affine line. Roughly speaking the slopes of the polygon
are the logarithms of the radii of convergence of the solutions of $\Fs$, 
in increasing order, counted with multiplicity. 
We prove that the radii of convergence of the solutions 
are continuous functions on $X$ that factorize through the retraction of 
$X\to\Gamma$ of $X$ onto a finite graph $\Gamma\subseteq X$. 
This finiteness result means that the behavior of the radii as 
functions on $X$ is controlled by a \emph{finite} family of data.
}\fi
We prove that the radii of convergence of the solutions of a 
$p$-adic differential equation $\Fs$ over an affinoid domain $X$ of the Berkovich affine line
are continuous functions on $X$ that factorize through the retraction of 
$X\to\Gamma$ of $X$ onto a finite graph $\Gamma\subseteq X$. 
We also prove their super-harmonicity properties.
This finiteness result means that the behavior of the radii as 
functions on $X$ is controlled by a \emph{finite} family of data.
\end{abstract}

\begin{document}
\maketitle

\if{
\begin{center}
Version of \today
\end{center}

\makeatletter
\renewcommand\tableofcontents{%
    \subsection*{\contentsname}%
    \@starttoc{toc}%
    }
\makeatother

\begin{small}
\setcounter{tocdepth}{2} \tableofcontents
\end{small}
}\fi

\setcounter{section}{0}

%\newpage\newpage
\section*{Introduction}
\addcontentsline{toc}{section}{\textsc{Introduction}}
%The classification of the $p$-adic differential equations 
%pass through the knowledge of their solutions.

The fundamental work of Christol and Mebkhout 
\cite{Ch-Me-I}, \cite{Ch-Me-II}, \cite{Ch-Me-III}, \cite{Ch-Me-IV}, 
together with \cite{An}, \cite{Me}, \cite{Ked},  
achieved a program (firstly initiated by P.Robba and B.Dwork 
\cite{DW-II}, \cite{RoI}, \cite{Dw-Robba}, \cite{Robba-Hensel},  
\cite{Ch-Dw}, \ldots) concerning differential equations  
``\emph{coming from rigid cohomology}''. 
These differential equations have maximal radius of 
convergence at the ``\emph{generic point}'', and have 
over-convergent coefficients. 

This paper, and its sequel \cite{NP-II}, deal with a more general 
program concerning locally free 
$\O_X$-modules with connection, over a rig-smooth 
$K$-analytic Berkovich curve $X$, with no conditions on the size of 
the radii of convergence. 

In this context there is a lack of results of global nature in the 
sense of Berkovich. Even the case of an open disk is not well 
understood. 
The existing results on curves mainly concern 
differential modules over a field of 
power series at a rational point, 
or over the so called Robba ring. 
From the point of view of Berkovich curves this means a germ 
of segment out of a point of type $1$, $2$, or $3$.  

The most basic, but central tool of the theory is the so called 
\emph{convergence Newton polygon} of a $p$-adic differential 
equation.\footnote{We consider $p$-adic 
differential equation in  the large sense of the word. 
This also covers the case of any ultrametric 
complete valued field $K$. 
In particular we treat the case of differential modules over 
the field of formal power series $K(\!(T)\!)$, 
where $K$ is trivially valued. 
The formal Newton polygon (in the sense of B.Malgrange 
and J.P.Ramis) is in fact the derivative of the convergence Newton 
polygon (cf. Remark \ref{Rk : Formal case}).}
Roughly speaking the slopes of that polygon at $x\in X$ 
are the logarithms of the radii of convergence of the 
Taylor solutions at $x$, in increasing order, counted with multiplicity 
(cf. Definition \ref{Def: intro NP} below).
The continuity of the convergence Newton polygon, 
as a function on $X$, appears in this program as 
the fundamental step, and the major tool in the classification of the 
equations, as illustrated (in the solvable case) by the 
work of G.Christol and Z.Mebkhout. 

Moreover the convergence Newton 
polygon carries important numerical invariants 
of the equation, in analogy with the Swan conductor, that are 
\emph{highly related to the residual wild ramification} in the spirit of 
\cite{Ma}, \cite{Tsu-Swan}, \cite{Crew-Can-ext}, 
\cite{An}, \cite{Adri-Swan}, \cite{Rk1-non-perf}.

%Understanding of the behavior of the equations  
%by a morphism $f:Y\to X$ between rig-
%smooth $K$-analytic Berkovich curves. 

In the more global setting of Berkovich curves 
there is an additional geometrical datum 
furnished by the convergence Newton polygon:  
a graph 
$\Gamma\subseteq X$, called \emph{controlling graph} of the 
differential equation. 
Roughly speaking  $\Gamma$ is a locally finite graph, such that
$X-\Gamma$ is a disjoint union of virtual open disks on which the
polygon is constant. 
So, by continuity, the behavior of the polygon as a function on $X$ 
is determined by its restriction to $\Gamma$. 

As an example, if $f:Y\to X$ is an étale morphism 
between rig-smooth $K$-analytic Berkovich curves,
the controlling graph $\Gamma$ of $f_*(\O_Y)$, and more precisely 
the derivative of the polygon as a function on $\Gamma$, 
are all invariants of the morphism, highly related to its residual 
wild ramification.\\

The main goals of this paper, and of its sequel \cite{NP-II} 
%(cf. also \cite{Potentiel}), 
are the following:
\begin{enumerate}
\item An unconditional definition of the convergence Newton 
polygon, based on \cite{Balda-Inventiones},  not involving 
formal models, and resulting completely \emph{within the 
framework of Berkovich analytic spaces};
\item \emph{The continuity} of each slope of the polygon, as 
a function on $X$;
\item \emph{The local finiteness} of the graph $\Gamma$.
\end{enumerate}
In this paper, we focus on the case of affinoid domains of the affine 
line. A great part of the literature about $p$-adic differential equations 
is devoted to the affine line. So this case has its own interest, and it is 
important to treat it explicitly and completely. Point i) is not really 
relevant in this setting, but we prove that points ii) and iii) hold, by 
using techniques from $p$-adic differential equations.

In \cite{NP-II}, we extend those results to arbitrary smooth curves, by using techniques from Berkovich geometry to reduce to the case treated in this paper.

\if{The first point will be achieved in \cite{NP-II}. 
%(below we explain it in a basic way). 
We also prove in \cite{NP-II} that ii) and iii) reduce 
to the case of an affinoid domain of the affine line, 
the case treated in this paper. 
Great part of the literature about $p$-adic 
differential equations is devoted to the affine line. 
So this case has its own interest, and it is important to treat it explicitly 
and completely. }\fi

We prove that 
the controlling graph $\Gamma$ 
of a differential equation is always a \emph{locally finite graph}
without particular assumptions 
(no solvability, no exponents, no Frobenius, \ldots). 
This implies that the entire convergence Newton 
polygon is determined, as a function on $X$, 
by a \emph{locally finite family of numbers} 
(finite in the case of this paper).

The continuity of the polygon (which is a consequence of the local 
finiteness of $\Gamma$) is the major ingredient 
for decomposition theorems of  global nature \cite{NP-III}. 
The finiteness of $\Gamma$, also represents the fundamental point 
permitting a computation of the de Rham cohomology of the equation. 
In particular the global finiteness of $\Gamma$ 
is the crucial property that gives the 
finite dimensionality of the de Rham cohomology of the differential 
equation \cite{NP-IV}. 
These results was unknown even in the elementary case 
of a non solvable differential equation over a disk or an annulus. 

Essential ingredients are the work of 
K.S.Kedlaya about subsidiary radii \cite{Kedlaya-Book}, 
and that of F.Baldassarri and L.Di Vizio about 
the generic radius of convergence \cite{DV-Balda}, 
\cite{Balda-Inventiones}, where the finiteness of $\Gamma$ 
was originally conjectured. 
The work of Kedlaya is a determinant refinement of 
classical ideas (together with the introduction of the 
crucial notion of super-harmonicity), while
Baldassarri's work is a change in perspective 
which opened up a whole new line of investigation.
Namely, in several recent 
talks, Baldassarri conjectured that the radii should factorize through 
some unspecified \emph{finite} graph that he baptized 
\emph{controlling graph}. 
He also established a link between the graph and the cohomology, 
and suggested some partial idea of proofs. 
The conjecture was supported by effective computations obtained by 
Christol for rank one equations \cite{Ch-Formula} 
(cf. final notes \ref{Notes final}).\\

\if{This paper is based on that of 
K.S.Kedlaya \if{about subsidiary radii }\fi\cite{Kedlaya-Book}, 
and that of F.Baldassarri and L.Di Vizio \if{about 
the generic radius of convergence}\fi \cite{DV-Balda}, 
\cite{Balda-Inventiones}. 
The work of Kedlaya is a determinant refinement of 
classical ideas (together with the introduction of the 
crucial notion of super-harmonicity), while
Baldassarri's work \cite{Balda-Inventiones} is 
a change in perspective which opened up a whole new line of 
investigation.}\fi

We now enter more specifically in the contents of this paper.
In the introduction we assume by simplicity that the base field 
$K$ is algebraically closed. 
Let $X$ be an affinoid domain of the affine line, and let $x\in X$ be a 
Berkovich point. In order to define Taylor solutions ``at $x$'', we need to consider 
a field extension $\Omega/K$ where $x$ become rational.
Let $t\in X_\Omega$ be any rational point lifting $x$.
The fiber $\pi^{-1}(x)$ 
of the projection $\pi_{\Omega/K}:X_{\Omega}\to X$ has a nice structure: 
it has a peaked point $\sigma_{\Omega/K}(x)$ (cf. \cite[p.98]{Ber}) with the 
property that $\pi^{-1}_{\Omega/K}(x)-\{\sigma_{\Omega/K}(x)\}$ is a disjoint union of 
open disks, all having $\sigma_{\Omega/K}(x)$ as a relative boundary in 
$X_\Omega$. 

We call these disks \emph{Dwork generic disks}. Up to further extends 
the ground field $K$, they are all isomorphic, and independent on $X$. 
We call $D(x)$ the one of them containing $t$. 

We now introduce 
the \emph{maximal disk $D(x,X)$}. This is the largest open disk in $X_\Omega$ 
containing $t$. 

The topological structure of $X$ is the following. The 
set of points without neighborhoods isomorphic to an open disk form a 
finite graph %\footnote{In fact this is a tree in our case.} 
$\Gamma_X\subseteq X$ (called the analytic skeleton of $X$) 
such that $X-\Gamma_X$ is a disjoint union of open 
disks. All these disks are the maximal disks of their points. 
While the maximal disk of a point in $\Gamma_X$ is by definition its Dwork generic 
disk $D(x)$.

Now fix a coordinate $T:X\to\mathbb{A}^{1,\mathrm{an}}_K$, 
and call $r(x)$ and $\rho_{x,X}$ the radii of the generic disk $D(x)$, 
and of the maximal disk $D(x,X)$, respectively. 
Now let $\Fs$ be a locally free $\O_X$-module of finite rank $r$, 
endowed with a connection $\nabla:\Fs\to\Fs\otimes\Omega_X^1$. 
\begin{Def-intro}\label{Def: intro NP}
Denote by $\R^{\Fs}_i(x)\leq \rho_{x,X}$ 
the radius of the largest open disk centered at $t_x$, contained in $D(x,X)$, 
on which $\Fs$ has at least $r-i+1$ linearly independent solutions.
We define the $i$-th radius of convergence of $\Fs$ at $x$ as
\begin{equation}\label{eq: normalized radius}
\R_{i}(x,\Fs)\;:=\;\R^{\Fs}_i(x)/\rho_{x,X}\;\leq\; 1\;.
\end{equation}
The convergence Newton polygon of $\Fs$ is the polygon with slopes  
$\ln(\R_1(x,\Fs))\leq\cdots\leq \ln(\R_r(x,\Fs))$.

We call $i$-th \emph{spectral radius} of $\Fs$ at $x$ the number 
$\R_i^{\mathrm{sp}}(x,\Fs)=\min(\R_{i}(x,\Fs),r(x)/\rho_{x,X})$.  

We say that the index $i$ is \emph{spectral, solvable, or over-solvable} 
at $x$ if $\R_{i}(x,\Fs)\leq r(x)/\rho_{x,X}$, 
$\R_{i}(x,\Fs)= r(x)/\rho_{x,X}$, or 
$\R_{i}(x,\Fs)> r(x)/\rho_{x,X}$ respectively.
\end{Def-intro}
The function $x\mapsto \min(\R^{\Fs}_1(x),r(x))$ is the ancient 
definition of spectral radius of 
\cite{Ch-Dw}, \cite{RoI}, \cite{RoIV}, \cite{Pons}, \cite{Ch-Me-III}, 
\ldots 
while $\R^{\Fs}_1(x)$ is the radius of convergence of 
\cite{DV-Balda}. The normalized definition 
\eqref{eq: normalized radius} is that of \cite{Balda-Inventiones}, 
and it has the merit of being independent on the coordinate. 
For $i\geq 2$ the definition is due  
K.S.Kedlaya \cite{Kedlaya-Book} (following Young 
\cite{Young}). 
%The same definition on Berkovich 
%curves is due to F.Baldassarri.

Spectral radii are related to the spectral norm of the connection, their 
nature is hence quite algebraic. 
They are not continuous (cf. \eqref{S(FM)=S(FY)}).
The novelty of 
\cite{DV-Balda} and \cite{Balda-Inventiones} 
consists in allowing over-solvable radii, and hence working with a more 
geometric notion. 

The continuity results of \cite{DV-Balda} and \cite{Balda-Inventiones} 
is proved using 
the same ingredients of the original proof of \cite{Ch-Dw}: 
it is obtained as a consequence of a certain Dwork-Robba theorem 
\cite{Dwork-Robba-Growth}, 
that gives a bound on the growth of the coefficients of the Taylor 
solution matrix. Unfortunately the Dwork-Robba's bound is not helpful  
in the understanding of the $i$-th radii for $i\geq 2$, because 
it doesn't applies to an individual solution, 
but only to the entire solution matrix.

\begin{Def-intro}[cf. Section \ref{Skeleton of a function on M(X).}]
\label{Def: intro skeleton}
Let $\mathcal{T}$ be a set, and let 
$F:X\to \mathcal{T}$ be a function. 
We define the \emph{controlling graph} (also called \emph{constancy 
skeleton}) of $F$ as the set 
$\Gamma(F)$ formed by the 
points $x\in X$ without neighborhoods in $X$, 
isomorphic to a open disk on which $F$ is constant. 

We say that $F$ is a finite function if $\Gamma(F)$ is a
 finite graph  (i.e. it is a finite union of intervals).
\end{Def-intro}

The set $\Gamma(F)$ is always a graph containing 
the skeleton $\Gamma_X$ of $X$. Moreover $X-\Gamma(F)$ is a 
disjoint union of  open disks. 
In particular there exists a canonical retraction 
$X\to\Gamma(F)$, which is continuous as soon as $\Gamma(F)$ is a
finite graph. As an example one easily proves that 
for all $i=1,\ldots,\mathrm{rank}(\Fs)$ one has 
$\Gamma(\R_i^{\mathrm{sp}}(-,\Fs))=X$.

The following theorem is our main result:
\begin{thm-intro}[cf. Thm. 
\ref{Theorem : MAIN THM GEN}]
\label{intro-Thm.7}
For all $i=1,\ldots,r$ the graph
$\Gamma(\R_i(-,\Fs))$ is finite, and the function  
$\R_i(-,\Fs):X\to ]0,1]$ factorizes through 
the retraction $X\to\Gamma(\R_i(-,\Fs))$.
As a consequence $\R_i(-,\Fs)$ is a continuous function.
\end{thm-intro}
The statement of Theorem  
\ref{Theorem : MAIN THM GEN} is more complex and complete. 
It assembles the main properties verified by the radii. It is structured in 
analogy with \cite[Thm. 11.3.2]{Kedlaya-Book} where the same 
properties are stated for spectral radii. 
Roughly speaking we establish the following properties:
\begin{enumerate}
\item Finiteness of each $\R_i(-,\Fs)$, and of each partial height 
$H_i(-,\Fs):=\prod_{j=1}^i\R_j(-,\Fs)$;
\item Integrality of the slopes of each $\R_i(-,\Fs)$ along the segments 
of $X$;
\item Concavity locus of each $\R_i(-,\Fs)$;
\item Super-harmonicity of the partial heights $H_i(-,\Fs)$ 
outside a (locally) finite subset $\mathscr{C}_i$;
\item Description of $\mathscr{C}_i$.
\end{enumerate}

 We now give some ideas about the proof. 
Our approach is different in nature from that of \cite{DV-Balda} and 
\cite{Balda-Inventiones}. 
It basically consists in applying Frobenius push-forward 
to make the spectral radii small, and then read them on the coefficient 
 of the operator in a cyclic basis by the theorem of Young \cite{Young}. 
This is a well known method (at least) since 
\cite{Ch-Dw}. The problem consists, in fact, in making this process 
global in the sense of Berkovich, and in particular in managing 
solvable or over-solvable radii for which the 
reduction of the radii 
by Frobenius push-forward fails.

The proof is based on a criterion 
(cf Section \ref{F_t defi}) providing the finiteness of a real valued 
function $F:X\to\mathbb{R}_{>0}$ 
satisfying six technical properties. One of them 
is the super-harmonicity outside a finite set $\mathscr{C}$, which is 
the crucial assumption of the criterion.
%The basic idea of the criterion is that the super-harminicity imposes 
%that each bifurcation point of $\Gamma(F)$ is a point where $F$ has 
%a change of slopes. 

Now the proof of Theorem \ref{intro-Thm.7} consists in verifying 
the six assumptions of the criterion for $F=H_i(-,\Fs)$. 
This is done by induction on $i$. 
Now, we are able to prove that the potential failure of the super-harmonicity of $H_i(-,\Fs)$
can only happens at the end points of some  
$\Gamma(\R_k(-,\Fs))$, with $k\leq i-1$. 
This ensures, by induction, 
that the locus $\mathscr{C}_i$ of non-super-harmonicity is a finite set.
 
The potential failure of the super-harmonicity for $i\geq 2$ is actually 
the major theoretical difference with respect to the case $i=1$ (indeed 
the first radius $\R^{\Fs}_1$ is super-harmonic outside the Shilow 
boundary). This is one of the deeper difficulties of the paper. 

The main points permitting to deal with this are the super-harmonicity 
in the spectral non solvable case 
(cf. Proposition \ref{Prop : SH-well done spectral NS}, generalizing 
\cite[11.3.2,(c)]{Kedlaya-Book}),
a description of the nature of the graphs around solvable points 
(cf. Lemma \ref{Lemma : KEY . .!}),
and a concavity property of the radii generalizing 
the Transfer principle for the first radius 
(cf. Proposition \ref{Prop. : (C3)}).

\begin{rk-intro}
Recently similar results have been announced by F.Baldassarri and 
K.S.Kedlaya \cite{Kedlaya-draft} (cf. Final notes \ref{Notes final}). 

Independently Jérôme Poineau and Amaury Thuillier 
pointed out that, if a rig-smooth $K$-analytic curve $X$ has no 
boundary, then the continuity of $\R_{1}(-,\Fs)$ on $X$ is a 
consequence of the super-harmonicity. This now a theorem 
\cite{Potentiel}.
\end{rk-intro}

\if{
%\subsection*{Structure of the paper.}
\noindent\textbf{\textsc{Structure of the paper.}}
In sections \ref{Notation} we provide notations, and is section 
\ref{Skeleton of a function on M(X).} 
we introduce the skeleton $\Gamma(F)$ together with the 
finiteness criterion 
(cf. Thm. \ref{th : finiteness theorem}). 
In Section \ref{Radii of convergence and statement of main result} 
we define the radii of convergence and state our main 
result (cf. Thm. \ref{Theorem : MAIN THM GEN}). 
In Section \ref{Spectral polygons and related results} 
we introduce the spectral polygons of a differential operator and of a 
module, with related results.
In section \ref{Pull-back and push-forward by Frobenius} 
we adapt to our context some results concerning Frobenius. 
Finally in section \ref{Proof of the MTH} we give the proof of 
Thm. \ref{Theorem : MAIN THM GEN}. 
}\fi

\subsubsection*{Acknowledgements.} 
Some of the ideas of this paper find their genesis in 
some discussions we had with Lucia Di Vizio, we thank her deeply.
We are particularly grateful to Jérôme Poineau for multiple hints about 
the redaction. 
We also thank Yves André, Francesco Baldassarri, 
Kiran S. Kedlaya, Adriano Marmora, Amaury Thuillier 
for helpful discussions.

\section{Notations}\label{Notation}
%\vspace{0.5cm}
All rings are commutative with unit element. $\mathbb{R}$ is the field 
of real numbers, and 
$\mathbb{R}_{\geq 0}:=\{r\in\mathbb{R}\;|\; r\geq 0\}$.
For all field $L$ we denote its algebraic closure by 
$L^{\mathrm{alg}}$,
by $L[T]$ the ring of polynomial with coefficients in $L$, 
and  by $L(T)$ the fraction field of $L[T]$. If $L$ is valued, $\widehat{L}$ will be its completion. 

In all the paper  $(K,|.|)$ will be a complete field of characteristic $0$ 
with respect to an ultrametric absolute value 
$|.|:K\to\mathbb{R}_{\geq 0}$ 
i.e. verifying $|1|=1$, $|a\cdot b|=|a||b|$, and 
$|a+b|\leq\max(|a|,|b|)$ for all $a,b\in K$, and $|a|=0$ if and only if 
$a=0$. 
We denote by $|K|:=\{r\in\mathbb{R}_{\geq 0} 
\textrm{ such that }r=|t|,\;\exists t\in K\}$.

The semi-norm of a matrix will always mean the maximum of the 
semi-norms of its entries.
%\subsubsection{}

Let $E(K)$ be the category of isometric ring morphisms 
$(K,|.|)\to(\Omega,|.|)$. 
A morphism $\Omega\to\Omega'$ in $E(K)$ is an isometric ring 
morphism inducing the identity on $K$.
For all 
$\Omega,\Omega'\in E(K)$ there exists $\Omega''\in E(K)$ together 
with two morphisms $\Omega\subseteq\Omega''$ and 
$\Omega'\subseteq\Omega''$ of $E(K)$. 

We refer to \cite{Ber} for the definition of Berkovich spaces. For any 
point $x$ we denote by $\H(x)$ the residual field of $x$. 
By convention an open disk $D$ always has finite radius. 
Similarly an open annulus 
$C=\{x\in\mathbb{A}^{1,\mathrm{an}}_K\;|\;
R_1<|T-c|(x)<R_2\}$ always satisfies $0<R_1\leq R_2<+\infty$. 
We recall that if $\Omega\in E(X)$, and if
$D\subset \mathbb{A}^{1,\mathrm{an}}_\Omega$ 
is an open disk of radius $R$ centered at $t\in\Omega$ we have
\begin{equation}
\O(D)\;:=\;\{\sum_{n\geq 0}a_n(T-t)^n\;|\;a_n\in\Omega,\;\textrm{ for all }\rho<R,\;\lim_{n}|a_n|\rho^n=0\}\;.
\end{equation}

A virtual disk (resp. annulus) is a non-empty 
connected analytic domain of 
$\mathbb{A}^{1,\mathrm{an}}_K$ which becomes isomorphic to a 
union of disks (resp. annuli whose orientation is preserved by 
$\mathrm{Gal}(\Ka/K)$) over $\Ka$ (cf.\cite[3.6.32 and 3.6.35]{Duc})

If $K$ is algebraically closed, an affinoid domain $X$ 
of the Berkovich affine line $\mathbb{A}^{1,\mathrm{an}}_K$
(cf. \cite[2.2]{Ber}) is a disjoint union of connected affinoid domains of the form
\begin{equation}\label{eq : affinoid}
X\;=\;D^+(c_0,R_0)-\cup_{i=1}^nD^-(c_i,R_i)\;,
\end{equation}
where $D^+(c_0,R_0)$ (resp. $D^-(c_i,R_i)$) denotes the 
closed (resp. open) disk centered at $c_0$ (resp. $c_i$) with radius 
$R_0$ (resp. $R_i$), $c_0,\ldots,c_n\in K$ satisfy $|c_i-c_0|\leq R_0$ 
for all $i$, and $0<R_1,\ldots,R_n\leq R_0$. In order to avoid overlaps 
we implicitly assume that for $0<i<j\leq n$ we have 
$D^-(c_i,R_i)\cap D^-(c_j,R_j)=\emptyset$.

If $K$ is general, an affinoid domain $X$ of 
$\mathbb{A}^{1,\mathrm{an}}_K$ is the quotient of 
$X_{\Ka}$ by $\mathrm{Gal}(\Ka/K)$ (cf. \cite[2.2.2]{Ber}). Without 
loss of generality we will always assume that $X$ is connected. 
So the holes of $X_{\Ka}$ form an orbit 
under $\mathrm{Gal}(\Ka/K)$,\footnote{Including the holes of 
$X_{\Ka}$ at $+\infty$, i.e. the complements in 
$\mathbb{P}^{1,\mathrm{an}}_{\Ka}$ of the disk 
$D^+(c_0,R_0)$.} 
which acts isometrically.
Hence we can speak about the holes of $X$, and of their radii 
$R_1,\ldots,R_n$. And also about the larger virtual disk containing $X$ 
whose radius is $R_0$. 
Such notations are fixed from now on.
 
We denote by $\O(X)$ the $K$-algebra of the
global sections of $X$, and by $\partial X$ its Shilov boundary.

\subsubsection{}\label{subsubs : t_x}
For all $c\in K$, and all  $\rho\geq 0$, we denote by 
$x_{c,\rho}\in\mathbb{A}^{1,\mathrm{an}}_K$  the semi-norm 
defined by
\begin{equation}
x_{c,\rho}(f)\;:=\;\sup_{n\geq 0}|f^{(n)}(c)/n!|_K\cdot \rho^n\;,\qquad 
f\in K[T]\;,
\end{equation}
where $f^{(n)}$ is the $n$-th derivative of $f$ with respect to a 
coordinate $T$ of $X$. This definition actually depends on $T$. 

For all $\Omega\in E(K)$ we have a map
\begin{equation}
i_\Omega\;:\;X(\Omega)\to X
\end{equation}
associating to $t\in X(\Omega)$ the image $\pi_{\Omega/K}(t)$ 
of $t\in X_\Omega$ by the projection 
$\pi_{\Omega/K}:X_\Omega\to X$, where 
$X_\Omega:=X\widehat{\otimes}_K\Omega$.
If $x=i_\Omega(t)$ we say that $t\in X(\Omega)$ is a 
\emph{Dwork generic point} for $x$. Each point $x\in X$ admits a 
canonical Dwork generic point $t_x\in X_{\H(x)}$. 
Indeed, by the canonical property of the cartesian diagram 
$X_{\H(x)}/\H(x)\to X/K$, the point $x:\mathscr{M}(\H(x))\to X$ lifts 
into uniquely into a rational point 
$t_x:\mathscr{M}(\H(x))\to X_{\H(x)}$. 
However for a given field $\Omega\in E(K)$ we can have several embeddings 
$\H(x)\to\Omega$, hence there are no canonical lifting of $x$ in $X_\Omega$.

\subsubsection{}\label{Section 1.0.4}
More generally for all $x\in \mathbb{A}^{1,\mathrm{an}}_K$ 
we denote by $\lambda_x(0):=x$ and, for all $\rho>0$, we set 
\begin{equation}
\lambda_x(\rho)(f)\;:=\;\sup_{n\geq 0}\;x(f^{(n)}/n!)\cdot \rho^n\;,
\qquad f\in K[T]\;.
\end{equation}
One sees that $\lambda_x(\rho)\in X$ if an only if $x$ lies in 
the maximal virtual disk containing $X$ and $\rho\in I_x$, 
where $I_x$ is either equal to $[0,R_0]$ if $x\in X$, or 
$I_x=[R_i,R_0]$ if $x$ lies in a hole of $X$ with radius $R_i$. 

It follows from the definition that 
if $t\in X(\Omega)$ is a Dwork generic point for 
$x$, then $\lambda_x(\rho)=\pi_{\Omega/K}(x_{t,\rho})$. 
In particular the path $\rho\mapsto\lambda_{x}(\rho):I_x\to X$ is 
continuous. 

We call \emph{generic radius} of $x$ the number
\begin{equation}\label{eq : r_K}
r_K(x)\;:=\;\max\{\rho\in[0,R_0]
\textrm{ such that }
\lambda_x(\rho)=\lambda_x(0)\}\;.
\end{equation}
We write $r(x):=r_K(x)$ if no confusion is possible.
\begin{lemma}\label{rho gen as inf Kbar}
Let $x\in \mathbb{A}^{1,\mathrm{an}}_K$, 
and let $t\in\X(\Omega)$ be a Dwork generic point for 
$x$. Assume that $K^{\mathrm{alg}}\subset\Omega$. 
Then $r_K(x)$ equals the distance of $t$ from 
$K^{\mathrm{alg}}$ i.e.  
%$r_K(x) = \inf_{c \in K^{\mathrm{alg}}} |t-c|_\Omega$.
\begin{equation}
r_K(x) \;=\; 
\inf_{c \in K^{\mathrm{alg}}} |t-c|_\Omega\;.
\end{equation}
\end{lemma}
\begin{proof}
Let $d_t:=\inf_{c \in K^{\mathrm{alg}}} |t-c|_\Omega$. 
The norm of a polynomial $f\in K[T]$ 
is constant on each disk without zeros of $f$,
then $|f(y)|=|f(t)|$ for all $y\in D^-(t,d_t)\subset \mathbb{A}^{1,\mathrm{an}}_\Omega$.
Hence $\lambda_\xi(d_t)=\lambda_\xi(0)$ and $d_t\leq r(\xi)$.
To show $r(\xi)\leq d_t$ observe that the norm of a polynomial $f$ 
is not constant on a disk containing a zero of $f$. So
$D^-(t,r(\xi))$ has no $K^{\mathrm{alg}}$-rational points.
\end{proof}
\begin{corollary}
The canonical path $\lambda_\xi$ is constant on $[0,r(\xi)]$, and it 
induces an homeomorphism 
of $[r(\xi),R_0]$ with its image in $X$. \hfill$\Box$
\end{corollary}

\begin{corollary}\label{Lemma: generic disk r(x)h}
Let  $t\in\Omega\in E(K)$ be a Dwork generic point for $x$, 
then for all $\Omega'\in E(\Omega)$ 
each $\Omega'$-rational point of $D^-(t,r(\xi))$
is a Dwork generic point for $\xi$. \hfill$\Box$
\end{corollary}

%\begin{definition}
%We say that $x$ is a point of type $1$ if $r(x)=0$ 
%(i.e. if $x\in X(\widehat{K^{\mathrm{alg}}})$). 
%\end{definition}

\subsubsection{} The following proposition describes the structure of the fiber $\pi_{\Omega/K}^{-1}(x)$ of a point $x\in X$
\begin{proposition}\label{Prop : FIBER}
\label{t-t'=rho}
Assume that $K$ is algebraically closed.
Let $\Omega\in E(K)$, and let 
$\pi_{\Omega/K}\;:\;X_\Omega\;\to\; X$ 
be the canonical projection. Let $x\in X$.
There exists a point $\sigma_{\Omega/K}(x)\in 
\pi_{\Omega/K}^{-1}(x)$ such that 
\begin{equation}
\pi_{\Omega/K}^{-1}(x)-\{\sigma_{\Omega/K}(x)\}
\end{equation} 
is a (possibly empty) disjoint union of open disks, 
all having $\sigma_{\Omega/K}(x)$ as relative boundary in 
$X_\Omega$. 

Moreover if $\Omega/K$ is algebraically closed and 
spherically complete, the group 
$\mathrm{Gal}^{\mathrm{cont}}(\Omega/K)$ of $K$-linear continuous automorphisms of $\Omega$ 
fixes $\sigma_{\Omega/K}(x)$ and it acts transitively on 
those disks, and also on the set $i_\Omega^{-1}(x)$ of 
their $\Omega$-rational points.
\end{proposition}
\begin{proof}%[Proof of Proposition \ref{Prop : FIBER}]
We can assume $\Omega$ algebraically closed.
Let $t\in i_\Omega^{-1}(x)$. 
By Corollary \ref{Lemma: generic disk r(x)h} one has 
$D^-(t,r(x))\subseteq\pi_{\Omega/K}^{-1}(x)$. 
It is then enough to show that all $t'\in i_\Omega^{-1}(x)$ verifies $|t'-t|\leq r(x)$. 

This follows from the fact that 
$\pi_{\Omega/K}^{-1}(x)$ is the spectrum 
$\mathscr{M}(\H(x)\widehat{\otimes}_K\Omega)$, so it is contained 
in all affinoid domains containing $x$. 
Hence we can replace $X$ by any $K$-rational closed disk in 
$\mathbb{A}^{1,\mathrm{an}}_K$ containing $x$. 
Now by Lemma \ref{rho gen as inf Kbar} we can find a sequence of 
closed $K$-rational disks with intersection $\D^+(t,r(x))$. 
This proves the claim.

The assertion about the Galois action 
follows from Lemma \ref{Gal transitive} below.
\end{proof}
\begin{lemma}\label{Gal transitive}
Let $\xi\in X$. 
If $\Omega\in E(K)$ is algebraically closed and maximally complete, 
then $i_\Omega^{-1}(\xi)$ is either the empty set, or 
$\mathrm{Gal}^{\mathrm{cont}}(\Omega/K)$ acts transitively on it. 

Namely for all $t,t'\in i_\Omega^{-1}(x)$ there is 
$\sigma\in\mathrm{Gal}^{\mathrm{cont}}(\Omega/K)$ such that 
$\sigma(t)=t'$.
\end{lemma}
\begin{proof}
Identify 
$\mathbb{A}^{1,\mathrm{an}}_K(\Omega)$ with $\Omega$, 
and $X(\Omega)$ with a subset of $\Omega$. 
With this identification we have to find an automorphism of $\Omega$ 
sending $t$ into $t'$.
Let $\widehat{K(t)}$ and $\widehat{K(t')}$ 
be the completions of the sub-fields of $\Omega$ generated by $t$ 
and $t'$. We consider the $K$-isomorphism $K(t)\simto K(t')$ sending 
$t$ into $t'$. 
Since $x=\pi_{\Omega/K}(x_{t,0})=\pi_{\Omega/K}(x_{t',0})$, these 
semi-norms coincide on $K[T]\subset\O(X_\Omega)$. 
Hence this $K$-isomorphism is isometric, and  
$\widehat{K(t)}\cong\H(\xi)\cong\widehat{K(t')}$. More precisely  
there exists a continuous isometric $K$-linear isomorphism 
$\sigma:
\widehat{K(t)}\simto\widehat{K(t')}$ such that  $\sigma(t)=t'$. Now 
$\sigma$ extends to an isometric automorphism of $\Omega/K$ 
(cf. \cite[Lemma 8.3]{Dw-Robba}, \cite{Poonen}, \cite{Reversat}, see \cite{NP-II} for more details).
\end{proof}

\begin{definition}[Generic and maximal disks]
Let $x\in X$, let $\Omega\in E(\H(x))$, and let $t\in X(\Omega)$ 
be a Dwork generic point for $x\in X$. We call 
\emph{generic disk} of $x$ the virtual open disk
\begin{equation}
D(x)\;\subset\; X_{\Omega}
\end{equation} 
which is the connected component of 
$\pi_{\Omega/K}^{-1}(x)-\{\sigma_{\Omega/K}(x)\}$ 
containing the point $t$. Its radius is $r(x)$.

We call \emph{maximal disk} of $x$
\begin{equation}
D(x,X)\;\subset\;X_{\Omega}
\end{equation}
the maximum virtual open disk in $X_{\Omega}$ containing $t$. 
It is also the connected component of $X-\Gamma_{X_{\Omega}}$ 
containing $t$. With the notation of \eqref{rho_s(t)}, its radius is $\rho_{\Gamma_X}(x)$, and it will be denoted by 
\begin{equation}\label{eq : rho_x,X=rho_Gamma_X}
\rho_{x,X}\;=\;\rho_{\Gamma_X}(x)\;.
\end{equation}
\end{definition}
By Lemma \ref{Gal transitive}, up to extend $\Omega$, 
all generic and maximal disks are isomorphic. 
The notation does not depend on $t$, and the definitions and results 
of this paper will be independent on its choice. 

\begin{lemma}\label{rhogen titi}
One has $r(\xi_{t,\rho})\;=\;\max(\rho,r(\xi_{t}))$. 
In particular if $t\in X(\widehat{K^\mathrm{alg}})$, then 
$r(\xi_{t,\rho})=\rho$.\hfill $\Box$
\end{lemma}
\subsection{Graphs}
As a topological space $X$ is a tree, in particular it is uniquely archwise 
connected.\footnote{This means that for all $x,y\in X$ there exists an 
injective continuous path $[0,1]\to X$ with initial point $x$ and end 
point $y$. Moreover two such paths have the same image in $X$.} 
If $x,y\in X$ we denote by $[x,y]\subset X$ the image of an injective 
continuous path $[0,1]\to X$ with initial point $x$ and end point $y$.
In particular the image of $\lambda_x:I_x\to X$ is the closed segment 
$\Lambda(x):=[x,x_{c_0,R_0}]$.\footnote{By an abuse 
here and below we identify 
$x_{c_0,R_0}\in X_{\Ka}$ with its image in $X$.}
We define in an evident way open and semi-open segments, 
denoted by $]x,y[$, $[x,y[$, $]x,y]$ respectively.

Following \cite{Duc} we say that a graph $\Gamma$ in $X$ 
is \emph{admissible} if $X-\Gamma$ is a disjoint union of virtual open 
disks, in particular 
$\Gamma$ is closed in $X$, and also connected (since we assume $X$ connected).

An example of admissible graph is the analytic skeleton 
$\Gamma_X\subseteq X$ defined as the locus of points without open 
neighborhoods in $X$ isomorphic to a virtual open disks. 
More explicitly $\Gamma_X$ is the union of the segments 
$\Lambda(x)$ for all point $x$ at the boundary of a hole of $X$ 
(i.e. for all $x$ of the Shilov boundary $\partial X$).
$\Gamma_X$ is also the set of semi-norms on $\O(X)$ that are 
maximal with respect to the partial order given by $x\leq x'$ if and only 
if $x(f)\leq x'(f)$ for all $f\in\O(X)$.

For any subset $A\subseteq X$ we set $\mathrm{Sat}(A):=
\cup_{x\in A}\Lambda(x)$. This is a tree in $X$. As an example 
$\Gamma_X=\mathrm{Sat}(\partial X)$. 
We say that a subset of $X$ is 
\emph{saturated} if it coincides with $\mathrm{Sat}(A)$, for some set 
$A$. 
\begin{lemma}\label{Lemma : Admissible}
A graph $\Gamma\subset X$ is admissible if and only if the following 
conditions hold: 
\begin{enumerate}
\item[(i)] $\Gamma_X\subseteq\Gamma$; \qquad\qquad
\emph{(ii)} $\Gamma=\mathrm{Sat}(\Gamma)$;\qquad\qquad
\emph{(iii)} $\Gamma$ contains its end points.\hfill$\Box$
\end{enumerate}
%\begin{enumerate}
%\item $\Gamma_X\subseteq\Gamma$;
%\item If $x\in\Gamma$, then $\Lambda(x)\subseteq\Gamma$;
%\item $\Gamma$ contains its end points.\hfill$\Box$
%\end{enumerate}
\end{lemma}
\begin{definition}\label{Def.  rho_S xi}
Let $\Gamma$ be a non empty saturated subset and let $x\in X$. 
We denote by 
\begin{equation}\label{rho_s(t)}
\rho_{\Gamma}(x)\;:=\;
\inf\{\rho\geq r(x)\textrm{ such that }
\lambda_{x}(\rho)\in\Gamma\}\;,\qquad\quad
\delta_\S(x)\;:=\;\lambda_{x}(\rho_\Gamma(x))\;.\qquad
\end{equation}
%\begin{eqnarray}
%\rho_{\Gamma}(x)&\;:=\;&
%\inf\{\rho\geq r(x)\textrm{ such that }\lambda_{x}
%(\rho)\in\Gamma\}\;,\\
%\delta_\S(x)&\;:=\;&\lambda_{x}(\rho_\Gamma(x))\;.
%\end{eqnarray}
\end{definition} 
The map $\delta_{\Gamma}:X\to X$ %
is a retraction onto the graph $\overline{\Gamma}$ obtained from 
$\Gamma$ by adding to it its end points. In other words 
$\delta_\Gamma$ induces the identity on $\overline{\Gamma}$ and  
$\delta_\Gamma(X)=\overline{\Gamma}$. 
We call $\delta_\Gamma:X\to\overline{\Gamma}$ the 
\emph{canonical retraction}. 

If $\Gamma$ is admissible, then each 
point $x\in X-\Gamma$ lies in a virtual open disk $D_x$ with boundary 
in $\Gamma$, and $\delta_\Gamma$ associates to $x$ that boundary.
If $\Gamma$ is finite admissible, endowed with the topology induced 
by $X$, then $\delta_\Gamma:X\to \Gamma$ is continuous, and 
the topology of $\Gamma$ is also the quotient topology by 
$\delta_\Gamma$.

%\begin{remark}\label{Def : rho_xX}
%We have $\rho_{\xi,X}\;:=\;\rho_{\S_X}(\xi)$.
%\end{remark}
\begin{remark}\label{rk : xleq x' then rhox=rhox'}
If $t_x\in X(\Omega)$ is a Dwork generic point for $\xi$. The radius $\rho_{x,X}$ of $D(x,X)$ verifies  
$\rho_{\xi,X}=\rho_{t,X_\Omega}:=
\min_{i=1,\ldots,n} (|t-c_i|_\Omega,R_0)$. 
We notice that if $\xi\leq\xi'$, then $\rho_{\xi,X}=\rho_{\xi,X'}$.
In fact the inequality $\xi\leq\xi'$ applied to $T-c_i$ and $(T-c_i)^{-1}$ 
provides $|t_x-c_i|=|t_{x'}-c_i|$. 
\end{remark}
\begin{remark}
For all $t\in X(\Omega)$, and all $\sigma\geq 0$ one has 
$\rho_{\xi_{t,\sigma},X}=\max(\sigma,\rho_{t,X})$.
\end{remark}
\subsection{Directions, slopes, directional finiteness, and harmonicity}
\label{Directions, slopes, directional finiteness, and harmonicity}
We define an equivalence relation between the open segments $]x,y[$ 
with boundary $x\in X$. We say that $]x,y[\sim]x,z[$ if there exists 
$]x,t[\subseteq]x,z[\cap]x,y[$. An equivalence class $b$ is called a 
\emph{germ of segment out of $x$}, or \emph{direction}, or again a 
\emph{branch}.

We denote by $\Delta_X(x)$, or simply by $\Delta(x)$ if no confusion 
is possible, the set of all directions out of $x$, and if 
$\Gamma$ is a graph containing $x$, we denote by 
$\Delta(x,\Gamma)$ the set of germs of segments out of $x$ that are contained in 
$\Gamma$. If $\Delta(x,\Gamma)$ is a finite set we say that 
$\Gamma$ is \emph{directionally finite} at $x$.

Let $x\in X$ and let $b=]x,y[$ be a germ of segment out of $x$. 
We will always provide $b$ with the orientation as outside $x$.
Clearly, if $y$ is close to $x$, then 
either $]x,y[\subset\Lambda(x)=[x,x_{c_0,R_0}]$ or 
$]x,y[\subset\Lambda(y)=[y,x_{c_0,R_0}]$. 
Assume that $b\subset\Lambda(x)$, and let 
$I\subset \mathbb{R}_{>0}$ be the inverse image of $]x,y[$ in 
$I_x=[0,R_0]$ (cf. Section \ref{Section 1.0.4}). 
We recall that, for all $x\in X$, the path 
\begin{equation}\label{eq : parametrization}
\lambda_x:[0,R_0]\to[x,x_{c_0,R_0}]\subset X\;,
\end{equation}
is continuous, it is constant on $[0,r(x)]$ with value $x$, and it 
identifies $[r(x),R_0]$ with $[x,x_{c_0,R_0}]$. So $I=]r(x),\rho[$ for 
some $\rho>r(x)$.

With these conventions let $F:X\to\mathbb{R}_{>0}$
%\begin{equation}
%F\;:\;X\to\mathbb{R}_{>0}
%\end{equation} 
be a function such 
that $\log\circ F\circ\lambda_x\circ\exp : \ln(I)\to 
\mathbb{R}$ is an affine function. We say that $F$ is $\log$-affine along 
$b=]x,y[$ and we denote by 
\begin{equation}
L_xF\;:=\;\ln\circ F\circ\lambda_x\circ\exp \;:\; 
]-\infty,\ln(R_0)]\;\xrightarrow{\quad}\;\mathbb{R}\;.
\end{equation}
We say that $L_xF$ is the $\log$-function attached to $F$. 
We denote its slope by $\partial_bF(x)$, this is the right 
derivative of $\log\circ F\circ\lambda_x\circ\exp$ at $\log(r(x))$. 
If now $b=]x,y[\subset\Lambda(y)$ we call $I$ the inverse image of $]x,y[$ 
in $I_y$, and we denote by $\partial_bF(x)$ the 
negative of the slope of 
$\log\circ F\circ\lambda_y\circ\exp : \ln(I)\to 
\mathbb{R}$.

\begin{definition}\label{Def. : LOG-convex}
If $]z,u[\subset \Lambda(x)=[x,x_{c_0,R_0}]$, we say that $F$ is 
$\log$-affine (resp. $\log$-concave, $\log$-decreasing, ...) along $]z,u[$, 
if $L_xF$ is affine (resp. concave, decreasing, ...)
over $(\lambda_x\circ\exp)^{-1}(]z,u[)$.
\end{definition}
\begin{notation}\label{Notation : slopes exist}
Assume that $F$ is $\log$-affine along all direction $b\in\Delta(x)$ out of $x\in X$, 
and that $\partial_bF(x)=0$ for almost, but a finite number of them.
\end{notation}
\begin{definition}\label{Def : Laplacian}
\label{Def : sub-harmonic}
We call \emph{Laplacian of $F$} at $x$ the finite sum
\begin{equation}
dd^cF(x)\;=\;\sum_{b\in\Delta(x)}m_b\cdot\partial_bF(x)\;,
\end{equation}
where $m_b\in\mathbb{N}$ is the multiplicity of $b$ 
(i.e. the number of germs of segments in 
$X_{\widehat{K^{\mathrm{alg}}}}$ lying over $b$).

If now $x\notin\partial X$, we say that $F$ is \emph{super-harmonic} 
(resp. \emph{sub-harmonic}; \emph{harmonic}) at $\xi$ if
\begin{equation}\label{eq : QH inequality}
dd^cF(\xi) \;\leq\; 0\;,\textrm{( resp. }dd^cF(\xi)\geq 0\;;\;
\;dd^cF(\xi)=0\textrm{)}\;.
\end{equation} 
We say that $F$ is (globally) \emph{super-harmonic} (resp. 
\emph{sub-harmonic}; \emph{harmonic}) on $X$,
if it is so at all point $\xi\notin \partial X$. 
\end{definition}
\begin{lemma}\label{Lemma: G dominates F, so F super-H}
Let $x\in X-\partial X$, and 
let $F,G:X\to\mathbb{R}$ be two functions on $X$ as in Notation 
\ref{Notation : slopes exist}.
Assume that $F_{|b}\leq G_{|b}$ along each germ of segment $b$ out of $x$, that
$F(x)=G(x)$, and that $G$ 
is super-harmonic at $x$. Then $F$ is super-harmonic at $x$.
\hfill$\Box$
\end{lemma}

\begin{remark}
The Laplacian of $F$ at the points of $\partial X$ does not give 
information since some directions out of  $x$ 
are ``\emph{removed}''.
As an example functions $f\in\O(X)$ are harmonic, but their 
Laplacian at the points $x\in\partial X$ of the boundary is not always 
negative.

Definition \ref{Def : sub-harmonic} is less general with respect to the 
usual definition of super-harmonicity, 
as for example those in \cite{Baker-Book}, \cite{Amaury-These}, 
\cite{Favre-Jonsson}. 
The general definition allows an infinite number of direction 
of non zero slope and the finite sum  
\eqref{eq : QH inequality} is replaced by an infinite one.
\end{remark}

\section{Constancy skeleton of a function on $\X$.}
\label{Skeleton of a function on M(X).}
Let $\mathcal{T}$ be a set and let $\F:X\to\mathcal{T}$ 
be an arbitrary function. 
\begin{definition}
\label{def :constancy sk}
We define the \emph{controlling graph} (also called \emph{constancy skeleton}) 
\begin{equation}
\S(X,\F)\;\subseteq \;X
\end{equation}
of $\F$ as the set of points of $\X$ without neighborhoods in $X$ isomorphic to an open virtual 
disk on which $\F$ is constant.
We write $\S(\F)$ if no 
confusion is possible.
\end{definition}

\begin{definition}[Constancy radius]
For all $x\in X$ let $t_x\in X_{\H(x)}$ be the canonical point of Section \ref{subsubs : t_x}. 
We define \emph{the constancy radius 
$\rho_{F}(x):=\rho_{\Gamma(F)}(x)$
of $F$ at $x$} as the radius of the largest open disk in 
$\X_{\H(x)}$ centered at $t_x$ on which 
the composite map 
$F_{\H(x)}:\X_{\H(x)}\to\X\to\mathcal{T}$ is constant. 
\end{definition}
\subsection{Basic properties.} 
\label{Basic properties.}
Since $D(x)=D^-(t_x,r(x))\subset\pi_{\H(x)/K}^{-1}(x)$, 
from the definition one immediately has 
\begin{equation}\label{encadrement rho}
r(x)\;\leq\;\rho_{\F}(x)\;\leq\;\rho_{x,X}\;\leq\;R_0\;.
\end{equation}
\begin{lemma}\label{Lemma : x in Gamma(R) ssi r(x)=rho_R(x)}
The following conditions are equivalent:
\begin{enumerate}
\item[(i)] $x\in\Gamma(\F)$;
\qquad \emph{(ii)} $r(x)=\rho_{\F}(x)$; 
\qquad \emph{(iii)} there exists 
$y\in X$ such that $x=\lambda_{y}(\rho_{\F}(y))\in\X$.
\end{enumerate}
\end{lemma}
\begin{proof}
If  $x\in\Gamma(\F)$, then $\rho_{\F}(x)=r(x)$, 
because if $r(x)<\rho_{\F}(x)$, the image in $X$ of 
$D^-(t_x,\rho_\F(x))$ is virtual disk containing $x$ on which $\F$ is 
constant. Hence $(i)\Rightarrow (ii)$. 
Now $x=\lambda_x(r(x))$, so $(ii)\Rightarrow (iii)$.
Assume now $(iii)$. If $D\subseteq X$ is a 
virtual open disk containing $x$ on which $\F$ is constant, 
then $y\in D$ and 
$D_{\H(y)}\subseteq D^-(t_y,\rho_\F(y))$. 
Hence we obtain the contradiction $x\neq\lambda_y(\rho_\F(y))$.
So $(iii)\Rightarrow(i)$.
\end{proof}
\begin{lemma}
Let $F:X\to\mathcal{T}$ and $F':X\to\mathcal{T}'$ be two functions. 
We have $\Gamma(F)=\Gamma(F')$ if and only if 
$\rho_F(x)=\rho_{F'}(x)$ for all $x\in X$.\hfill$\Box$
\end{lemma}

\begin{proposition}\label{prof propp}
$\S(\F)$ is an admissible graph in $X$. Moreover it satisfies:
\begin{enumerate}
\item $\xi\in\Gamma(\F)$ if and only if $\rho_{\F}(\xi)=r(\xi)$;
\item $\lambda_\xi(\rho_{\F}(\xi))\in\S_X$ if and only if $\rho_{\F}(\xi)=\rho_{\xi,X}$;
\item $\rho_{\F}(\xi)=\rho_{\Gamma(\F)}(\xi)$ for all 
$\xi\in X$ (cf. \eqref{rho_s(t)});
\item For all $\xi\in X$, and all $\rho\in[0,R_0]$ one has 
$\rho_{\F}(\lambda_\xi(\rho))=\max(\rho,\rho_{\F}(\xi))$;
\item If $\F$ is constant on a virtual open disk $D\subset X$, then $D\cap\S(\F)$ is empty.
\end{enumerate}
\end{proposition}
\begin{proof}
We can assume $K=\widehat{K^{\mathrm{alg}}}$. 
By definition 
$\Gamma(\F)$ is the complement of a union of disks, so it is 
admissible by Lemma \ref{Lemma : Admissible}. Point i) follows from Lemma 
\ref{Lemma : x in Gamma(R) ssi r(x)=rho_R(x)}, 
and property v) is evident. 

Now ii), iii) and iv) are straightforward.
\end{proof}

For all $x\in X$ we set 
$\delta_\F(x):=\delta_{\Gamma(\F)}(x)=\lambda_x(\rho_\F(x))$. 
We say that $\F$ is \emph{finite} if $\Gamma(\F)$ is a finite graph. 
In this case $\delta_\F:X\to\Gamma(\F)$ is a continuous retraction 
(cf. after Def. \ref{Def.  rho_S xi}). 

\begin{remark}\label{deltadelta}
The correspondence $\F\mapsto \delta_\F$ is idempotent : 
%\begin{equation}
$\delta_{\delta_\F}=\delta_\F$.
%\end{equation}
More precisely if $\S\subseteq\X$ is a saturated subset, and if 
$\F=\delta_\S:\X\to\overline{\S}$ is its retraction, then 
\begin{equation}
\delta_{\delta_\Gamma}\;=\;\delta_{\Gamma\cup X}\;=\;\delta_{\overline{\Gamma}\cup\Gamma_X}\;.
\end{equation}
Every admissible graph $\Gamma$ 
is the skeleton of its retraction map $\delta_\S$ (i.e. $\S=\S(\delta_\S)$).
\end{remark}

\begin{remark}\label{Rk : min max fini}
Let $\F_i:\X\to\mathcal{T}_i$, $i=1,2$, and let 
$g:\mathcal{T}_1\times\mathcal{T}_2\to\mathcal{T}_3$ be any 
functions. Then 
\begin{equation}
\S\bigl(g\circ (\F_1\times\F_2)\bigr)\;\subseteq\;
\S(\F_1)\cup\S(\F_2)\;.
\end{equation}
Indeed clearly $\rho_{\F_3}(\xi)\geq
\min(\rho_{\F_1}(\xi),\rho_{\F_2}(\xi))$, 
and $\S(\F_1)\cup\S(\F_2)$ is saturated. 

If $\mathcal{T}_i=\mathbb{R}$, 
this holds in particular for $\max(\F_1,\F_2)$ or $\min(\F_1,\F_2)$.
\end{remark}

\begin{remark}\label{Rk : restr to X'}
Let $X'\subseteq X$ be a sub-affinoid, and let $\F':X'\to\mathcal{T}$ be the 
restriction of $\F:X\to\mathcal{T}$ to 
$X'$. To avoid confusion we denote by $\Gamma(X,\F)\subseteq X$, 
$\Gamma(X',\F')\subset X'$, $\rho_{\F}(X,-)$, $\rho_{\F'}(X',-)$ 
the respective skeletons and constancy radii. 
If $\xi'\in X'$ one has $\rho_{\F'}(X',x')=
\min(\rho_{\F}(X,x'),\rho_{x',X'})$, so
%\begin{equation}
%\rho_{\F'}(X',\xi')\;=\;
%\min(\rho_{\F}(X,\xi'),\rho_{\xi',X'})\;,
%\end{equation} 
%hence 
\begin{equation}
\S(X',\F')\;=\;\Bigl(\Gamma(X,\F)\cap X' \Bigr)\cup\Gamma_{X'}\;.
\end{equation}
Hence the directional finiteness of $\F$ at $x'\in X'$ is 
equivalent to that of $\F'$ at $x'$. 
Moreover the finiteness of $\F$ on $X$ implies that of $\F'$ on $X'$. 
\end{remark}

\begin{proposition}[Scalar extension]\label{Prop: iso O/K}
Let $\Omega\in E(K)$ and let as usual 
$\pi_{\Omega/K}: X_{\Omega}\to X$ 
be the canonical projection. Denote by 
$\F_{\Omega}:X_\Omega\to\mathcal{T}$ the composite map 
$\F\circ\pi_{\Omega/K}$. 
One has $\Gamma(\F)=\pi_{\Omega/K}(\Gamma(\F_\Omega))$.
Moreover if $K$ is algebraically closed, then $\pi_{\Omega/K}$ induces 
a bijection between $\Gamma(\F_\Omega)$ and $\Gamma(\F)$ with 
inverse $\sigma_{\Omega/K}$ (cf. Prop. \ref{Prop : FIBER}).
In particular $\F$ is finite if and only if $\F_{\Omega}$ is finite.
\end{proposition}
\begin{proof}
We have $X=X_{\widehat{K^{\mathrm{alg}}}}/\mathrm{G}$ where 
$\mathrm{Gal}(K^{\mathrm{alg}}/K)$. Hence 
$\Gamma(\F)=
\Gamma(\F_{\widehat{K^{\mathrm{alg}}}})/\mathrm{G}$.
As a consequence $\F$ is a finite function if and only if 
$\F_{\widehat{K^{\mathrm{alg}}}}$ is. So we can assume 
$K$ algebraically closed.
By Proposition \ref{Prop : FIBER}, 
there is a open disk containing $x\in X$ on which 
$\F$ is constant if and only if there is a open disk containing 
$\sigma_{\Omega/K}(x)\in X_{\Omega}$ on which 
$\F_\Omega$ is constant. The claim follows.
\end{proof}

\begin{remark}\label{Examples of skeletons.}
\begin{enumerate}
\item Let $\F=\mathrm{Id}_{\X}:\X\to\X$ be the identity, 
then $\S(\mathrm{Id}_{\X})=\S(r_K)=\X$ (cf. \eqref{eq : r_K}).
\item Let $\F=1:\X\to\{pt\}$ be a constant map, then 
$\S(1)=\S(\rho_{-,X})=\S_X$ is the skeleton of $\X$. 
\item Let $f_1,\ldots,f_n\in\O(X)$, 
let $\alpha_1,\ldots,\alpha_n> 0$, and let 
$\F(\xi):=\min_i(|f_i(\xi)|^{\alpha_i})$.
Then $\S(\F)=\mathrm{Sat}(\{z_1,\ldots,z_r\})\cup\S_X$, 
where 
$\{z_1,\ldots,z_r\}\subset X(K^{\mathrm{alg}})$ is the union of 
all zeros of $f_1,\ldots,f_n$.
\item With the above notations if 
$\F(\xi):=\max_i(|f_i(\xi)|^{-\alpha_i})$, intended as a function 
with values in the set 
$\mathcal{T}:=\mathbb{R}_{> 0}\cup\{\infty\}$, then one again 
has $\S(\F)=\mathrm{Sat}(\{x_{z_1},\ldots,x_{z_r}\})\cup\S_X$.
\item Assume now that $\F(\xi):=\max_i|f_i(\xi)|^{\alpha_i}$ (resp. 
$\F(\xi):=\min_i|f_i(\xi)|^{-\alpha_i}$ as a function with 
values in $\mathcal{T}:=\mathbb{R}_{> 0}\cup\{\infty\}$). 
In this case the explicit description of the skeleton $\S(\F)$ is more 
complicate. However, one can easily deduce its finiteness from Remark 
\ref{Rk : min max fini}.
\end{enumerate}
\end{remark}
\subsection{Branch continuity and dag-skeleton.}
\label{Branch continuity and dag-skeleton.}
We investigate now whether the function $\F$ admits a factorization as 
$\F=\F_{|_{\S(\F)}}\circ\delta_\F$:
\begin{equation}
\xymatrix@=10pt{
&\X\ar[rr]^-{\F}\ar[dr]^-{\begin{picture}(0,0)\put(0.5,0.5)
{\begin{scriptsize}$\delta_\F$\end{scriptsize}}
\end{picture}}&&\mathcal{T}\\
\S(\F)\ar[ur]\ar@{=}[rr]&&\S(\F)\ar@{..>}[ur]_-{\F_{|_{\S(\F)}}}&}
\end{equation}
This is not automatically verified. In fact for a given $\xi\in\X$ the 
restriction 
$\F\circ\lambda_{\xi}:[0,R_0]\to\mathcal{T}$ is constant for 
$\rho\in[0,\rho_\F(\xi)[$, but one may have a 
different value at $\rho=\rho_{\F}(\xi)$. 

We say that $\F$ is \emph{branch continuous} if for all $\xi\in\X$ one 
has 
\begin{equation}
\F(\lambda_{x}(\rho_\F(x) ) )\;=\;
\lim_{\rho\to\rho_\F(x)^-}\F(\lambda_{x}(\rho))\;=\;
\F(x)\;.
\end{equation}

A branch continuous map factorizes as 
$\F=\F_{|_{\S(\F)}}\circ\delta_\F$ and is determined by its values on 
$\S(\F)$. 

A continuous function with values in a Hausdorff space $\mathcal{T}$ 
is branch continuous. 

Conversely a finite and 
branch continuous function is continuous if, and only if, its restriction to 
$\S(\F)$ is continuous. 

For some purposes this situation may be unsatisfactory since we want 
to factorize all functions.
For this we define the \emph{dag-skeleton} $\S(\F)^{\dag}$ as 
the union of $\S(\F)$ 
together with an (unspecified) germ of segment out of all point $x$ of 
$\Gamma(\F)$, for all direction $b\in\Delta(x)$.  
Clearly, any function $\F$ factorizes through its 
dag-skeleton $\S(\F)^{\dag}$. 
This situation will not occur in this paper since all the functions will be 
branch continuous.
This idea can be better expressed in term of Huber spaces 
\cite{Huber-Book}, indeed germs of segments correspond to 
Huber points, but this lies outside the scopes of this paper.

\subsection{Minimal triangulation}
We denote by $S_X$ the union of 
the Shilov boundary $\partial X$ and of the bifurcation points of 
$\Gamma_X$. If $K$ is algebraically closed, $X$ is of the form 
\eqref{eq : affinoid}, then we explicitly have
\begin{equation}\label{eq : S_X}
S_X\;:=\;\{x_{c_i,R_i}\}_{i=0,\ldots,n}\cup
\{x_{c_i,|c_i-c_j|}\}_{i\neq j, i,j=1,\ldots,n}\;.
\end{equation}
If $K$ is general, $S_X$ is the image of $S_{X_{\Ka}}$ by the 
projection.
We notice that all points of $S_X$ are of type $2$ or $3$, and that 
$X-S_X$ is a disjoint union of virtual open disks or annuli that are 
relatively compact in $X$. This is called a \emph{triangulation} of $X$ 
in \cite{Duc}, and it is related to the existence of a formal model of 
$X$. 
More precisely $S_X$ is the minimum triangulation of $X$.

\subsection{A Criterion for the finiteness of a positive 
real valued function $\F$.}
\label{F_t defi} 
\label{Main result}
Let as usual $X$ be an affinoid domain of the affine line.
Let %$F:X\to\mathbb{R}_{>0}$ 
\begin{equation}
\F\;:\;\X\to\mathbb{R}_{>0}
\end{equation} 
be a positive function. 
We know that $L_x\F$ is constant at least on $]-\infty,\ln(r(x))]$ 
(cf. Section \ref{Directions, slopes, directional finiteness, 
and harmonicity}). 

Let $\Gamma\subseteq X$ be a finite admissible graph. We consider the following 
conditions:
\begin{enumerate}
\item[(C1)] For all $\xi\in\X(\widehat{K^{\mathrm{alg}}})$ 
one has $\rho_{\F}(\xi)>0$ (equivalently $\Gamma(\F)$ has no points 
of type $1$). 
\item[(C2)] For all $\xi\in\X$ the function 
$L_x\F:]-\infty,\ln(R_0)]\to\mathbb{R}$ is 
continuous on $]-\infty,\ln(R_0)]$, piecewise affine on it, 
and with a finite number of breaks all along $]-\infty,\ln(R_0)]$.
\item[(C3)] 
For all $x\in X$, the function $L_x\F$ is concave on 
$]-\infty,\ln(\rho_\Gamma(x))[$. This implies in particular that 
if $]x,y[\cap\Gamma=\emptyset$, 
then $\F$ is $\log$-concave on $]x,y[$ (cf. Def.  
\ref{Def. : LOG-convex}).

\item[(C4)] The \emph{non zero} slopes of $\F$ can not be 
arbitrarily small. Namely there exists a positive constant $\nu_\F>0$ 
such that for all $x\in X$, and all germ of segment $b$ out of $x$ 
one has
\begin{equation}
\partial_b\F(x) \in\; ]-\infty,-\nu_\F[\;\cup\;\{0\}\;\cup\;]\nu_\F,+\infty[\;.
\end{equation}

\item[(C5)] $\Gamma(\F)$ is directionally finite at all its bifurcation 
points (cf. Section 
\ref{Directions, slopes, directional finiteness, and harmonicity}).
 
\item[(C6)] There exists a finite set $\C(\F)\subseteq X$  
such that if $x$ is a bifurcation point of $\Gamma(\F)$ 
not in $\C(\F)\cup\partial X$, then $\F$ is super-harmonic at $x$ 
(cf. Def. \ref{Def : sub-harmonic}). 
\end{enumerate}
\begin{remark}
\begin{enumerate}
\item 
By \eqref{encadrement rho}, condition $\rho_{\F}(x)>0$ for all   $x\notin X(\Ka)$;
\item If $\F$ verifies (C1) and (C2) then it is 
branch-continuous (cf. Section 
\ref{Branch continuity and dag-skeleton.});
\item (C1) plus (C3) imply that $\F$ is 
logarithmically not increasing over each segment $]x,y[$ (oriented 
towards $+\infty$) such that $]x,y[\cap\Gamma=\emptyset$.

\item Conditions (C2) and (C5) ensure Notation 
\ref{Notation : slopes exist} 
so that $dd^c\F(x)$ is defined for all $x\in X$.
\end{enumerate}
\end{remark}

\begin{proposition}[Permanence of (C1)--(C6) by scalar extension]
\label{Invariance of (C1)--(C6) by scalar extension of the ground field}
With the notations of Proposition \ref{Prop: iso O/K}, if $i\in\{1,\ldots,6\}$,  
then $\F_{\Omega}:=\F\circ\pi_{\Omega/K}$ 
verifies (Ci) if and only if $\F$ verifies (Ci). 
\end{proposition}
\begin{proof}
The claim holds immediately for (C1),(C2),(C3),(C4) since for all 
$x\in X_{\Omega}$ and all $\rho>0$ one has 
$\pi_{\Omega/K}(\lambda_x(\rho))=
\lambda_{\pi_{\Omega/K}(x)}(\rho)$, and 
$\rho_{\F_\Omega}(x)=\rho_\F(\pi_{\Omega/K}(x))$.

For (C5) and (C6) we can assume $K$ algebraically closed since 
$\Gamma(\F)=
\Gamma(\F_{\widehat{K^{\mathrm{alg}}}})/
\mathrm{Gal}(K^{\mathrm{alg}}/K)$. 
By Proposition \ref{Prop: iso O/K} the claim is evident for (C5) since 
$\pi_{\Omega/K}$ induces an isomorphism 
$\Delta(x,\Gamma(\F_\Omega))\simto
\Delta(\pi_{\Omega/K}(x),\Gamma(\F))$.
Finally $\pi_{\Omega/K}$ preserves the slopes so (C6) descends.
\end{proof}

\begin{lemma}[Flat directions do not belongs to  
$\Gamma(\F)$]\label{Lemma : break at rho_F}
Let $D\subset X$ be an open virtual disk of radius $\rho$. 
Assume that $\F:X\to\mathbb{R}_{>0}$ is a function verifying
\begin{enumerate}
\item[(C1-$D$):] $\rho_{\F}(x)>0$ for all $x\in D$;
\item[(C3-$D$):] For all $x\in D$, the function 
$L_x\F$ is concave on $]-\infty,\ln(\rho)[$. 
\end{enumerate}

Then $\F$ is constant on $D$ if and only if $\F$ is constant on an 
individual complete segment $\Lambda(x)\cap D$. 

Moreover if $\F$ is non constant, then the first break of $L_x\F$ arises at $\ln(\rho_{\F}(x))$.
%$L_x\F$ has a break at $\ln(\rho_{\F}(x))$.
\end{lemma}
\begin{proof}
Assume that $\F$ is constant on $\Lambda(x)\cap D$. 
Since $D$ is topologically a tree, for all $x'\in D$ the segment 
$I:=\Lambda(x)\cap \Lambda(x')\cap D$ is not empty. 
So $\F$ is constant, with value $\F(x)$, on $I\subset\Lambda(x')$.
Now condition (C1-$D$) imply the constancy around $x'$. 
So the concavity (C3-$D$) implies constancy on the whole 
$\Lambda(x')\cap D$ (Concavity 
implies continuity on $]-\infty,\ln(\rho)[$).
Hence $\F(x)=\F(x')$.
\end{proof}

As a consequence we have the following
\begin{proposition}[Decreasing on disks]
\label{Prop : negative slope}
Assume that $\F$ satisfies (C1),(C2),(C3). Let $D$ be a virtual 
disk such that $D\cap \Gamma=\emptyset$. Let $x$ be the boundary 
point of $D$, and let $b$ be the germ of segment out of $x$ contained 
in $D$ ($b$ is oriented as out of $x$).
Then $D$ intersects $\Gamma(\F)$ if and only if 
$\partial_b\F(x)>0$ 
(equivalently $\Gamma(\F)\cap D=\emptyset$ if and only if 
$\partial_b\F(x)=0$).\hfill$\Box$
\end{proposition}

\begin{proposition}[no breaks implies no bifurcations]
\label{Prop : annulus}
Assume that $\F$ satisfies (C1), (C2), (C3), (C5), (C6), 
but not necessarily (C4). 
Let $]x,y[\subset X$ be a segment satisfying
\begin{enumerate}
\item $]x,y[$ is the analytic skeleton of a virtual open annulus 
$C(]x,y[)$ in $X$;\footnote{We recall that the 
analytic skeleton of an open 
annulus $\{x\in\mathbb{A}^{1,\mathrm{an}}_K\;|\;
0<R_1<|T-c|(x)<R_2<+\infty\}$ is the set of points without open 
neighborhoods isomorphic to a virtual open disk.}
\item $]x,y[\cap \mathscr{C}(\F)=\emptyset$, and 
$(C(]x,y[)\cap \Gamma)\subseteq ]x,y[$;
\item $\F$ has no breaks along $]x,y[$.
\end{enumerate}
Then $\Gamma(\F)$ has no bifurcations points along $]x,y[$, and $\F$ is harmonic on $C(]x,y[)$.
\end{proposition}
\begin{proof}
Let $z\in]x,y[$. 
Each direction $b$ out of $z$ which is not in $]x,y[$ 
lies inside a disk $D_b\subset C(]x,y[)$ with boundary $z$. 
By ii), Proposition \ref{Prop : negative slope} holds over $D_b$, 
so $\partial_b F(z)\geq 0$, 
and $\partial_b F(z)> 0$ if and only if $b\in\Gamma(F)$. 
If $z$ is a bifurcation point of 
$\Gamma(\F)$, this shows that 
$\sum_{b\notin]x,y[}\partial_bF(z)>0$. 
But by (C6) we have $dd^cF(z)\leq 0$, hence $F$ 
must have a break along $]x,y[$ at $z$ contradicting iii).
\if{
By (C2) and (C5) the Laplacian $dd^c\F(z)$ is defined at 
all $z\in ]x,y[$, 
and by super-harmonicity (C6) we have $dd^c\F(z)\leq 0$. 
Now by condition iii), if $b_x$ and $b_y$ 
are the germs of segment out of $z$ 
directed (and oriented) toward $x$ and $y$ respectively, then 
$\partial_{b_1}\F(z)=-\partial_{b_2}\F(z)$. 
Moreover, all other germ of segment $b$ out 
of $z$ is always contained in a disk $D_b$ with boundary $z$. 
By ii) one has $D_b\cap\Gamma=\emptyset$. 
Hence we can apply Lemma \ref{Lemma : break at rho_F}, 
to prove that $\partial_b\F(z)\geq 0$ for all $b\neq b_1,b_2$.  
This proves that $dd^c\F(z)\geq 0$, hence $dd^c\F(z)=0$, and 
$\partial_b\F(z)=0$ for all $b\in\Delta(z)-\{b_1,b_2\}$. 
We conclude by Proposition \ref{Prop : negative slope}.
}\fi
\end{proof}
\begin{proposition}[Finiteness over a disk]
\label{Prop : S(F) has finite branches in max disk}
Assume that $\F$ satisfies the six properties (C1)--(C6).
Let $D\subset X$ be an open virtual disk such that 
$D\cap(\Gamma\cup\mathscr{C}(\F))=\emptyset$. 

Then there is a finite number $N$ of bifurcation points of 
$\Gamma(\F)$ inside $D$.

Moreover, let $x$ be the point at the boundary of $D$, and let $b$ be the germ of 
segment out of $x$ contained in $D$ ($b$ is oriented as out of $x$).

Then $ N\;\leq\; \max\Bigl(0\;,\;\Bigl[\frac{\partial_b\F(\xi)}{\nu_\F}\Bigr]-1\Bigr)$, 
where $[r]$ denotes the largest integer $\leq r$.
%\begin{equation}
% N\;\leq\; \max\Bigl(0\;,\;\Bigl[\frac{\partial_b\F(\xi)}{\nu_\F}\Bigr]-1\Bigr)\;.
%\end{equation}
\end{proposition}
\begin{proof}
By Proposition \ref{Prop : negative slope} plus (C4) we can assume 
$\partial_b\F(x)\geq\nu_\F$. By (C2) there is a segment  
$]y,x[\subset D$ where $\F$ has no breaks. 
By Proposition \ref{Prop : annulus} $]y,x[$ is the skeleton 
$\Gamma_C$ of a virtual open annulus $C\subset D$ over which 
$\Gamma(\F)$ has no bifurcations.
Let $z\in D$ be the first bifurcation point  of $\Gamma(\F)$ 
that one encounters proceeding from $x$ towards the interior of $D$.  
Let $b_\infty:=]z,x[$, and let $b_1,\ldots, b_{n_z}$ 
be the others germs of segments out of $z$ 
belonging to $\Gamma(\F)$ ($b_{\infty},b_1,\ldots,b_{n_z}$ are 
now all oriented as outside $z$). 
By super-harmonicity (C6) one has $dd^c\F(z)\leq 0$, so 
\begin{equation}
\sum_{i=1}^{n_z}\partial_{b_i}\F(z)\;\leq\;-\partial_{b_\infty}\F(z)=\partial_b\F(x)\;.
\end{equation}

By Proposition \ref{Prop : negative slope} one has 
$\partial_{b_i}\F(z)>0$, 
for all $i=1,\ldots,n_z$. And, by (C4), for all $i$ one has 
$\partial_{b_i}\F(z)\geq \nu_\F$. 
%
%
%
%Now by super-harmonicity one has $dd^c\F(z)\leq 0$, so 
%\begin{equation}
%\partial_{b_\infty}\F(z)\;\leq\;
%-\sum_{i=1}^{n_z}\partial_{b_i}\F(z)\;\leq\;-n_z\cdot \nu_\F\;.
%\end{equation}
%
%
%
So, since $n_z\geq 2$, for all $i$ one has 
$\partial_{b_i}\F(z)\leq-\partial_{b_\infty}\F(z)-\nu_{\F}=
\partial_b\F(x)-\nu_\F$.
Let $D_i$ be the virtual open disk with boundary $z$ containing $b_i$. 
Then $D_i$ fulfills the same assumptions of $D$, but its last 
slope is now less than $\partial_{b}\F(x)-\nu_\F$. We then conclude 
by induction on  $[\partial_b\F(\xi)/\nu_\F]$.
\end{proof}
\begin{theorem}\label{th : finiteness theorem}
If $\F:\X\to\mathbb{R}_{> 0}$ satisfies the six conditions (C1)--(C6), 
then $\F$ is finite. 
\end{theorem}
\begin{proof}
Since $\Gamma$ is finite we are reduced to prove that 
$\Gamma'(\F):=\Gamma(\F)\cup\Gamma$ is finite. 
Moreover up to replacing $\Gamma$ by 
$\Gamma\cup\mathrm{Sat}(\C(\F))$ we can 
assume $\C(\F)\subset\Gamma$.
Since $\Gamma(\F)$ is directionally finite at its bifurcation points, it is 
enough to prove that there are a finite number of bifurcation points of 
$\Gamma'(\F)$. 

Now $X-\Gamma$ is a disjoint union of virtual open disks on which we 
can apply Proposition \ref{Prop : S(F) has finite branches in max disk}.
So, by directionally finiteness (C5), we know that for all 
$x\in\Gamma$ there are a finite number of virtual open disks $D$ 
 intersecting $\Gamma(\F)$ with boundary $x$.

Hence we are reduced to prove that there are a finite number of 
bifurcation points of $\Gamma(\F)$ belonging to $\Gamma$.
The  set $\C$ formed by the points in $\C(\F)$, the bifurcation points 
of $\Gamma$, and the points in $\Gamma\cap \partial X$, is finite 
and we can neglect it. 

So we have to prove that $\Gamma(\F)$ has a finite number of 
bifurcation points along each connected component $]x,y[$ of 
$\Gamma-\mathscr{C}$. 
This follows from Proposition \ref{Prop : annulus} and by  (C2).
\end{proof}

\subsubsection{Assumption (C4) is superfluous.}

Assumption (C4) is satisfied by the radii of convergence of a differential 
equation, and it is important for the explicit computation of the number 
of edges of $\Gamma(\F)$ (cf. \cite{NP-III}). 
So we preserve the above claims. 

Nevertheless we add the following result  
derived from Theorem \ref{th : finiteness theorem}. 
Its proof does not involve (C4), which is replaced by a 
compactness argument.

\begin{theorem}\label{Thm : copycat}
Let $\F:X\to\mathbb{R}_{>0}$ be a function satisfying 
(C1),(C2),(C3),(C5),(C6) (but not necessarily (C4)). Then $\F$ is finite.
\end{theorem}
\begin{proof}
We prove that $\Gamma'':=\Gamma(\F)\cup\Gamma\cup
\mathrm{Sat}(\mathscr{C}(\F))$ is locally finite in the 
Berkovich topology of $X$. 
Recall that this is an admissible graph in $X$, so $X-\Gamma''$ is a 
disjoint union of open disks.

Let $x\in\Gamma''$. 
Let $V(x)$ be the union of $x$ with all the virtual open
disks in $X$ with boundary $x$ on which $\F$ is constant. 
By (C5) and Proposition \ref{Prop : negative slope}, 
$V(x)$ is an affinoid domain of $X$ on which $\F$ is constant. 

Let $b_1,\ldots,b_n$ be the family of germs 
of segments out of $x$ not in 
$V(x)$, then $b_1,\ldots,b_n\in\Gamma''$. 
For all $i=1,\ldots,n$ there is $]x,y_i[\in b_i$ which is the skeleton 
of a virtual open annulus $C_i$ such that 
$\Gamma''\cap C_i=]x,y_i[$. By (C2) we can 
choose $]x,y_i[$ small enough to fulfill the assumptions 
of Proposition \ref{Prop : annulus}. Hence 
$U:=V(x)\cup(\bigcup_iC_i)$
%\begin{equation}
%U\;:=\;V(x)\cup(\bigcup_iC_i)
%\end{equation} 
is an open neighborhood of $x$ in $X$
such that $U\cap\Gamma''=\cup_{i=1}^n[x,y_i[$. 
Together with the complement of $\Gamma''$ in $X$, this gives a covering of $X$ by 
opens whose intersection with $\Gamma''$ is a finite graph.
Since $X$ is compact, we can extract a finite sub-covering, so $\Gamma''$ is finite.
\end{proof}
\subsubsection{Non compact disks and annuli.}
Let $C(I)=\{x\in\mathbb{A}^{1,\mathrm{an}}_K,|T|(x)\in I\}$ be a 
(possibly not closed) annulus, or disk if $0\in I$. 
Definition \ref{def :constancy sk} extends to $X=C(I)$ in an evident way.

In this case $\Gamma(\F)$ is finite if 
there is a compact sub-interval $J\subset I$  
(resp. if $0\in I$, then $0\in J$)
such that $\Gamma(\F_{|J})$ is finite over $C(J)$, and 
$\S(\F)=\S(\F_{|C(J)})\cup\Gamma_{C(I)}$.
%\begin{equation}\label{eq : hdpmyyyyy}
%\S(\F)\;=\;\S(\F_{|C(J)})\cup\Gamma_{C(I)}\;.
%\end{equation} 
\begin{corollary}\label{cor : finiteness on a disk}
Let $\F:C(I)\to\mathbb{R}_{>0}$. Assume that $\C(\F)$ is finite, 
and contained in $C(J)$, for some compact $J\subset I$.
If $\F_{|C(J)}$ is finite, and if $\F$ is $\log$-affine along each 
connected component of $I-J$, then $\F$ is finite and 
$\S(\F)=\S(\F_{|C(J)})\cup\Gamma_{C(I)}$.
\end{corollary}
\begin{proof}
Apply Prop. \ref{Prop : annulus} over the open annuli that are 
connected components of $C(I)-C(J)$.
\end{proof}

\begin{example}\label{ex: exercice ttt}
1. Let $\Gamma$ be a finite admissible graph.
The function $\xi\mapsto\rho_{\Gamma}(x)$ verifies the six properties 
(C1)--(C6) with respect to $\Gamma$, and 
$\C(\rho_{\Gamma})=\emptyset$. If $I\subseteq\Gamma$ is any segment contained in some $\Lambda(x)$, 
and if $I$ is oriented as towards $+\infty$, then $\rho_{\Gamma}$ is 
$\log$-affine on $I$ with slope $+1$. In particular it is super-harmonic 
in the sense of definition \ref{Def : sub-harmonic}, and 
$\Gamma(\rho_{\Gamma})=\Gamma$.

2. If $\F_1,\ldots,\F_n$ are functions satisfying the six properties (C1)--(C6), then so does $\min(\F_1,\ldots,\F_n)$.

3. Let $f_1,\ldots,f_n\in\O(X)$ and $\alpha_1,\ldots,\alpha_n>0$. 
Assume that each $f_i$ has no zeros on $X(\Ka)$. Then the function 
$\F(x):=\min_i|f_i|(x)^{-\alpha_i}$ verifies (C1)--(C6), with 
$\Gamma=\Gamma_X$, and $\Gamma(\F)=\Gamma_X$. 
Moreover $\F$ is also super-harmonic (cf. Def. 
\ref{Def : sub-harmonic}) because so is each function 
$x\mapsto|f_i(x)|^{-\alpha_i}$. 
\end{example}

\section{Radii of convergence and statement of main result}
\label{Radii of convergence and statement of main result}
We here give the definition of the radii of convergence 
\eqref{eq : (3.5)}, and of
the convergence Newton polygon \eqref{eq : NPconv}. 
We then state our main result 
(cf. Thm. \ref{Theorem : MAIN THM GEN}) whose proof will be given in the next sections.

\subsection{Newton polygons (formal definition).}
Let $r\geq 1$ be a natural number. 
Let $v:\{0,1,\ldots,r\}\to\mathbb{R}\cup\{+\infty\}$,
%\begin{equation}
%v\;:\;\{0,1,\ldots,r\}\to\mathbb{R}\cup\{+\infty\}\;,
%\end{equation} 
be any sequence $i\mapsto v_i$ satisfying $v_0=0$.
The \emph{Newton polygon $NP(v)\subset\mathbb{R}^{2}$} is the 
convex hull in $\mathbb{R}^2$ of the family of half-lines 
$L_v:=\cup_{i=0,\ldots,r}\{(x,y)\in\mathbb{R}^2\;|\; 
x=i,y\geq v_i\}$ 
i.e. the intersection of all upper half planes 
$H_{a,b}:=\{(x,y)\in\mathbb{R}^2\textrm{ such that }
y\geq ax+b\}$, $a,b\in\mathbb{R}$, containing $L_v$.

For $i=0,\ldots,r$, 
we call the \emph{$i$-th partial height} of the polygon the value
\begin{equation}
h_i\;:=\;
\min\{y\in\mathbb{R}\cup\{+\infty\}
\textrm{ such that }(i,y)\in NP(v)\}\;.
\end{equation}

If $h:\{0,\ldots,r\}\to\mathbb{R}\cup\{+\infty\}$ denotes the function 
$i\mapsto h_i$, 
then $NP(v)=NP(h)$, and $h$ is the smallest function with this 
property. 

We have $h_i=\sup_{s\in\mathbb{R}}
(s\cdot i+\min_{j=0,\ldots,r}(v_j-s\cdot j))$.
In fact if $y=sx+q_s$ is the line of slope $s$ which is tangent to 
$NP(v)$, then $q_s=\min_{j=0,\ldots,r}(v_j-s\cdot j)$, 
and $h_i$ is the supremum of the values of those lines at $x=i$. 
In particular, since $v_0=0$, for $i=1$ we have $h_1=\min_{i=1,\ldots,r}(v_i/i)$.

We call  \emph{slope sequence} any increasing sequence $s:\{1,
\ldots,r\}\to\mathbb{R}\cup\{+\infty\}$: $s_1\leq\ldots\leq s_r$. 

The \emph{slope sequence of $NP(v)=NP(h)$} is defined by 
$s_i:=h_{i}-h_{i-1}$, $i=1,\ldots,r$, 
where $s_i=+\infty$ if $h_{i}$ or $h_{i-1}$ are equal to $+\infty$. 
The slope sequence of $NP(h)$ determines the function 
$h_i=s_1+\cdots+s_i$, and hence $NP(h)$. 

If $s_i<s_{i+1}$, or if $i=r$, we say that 
\emph{$i$ is a vertex of $NP(v)$}.

Let $s:s_1\leq\ldots\leq s_r$ be a slope sequence, the 
\emph{truncated slope sequence} by the constant 
$C\in\mathbb{R}$ is by definition the sequence 
$s|_C:=(s'_i)_{i=1,\ldots,r}$, where $s'_i:=\min(s_i,C)$, for all $i$.

As a matter of facts in the sequel we will deal only with truncated 
slope sequences by a convenient constant $C<+\infty$, so we do 
not have to deal with infinite slopes.
\begin{example} 
Let $(F,|.|_F)$ be a valued field and let 
$P(T):=\sum_{i=0}^ra_{r-i}T^i\in F[T]$ be such that $a_0=1$. 
Let $v_{P,i}:=-\ln(|a_i|)\in\mathbb{R}\cup\{+\infty\}$. 
The \emph{Newton polygon of $P(T)$} is by definition $NP(v_P)$.
\end{example}

\subsection{Convergence Newton polygon of a differential equation}

Let $X$ be an affinoid domain of $\mathbb{A}^{1,\mathrm{an}}_K$.
A differential equation over $X$ is a locally free 
$\O_X$-module $\Fs$ of finite rank together with a connection 
$\nabla:\Fs\to\Fs\otimes\Omega^1_{X}$. Let $r$ be the rank of 
$\Fs$.

We now define the radii of $\Fs$ at $x\in X$. 
We fix a field extension $\Omega\in E(\H(x))$ which is 
algebraically closed, spherically complete, and with value group 
$|\Omega^\times|=\mathbb{R}_{> 0}$. 
Let $t\in X(\Omega)$ be a Dwork generic point for $x$, and let 
$\Fs_{|D(x,X)}$ be the restriction of 
$\Fs_\Omega=\Fs\widehat{\otimes}_K\Omega$ to 
$D(x,X)\subset X_\Omega$. 

We recall that the radius of $D(x,X)$ is $\rho_{x,X}$. 
For all $0<R\leq\rho_{\xi,X}$ we denote by $D(x,R)\subset D(x,X)$ 
the open sub-disk centered at $t$ with radius $R$, and by 
$\mathrm{Fil}^{\geq R}\Sol(\Fs,t,\Omega)\subset\Fs_{|D(x,R)}$
%
%\begin{eqnarray}
%\mathrm{Fil}^{\geq R}\Sol(\Fs,t,\Omega)
%&\;\subset\;&
%\Fs_{|D(x,R)}
%\end{eqnarray} 
the $\Omega$-vector space of solutions of $\Fs$ with values in 
$\O(D(x,R))$, i.e. the kernel of $\nabla\otimes 1+1\otimes d/dT$ 
acting on $\Fs_{|D(x,R)}$. 
The space $\Sol(\Fs,t,\Omega)$ of all Taylor 
solutions of $\Fs$ around $t$ is given by
\begin{equation}
\label{eq : Sol(F,t,Omega)}
\Sol(\Fs,t,\Omega)\;:=\;
\bigcup_{R>0}\mathrm{Fil}^{\geq R}\Sol(\Fs,t,\Omega)\;.
\end{equation}
Since $\Omega$ is spherically closed, by a result of Lazard 
\cite{Lazard}, $\Fs_{|D(x,X)}$ is free. 
So, once a basis is chosen we have 
a differential equation $Y'=G\cdot Y$, $G\in M_r(\O(D(x,X)))$, and 
hence, by the Cauchy existence theorem,  $\Sol(\Fs,t,\Omega)$ has 
dimension $r$ over $\Omega$ (cf. \cite[Appendix]{DGS}). 
If $\Omega\subseteq\Omega'$, a descent argument 
(cf. \cite[Prop. 6.9.1]{Kedlaya-Book}) shows that  
$\mathrm{Fil}^{\geq R}\Sol(\Fs,t,\Omega')=
\mathrm{Fil}^{\geq R}\Sol(\Fs,t,\Omega)
\widehat{\otimes}_\Omega\Omega'$.
%\begin{equation}
%\mathrm{Fil}^{\geq R}\Sol(\Fs,t,\Omega')\;=\;
%\mathrm{Fil}^{\geq R}\Sol(\Fs,t,
%\Omega)\widehat{\otimes}_\Omega\Omega'\;.
%\end{equation}

\begin{proposition}\label{Prop : Fil Indep on tOmega}
The filtration is independent on the choice of $\Omega$ and $t$ in the 
following  sense. 
If $(t',\Omega')$ is another choice, there exists 
$\Omega,\Omega'\leq \Omega''\in E(K)$, together with 
a Galois automorphisms $\sigma\in \mathrm{Gal}^{\mathrm{cont}}(\Omega''/K)$, 
such that $\sigma(t)=t'$, inducing for all  $R\leq \rho_{x,X}$ the identification 
\begin{equation}\label{eq : Fil_it=Fil_it'}
\sigma\;:\;\mathrm{Fil}^{\geq R}\Sol(\Fs,t,\Omega)
\widehat{\otimes}_\Omega\Omega''\xrightarrow[\sigma]{\;\;\sim\;\;}
\mathrm{Fil}^{\geq R}\Sol(\Fs,t',\Omega'')
\widehat{\otimes}_{\Omega'}\Omega''\;.
\end{equation}
\end{proposition}
\begin{proof}
The existence of $\sigma$ such that $\sigma(t)=t'$ follows from 
Lemma \ref{Gal transitive}. 
Since $\sigma$ is isometric, then 
$\sigma(D^-(t,R))=D^-(t',R)$ for all 
$0<R\leq\rho_{\xi,X}$. This provides an isomorphism of rings 
$\sum a_i (T-t)^i\mapsto\sum\sigma(a_i)(T-t')^i:
\O(D^-(t,R))\simto\O(D^-(t',R))$, over $\Omega''$, 
commuting with $d/dT$.
\end{proof}

\begin{definition}[Convergence radii]\label{Def. CRAD}
For all $i=1,\ldots,r$ we define $\R^{\Fs}_i(\xi)$ as the largest value 
of $R\leq\rho_{\xi,X}$ such that 
$\mathrm{dim}_{\Omega}\mathrm{Fil}^{\geq R}
\mathrm{Sol}(\Fs,t,\Omega)\geq r-i+1$. 
We set $H^{\Fs}_i(x):=\prod_{k=1}^i\R^{\Fs}_k(x)$ and
\begin{equation}\label{eq : (3.5)}
\R_i(x,\Fs)\;:=\;\R_i^{\Fs}(x)/\rho_{x,X}\;,\qquad
H_i(x,\Fs)\;:=\;\prod_{k=1}^i\R_k(x,\Fs)=H_i^{\Fs}(x)/\rho_{x,X}^i\;.
\end{equation}
We also set $s_i^\Fs(x):=\ln(\R^{\Fs}_i(x))$ and 
$h_i^\Fs(x):=s_1^\Fs(x)+\cdots+s_i^\Fs(x)$, $h_0^\Fs(x)=0$. 

The polygon $NP(\ln(H_i(x,\Fs)))$ is called the 
convergence Newton polygon and it is denoted by
\begin{equation}\label{eq : NPconv}
NP^{\mathrm{conv}}(x,\Fs)\;.
\end{equation}
\end{definition}
\begin{remark}
\label{Def.: radius on a disk}
(1) By Prop. \ref{Prop : Fil Indep on tOmega}, the above 
functions are independent on the choices of $t$ and $\Omega$. 

(2) Obviously the definition only depend on the restriction 
$\Fs_{|D(x,X)}$, so the same definitions can be given for a differential 
module over a virtual open disk $D$, replacing $\rho_{x,X}$ by the 
radius of $D$.

(3) In particular the definition is insensitive by extension of $K$: for all 
$\Omega\in E(K)$ and all $y\in X_\Omega$
\begin{equation}
\R_i(y,\Fs_\Omega)\;=\;\R_i(\pi_{\Omega/K}(y),\Fs)\;,
\qquad \forall \;i=1,\ldots,r\;.
\end{equation}
In particular the assumptions of Proposition 
\ref{Invariance of (C1)--(C6) by scalar extension of the ground field} 
are verified.

(4) Since $y\mapsto\rho_{y,X}$ is constant on each maximal disk 
$D(x,X)$, it immediately follows that
\begin{equation}
\Gamma(\R_i(-,\Fs))\;=\;\Gamma(\R^{\Fs}_i)\;,\qquad
\Gamma(H_i(-,\Fs))\;=\;\Gamma(H^{\Fs}_i)\;.
\end{equation}

(5) More precisely $\R_i(-,\Fs)$ and $\R_i^{\Fs}$ differ by a constant 
function over each maximal disk $D(x,X)$. Hence if $b$ is a germ of 
segment out of $x\in X$ we have either 
$\partial_b\R_i(x,\Fs)=\partial_b\R_i^{\Fs}(x)$ if 
$b\notin\Gamma_X$, or 
$\partial_b\R_i(x,\Fs)=\partial_b\R_i^{\Fs}(x)-1$ otherwise 
if $b$ is oriented as towards $+\infty$.

(6) The dimension $\mathrm{dim}_\Omega\mathrm{Fil}^{\geq R}\mathrm{Sol}(\Fs,t,\Omega)$ is obviously 
constant on $D(x,R)$. Hence $\R_i(x,\Fs)$ and $\R_i^{\Fs}(x)$ are constant on 
$D(x,\R_i^{\Fs}(x))$, so
\begin{equation}\label{(2)}
\max(\R_i^{\Fs}(x),r(x))\;\leq\;
\rho_{\R_i^{\Fs}}(x)\;=\;\rho_{\R_i(-,\Fs)}(x)\;.
\end{equation}

(7) It follows  from the definition that if 
$\Fs'\subset\Fs$ is a sub-differential equation, the radii of $\Fs'$ 
 all appear among the radii of $\Fs$, with at least the same multiplicity than they had in $\Fs'$.
\end{remark}

The radii do not behave well by exact sequences, but we have the 
following
\begin{proposition}\label{Prop. direct sum then ok with radii}
Let $\Fs=\Fs_1\oplus\Fs_2$ be a direct sum of differential equations 
over $X$ of ranks $r_1$ and $r_2$ respectively. 
Then, up to 
permutation\footnote{If a radius $R$ appears 
$n_i$-times in $NP^{\mathrm{conv}}(\Fs_i,\xi)$, it is understood that 
it appears $n_1+n_2$-times in $NP^{\mathrm{conv}}(\Fs,\xi)$.}, for all $\xi\in X$ one has  
\begin{equation}
\{\R^{\Fs}_1(x),\ldots,\R^{\Fs}_{r_1+r_2}(x)\}\;=\;
\{\R^{\Fs_1}_1(x),\ldots,\R^{\Fs_1}_{r_1}(x)\}
\cup
\{\R^{\Fs_2}_1(x),\ldots,\R^{\Fs_2}_{r_2}(x)\}\;.
\end{equation}
The same holds replacing $X$ by an open disk, or replacing 
$\R_{i}^{\Fs}$ by $\R_i(-,\Fs)$.
\end{proposition}
\begin{proof}
The functor $\Fs\mapsto \mathrm{Fil}^{\geq R}\Sol(\Fs,t,\Omega)$ is 
additive, so for $R\leq\rho_{x,X}$ we have 
$\mathrm{Fil}^{\geq R}\Sol(\Fs,t,\Omega)=
\mathrm{Fil}^{\geq R}\Sol(\Fs_1,t,\Omega)\oplus
\mathrm{Fil}^{\geq R}\Sol(\Fs_2,t,\Omega)$.
%\begin{equation}
%\mathrm{Fil}^{\geq R}\Sol(\Fs,t,\Omega)\;=\;
%\mathrm{Fil}^{\geq R}\Sol(\Fs_1,t,\Omega)\oplus
%\mathrm{Fil}^{\geq R}\Sol(\Fs_2,t,\Omega)\;. 
%\end{equation}
The claim then follows directly from Definition \ref{Def. CRAD}.
\end{proof}

\subsection{Statement of main result}
%Recall that the index $i$ is a \emph{vertex at $x$} of $NP^{\mathrm{conv}}(x,\Fs)$  if $\R_i(x,\Fs)<\R_{i+1}(x,\Fs)$, 
%or if $i=r$.
%
\begin{definition}\label{Def : index free of solvabl}
We say that the index $i$ (resp. $\R_i(x,\Fs)$) is  
\begin{equation}
\left\{\begin{array}{lcl}
\textrm{spectral at }x\in X&\textrm{ if }& \R_i^{\Fs}(x)\leq r(x)\;,\\
\textrm{solvable at }x\in X&\textrm{ if }& \R_i^{\Fs}(x)=r(x)\;,\\
\textrm{over-solvable at }x\in X&\textrm{ if }&\R_i^{\Fs}(x)>r(x)\;.
\end{array}\right.
\end{equation}

We say that the index $i$ is \emph{free of solvability at $x$}  
if none of the indexes $j\leq i$ is solvable.

We say that $\Fs$ is \emph{free of solvability at $x$} if none of the 
indexes $i=1,\ldots,r$ is solvable at $x$.
\end{definition}
\begin{remark}\label{Rk : radius is spectral over controlling graph}
From Prop. \ref{prof propp} and \eqref{(2)}
it follows that $i$ is spectral at all points of $\Gamma(\R_i(-,\Fs))$. 
\end{remark}
\begin{definition}\label{Def : Gamma_i}
For all $i=1,\ldots,r$, we set 
\begin{equation}
\Gamma_0\;:=\;\Gamma_X\;,\qquad\quad\Gamma_i\;:=\;\bigcup_{j=1,\ldots,i}\;\Gamma(\R_{j}(-,\Fs))\;.
\end{equation}
\end{definition}
Recall that the index $i$ is a \emph{vertex at $x$} of $NP^{\mathrm{conv}}(x,\Fs)$  if $\R_i(x,\Fs)<\R_{i+1}(x,\Fs)$, 
or if $i=r$.

The main result of this paper is the following:
\begin{theorem}\label{Theorem : MAIN THM GEN}
Let $\Fs$ be a differential module of rank $r$ over $X$.

For $i=1,\ldots,r$ the functions $\R_i(-,\Fs)$ and $H_i(-,\Fs)$ 
(hence also $s_i^\Fs$, $h_i^\Fs$, $\R^{\Fs}_i$, $H^{\Fs}_i$) 
are finite. 

They enjoy moreover the following properties:

\begin{enumerate}
\item For all $i=1,\ldots,r$ the $i$-th partial heights $H_i(-,\Fs)$ and $H_i^{\Fs}$ both verify 
(C1), (C2), (C4), (C5) of Section \ref{F_t defi}, 
and also (C3) with respect to $\Gamma:=\Gamma_{i-1}$. 
\item\;\!\!\textbf{$\bs{[}$Integrality$\bs{]}$} Let 
$\xi\in X$ be a point, then:
\begin{enumerate}
\item%[(i.a)] 
If $i$ is a vertex of 
$NP^{\mathrm{conv}}(x,\Fs)$, 
then for all germ of segment $b$ out of $x$, we have 
\begin{equation}\label{eq : integrality of b}
\partial_b H_i(x,\Fs)\;,\;\partial_bH_i^{\Fs}(x)\;\in\;\mathbb{Z}\;.
\end{equation}
\item%[(i.b)] 
If $i$ is not a vertex, one proves by 
interpolation\footnote{Interpolation means that we proceed as in
the proof of point iv) of Proposition \ref{Prop. NP of a operator}.} 
from \eqref{eq : integrality of b} that 
\begin{equation}
\partial_bH_i(x,\Fs)\;,\;\partial_bH_i^{\Fs}(x)\;\in \;
\mathbb{Z}\cup\frac{1}{2}\mathbb{Z}\cup\cdots\cup
\frac{1}{r}\mathbb{Z}\;.
\end{equation} 
\end{enumerate}
\item\;\!\!\textbf{$\bs{[}$Concavity$\bs{]}$}
Let $]x,z[$ be an 
open segment in $X$. Let $]x,y[:=]x,z[-\Gamma_X$.\footnote{In 
other words if $D$ is the largest virtual open disk in $X$ 
intersecting $]x,z[$, then $]x,y[=D\cap]x,z[$. If 
$]x,z[\subset\Gamma_X$, it is understood that $]x,y[=\emptyset$.}  
For all $i=1,\ldots,r$ let $H_i$ denote  the $i$-th partial height 
$H_i(-,\Fs)$ or $H^{\Fs}_i$. 
Then:
\begin{enumerate}
\item
$H_i$ is $\log$-concave on each 
sub-segment of $]y,z[$ which is the skeleton of a virtual 
annulus contained in $X$ (cf. Def. \ref{Def. : LOG-convex}).
\item
$H_i$ is $\log$-concave on each sub-segment 
of $]x,y[$ which does not contain the points 
\begin{equation}
\{\lambda_x(\R_{1}^{\Fs}(x)),\ldots,\lambda_x(\R_{i}^{\Fs}(x))\}\;.
\end{equation}
Moreover let $\tau\in\{\R^{\Fs}_k(x)\}_{k\leq i}$. 
If for all $k\leq i$ such that 
$\R^{\Fs}_k(x)=\tau$ the function $\R_k(-,\Fs)$ 
(or equivalently $\R_k^{\Fs}$) 
is $\log$-concave at $\log(\tau)$, then 
 $H_i$ is also $\log$-concave at 
$\log(\tau)$.\footnote{In particular this 
happens by definition if $\tau<r(x)$, 
since $L_x\R_k(-,\Fs)$ is constant on $]-\infty,r(x)]$.}

\item
$H_i$ is logarithmically non-increasing on each sub-segment $I\subset ]x,y[$ on which $i$ is free of solvability
%over which the indexes $1,2,\ldots,i$ are never solvable 
(i.e. $\R^{\Fs}_j(x')\neq r(x')$ for all $x'\in I$, and all $j\leq i$).
\end{enumerate}
\item\;\!\!\textbf{$\bs{[}$Weak super-harmonicity$\bs{]}$} 
We define inductively a family $\C_{1}(\Fs),\ldots,\C_{r}(\Fs)\in 
X-\Gamma_X$ of finite subsets as
\begin{equation}
\C_i\;:=\;\cup_{i=1}^iA_i\;,
\end{equation}
where $A_i$ is the finite set of points $\xi\in X$ satisfying 
\begin{enumerate}
\item The index $i$ is solvable at $x$; 
\item
$\xi$ is an end point of  $\S(\R_i(-,\Fs))$.
\item
$\xi\in\S(\R_i(-,\Fs))\cap\S(H_i(-,\Fs))\cap\S_{i-1}$;
\end{enumerate}
Then for all 
$x\notin S_X\cup \C_i$ (cf. \eqref{eq : S_X}) we have
\begin{equation}\label{eq : H_i sup-harm}
dd^cH_i(x,\Fs)\;\leq \;0\;.
\end{equation}
While for $x\in S_X-\partial X$ we have
\begin{equation}\label{eq : H_i GOS}
dd^cH_i(x,\Fs)\;\leq\;(N_X(x)-2)\cdot\min(i,i^{\mathrm{sp}}_x)\;,
\end{equation}
where $N_X(x)=\sum_{b\in\Delta(x,\Gamma_X)}m_b$, where $m_b$ 
is the multiplicity of $b$ (cf. Definition \ref{Def : Laplacian}), and 
$0\leq i^{\mathrm{sp}}_x\leq r$ is the largest index of $\Fs$ which is 
spectral non solvable at $x$.\footnote{It is understood that 
$i^{\mathrm{sp}}_x=0$ if and only if all the radii of $\Fs$ are 
solvable or over-solvable at $x$.}

This is equivalent to say that 
$H^{\Fs}_i$ is super-harmonic (at least) at all 
$x\in X-(\C_{i}\cup\partial X)$. 

In particular $\R^{\Fs}_1$ is super-harmonic (outside $\partial X$). 

\item\;\!\!\textbf{$\bs{[}$Weak harmonicity of the vertexes$\bs{]}$}  
Let $\xi\in X-\partial X$. Then:
\begin{enumerate}
\item 
If $x\notin \Gamma(H_i(-,\Fs))$, then  for all $b\in\Delta(x)$
we have $\partial_bH_i(x,\Fs)=0$, so $H_i(-,\Fs)$ is harmonic at $x$;
\item 
If $x\in\Gamma(H_i(-,\Fs))$, and if $i$ is a vertex free of solvability at $x$, then 
\eqref{eq : H_i sup-harm} and \eqref{eq : H_i GOS} are equalities.
In particular $H_i^{\Fs}$ is harmonic at $x$.
\end{enumerate}	 

\end{enumerate}
\end{theorem}
The proof of Theorem \ref{Theorem : MAIN THM GEN} 
is placed in section \ref{Proof of the MTH}. 

As a straightforward generalization of Theorem
\ref{Theorem : MAIN THM GEN} we have the following 
\begin{corollary}\label{Corollary - after thm}
Let $C(I):=\{x\textrm{ such that }|T|(x)\in I\}$ be a possibly 
not closed annulus or disk (if 
$0\in I$ one has a disk). 
Let $\Fs$ be a differential module of rank $r$ 
over a differential ring $\O$. 

Then Theorem \ref{Theorem : MAIN THM GEN} holds for $\Fs$ in the 
following cases:
\begin{enumerate}
\item if $\O$ is the ring of Krasner analytic elements over $C(I)$ (cf. 
\cite[Def. 8.1.1]{Kedlaya-Book}); 
\item if $K$ is discretely valued, and $\O$ 
is the ring $\mathcal{B}(C(I))$ of bounded analytic functions 
on $C(I)$; 
\item if $\O=\mathcal{B}(C(I))$ or $\O=\O(C(I))$, and all 
$\R_1(-,\Fs),\ldots,\R_r(-,\Fs)$ (or equivalently all $H_i(-,\Fs)$) 
have a finite number of breaks along the skeleton 
$\Gamma_{C(I)}=\{\xi_{0,\rho}\}_{\rho\in I}$.
\end{enumerate}
Moreover if $\O=\mathcal{B}(C(I))$ or $\O=\O(C(I))$, and if 
there exists $i\leq r$ such that all $\R_1(-,\Fs),\ldots,\R_i(-,\Fs)$ 
have a finite number of breaks along $\Gamma_{C(I)}$, 
then $\R_1(-,\Fs),\ldots,\R_i(-,\Fs)$ are finite. \hfill$\Box$
\end{corollary}

\begin{corollary}\label{Cor : Rk=1 then boundary explicit}
Assume the $\mathrm{rank}(\Fs)=1$, and that $x\notin \Gamma_X$. 
Then $x$ is an end point of $\Gamma(\R_1(x,\Fs))$ if and only if 
$\R_{1}(x,\Fs)$ is solvable at $x$ 
and $\partial_{b_\infty}\R_1(x,\Fs)<0$, 
where $b_{\infty}$ denotes the germ of segment out of $x$ directed 
towards $+\infty$ (and oriented as out of $x$).
\end{corollary}
\begin{proof}
If $\R^{\Fs}_1(x)=r(x)$ and if $\partial_{b_\infty}\R_1(x,\Fs)<0$, then 
$\rho_{{\R_1(-,\Fs)}}(x)=r(x)$ by 
Lemma \ref{Lemma : break at rho_F}. Hence 
$x\in\Gamma(\R_1(-,\Fs))$ by Proposition \ref{prof propp}. 
Now $x$ is an end point of $\Gamma(\R_1(-,\Fs))$ 
by Lemma \ref{Lemma : KEY . .!}. 

Reciprocally by Lemma \ref{Lemma : break at rho_F} and Proposition 
\ref{Prop : negative slope} a boundary point $x$ of 
$\Gamma(\R_1(-,\Fs))$ not in $\Gamma_X$ verifies 
$\partial_{b}\R_1(x,\Fs)=0$ for all 
$b\neq b_\infty$, and $\partial_{b_\infty}\R_1(x,\Fs)<0$. 
In particular $\R_1(-,\Fs)$ is not harmonic at $x$. 
Hence $\R_1(-,\Fs)$ must be 
solvable at $x$ by point iv) of Theorem \ref{Theorem : MAIN THM GEN}. 
\end{proof}
\begin{remark} \label{Rk : Formal case}
Assume $K$ trivially valued. The field of Laurent formal power series 
$K((T))$ (resp. Laurent polynomials $K[T,T^{-1}]$) 
coincides in this case with the ring of analytic functions over 
$\{|T|\in I\}$ for all (open or closed) interval $I\subseteq ]0,1[$ 
(resp. $I\subseteq\mathbb{R}_{>0}$, with $1\in I$). 
Analytic functions are always bounded, and point ii) of Corollary \ref{Corollary - after thm} holds. 
Moreover the radii have no breaks along $]0,x_{0,1}]$, and all differential equation are solvable at $x_{0,1}$.
Moreover $\omega=1$, and the radii are always explicitly 
intelligible by Prop. \ref{Prop. : small radius}.
The slopes along $]0,x_{0,1}[$ are also directly related to the	 
Formal Newton polygon of $\Fs$  
\cite{Ramis-Gevrey}, 
\cite[p.97--107]{Correspondance-Malgrange-Ramis}, 
\cite{Robba-Hensel} (see \cite{NP-III} for more details).
\end{remark}
\begin{remark}
In another language, if $K$ is spherically complete 
and $|K|=\mathbb{R}$, Theorem \ref{intro-Thm.7} 
says in particular that the functions $\R_i(-,\Fs)$ are all 
\emph{definable} in the 
sens of \cite{loeser}. 
%as soon as $\Gamma(\R_i(-,\Fs))$ has no 
%end-points of type $3$ nor $4$, and the breaks of the function arises 
%only at points of type $2$.
\end{remark}
The remaining of the paper is devote to prove Theorem 
\ref{Theorem : MAIN THM GEN}. 
The definition of $\R^{\Fs}_i$ and $\R_i(-,\Fs)$ are stable by scalar 
extensions of $K$ (cf. Remark \ref{Def.: radius on a disk}). So we 
assume the following
\begin{hypothesis}
From now on we assume $K$ algebraically closed.
\end{hypothesis}

\section{Spectral polygons and related results}
\label{Spectral polygons and related results}
\if{We now introduce the spectral polygons of a differential equation 
(cf. Section \ref{NP of a module}) and  of a differential operator 
(cf. Section \ref{Spectral Newton polygon of a differential operator.}). 
We quickly recall and prove several results. 
}\fi
The ring $\O(X)$ is a principal ideal domain, whose ideals 
are generated by a polynomial, hence there are no non trivial 
ideals stable by $d/dT$. This implies that each coherent 
$\O_X$-module with connection is free over $\O(X)$ (the proof of 
\cite[9.1.2]{Kedlaya-Book} works).
The choice of a basis $e_1,\ldots,e_r\in\Fs(X)$ 
gives an isomorphism $\Fs(X)\simto\O(X)^r$ in which the 
connection $\nabla$ becomes of the form 
\begin{equation}
\nabla(f_1,\ldots,f_r)^t\;=\;
(f_1',\ldots,f_r')^t-G\cdot (f_1,\ldots,f_r)^t\;,
\end{equation}
 with 
$G\in M_{r\times r}(\O(X))$, where 
$\Omega^1_X(X)\simto\O(X)$ via the map $f\cdot dT\mapsto f$. 
The matrix $G$ is called the matrix of $\nabla$. In that basis, the 
fundamental Taylor solution matrix of $\Fs$ at a point  $t\in X(\Omega)$ is 
\begin{equation}
Y(T,t)\;:=\;\sum_{n\geq 0}G_n(t)(T-t)^n/n!\;,
\end{equation}	
where $G_n$ is 
inductively defined by $G_0=\mathrm{Id}$, $G_1=G$, 
$G_{n+1}=G_nG+G_n'$. The columns of $Y(T,t)$ form a basis of 
$\mathrm{Sol}(\Fs,t,\Omega)$ (cf. \eqref{eq : Sol(F,t,Omega)}). 
We set
\begin{equation}\label{eq : R^Y}
\R^Y(x)\;:=\;
\liminf_{n}|G_n/n!|(x)^{-1/n}
\;=\;
\liminf_{n}|G_n(t)/n!|^{-1/n}_\Omega \;.
\end{equation}
This is a  function 
$\R^Y:X\xrightarrow{\quad}\mathbb{R}_{\geq 0}\cup\{+\infty\}$. 
Clearly $\R^{\Fs}_1(x)=\min(\R^Y(x),\rho_{x,X})$, 
and we set
\begin{equation}\label{eq : Def : R_1sp}
\R_1^{\Fs,\mathrm{sp}}(x)\;:=\;\min(\R^Y(x),r(x))\;=\;
\min(\R^{\Fs}_1(x),r(x))\;.
\end{equation}
Te function $\R_1^{\Fs,\mathrm{sp}}:
X\to[0,R_0]$ is called 
\emph{spectral radius of $\Fs$} (or also \emph{generic radius}).

Notice that $\R^Y(x)$ depends on the chosen basis of $\Fs(X)$, 
while $\R_1^{\Fs,\mathrm{sp}}(x)$ does not.
\begin{lemma}
For all $x\in X$, one has
\begin{equation}\label{(3)}
\rho_{\R^Y}(x)\;=\;
\rho_{\R^{\Fs}_1}(x)\;=\;
\rho_{\R_1(-,\Fs)}(x)\;,\qquad
\rho_{\R_1^{\Fs,\mathrm{sp}}}(x)\;=\;r(x)\;.
\end{equation}
\end{lemma}
\begin{proof}
We have $Y(T,t)\in GL_r(\O(D^-(t,\R_1^{\Fs}(t))))$, and if 
$|t'-t|<\R_1^{\Fs}(t)$, one has the cocycle relation 
$Y(T,t)=Y(T,t')\cdot Y(t',t)$ (cf. \cite{Astx}).
From this it follows that $\R^Y(t)\geq\R^Y(t')$, and by symmetry we 
have $\R^Y(t)=\R^Y(t')$. 
Hence $\R^Y$ and $\R^{\Fs}_1$ are both constant on 
$D^-(t,\R_1^{\Fs}(x))$. 

The claim follows from this fact, together with \eqref{encadrement rho} and $\R_1^{\Fs}(x)=\min(\R^Y(x),\rho_{x,X})$.
\end{proof}
\if{We have $Y(T,t)\in GL_r(\O(D^-(t,\R_1^{\Fs}(t))))$, and if 
$|t'-t|<\R_1^{\Fs}(t)$, one has the cocycle relation 
$Y(T,t)=Y(T,t')\cdot Y(t',t)$ (cf. \cite{Astx}).
From this it follows that $\R^Y(t)\geq\R^Y(t')$, and by symmetry we 
have $\R^Y(t)=\R^Y(t')$. 
Hence $\R^Y$ and $\R^{\Fs}_1$ are both constant on 
$D^-(t,\R_1^{\Fs}(x))$. 

This, together with \eqref{encadrement rho} and $\R_1^{\Fs}(x)=\min(\R^Y(x),\rho_{x,X})$, 
implies that, for all $x\in X$, one has
\begin{equation}\label{(3)}
\rho_{\R^Y}(x)\;=\;
\rho_{\R^{\Fs}_1}(x)\;=\;
\rho_{\R_1(-,\Fs)}(x)\;,\qquad
\rho_{\R_1^{\Fs,\mathrm{sp}}}(x)\;=\;r(x)\;.
\end{equation}
}\fi
From \eqref{(3)} and 
Lemma \ref{Lemma : x in Gamma(R) ssi r(x)=rho_R(x)}  
one immediately has (here $r$ is the function of  \eqref{eq : r_K})
\begin{equation}\label{S(FM)=S(FY)}
\Gamma(\R^Y)\;=\;
\Gamma(\R_1^{\Fs})\;=\;
\Gamma(\R_1(-,\Fs))\;,\qquad 
\Gamma(\R_1^{\Fs,\mathrm{sp}})\;=\;\Gamma(r)\;=\;X\;.
\end{equation} 
%So $\R_1^{\Fs,\mathrm{sp}}$ is not finite (nor continuous) as soon 
%as $X$ is an infinite graph. 
%
%
%
%
%
%
%
%
% 
\if{

From  \eqref{(2)} and \eqref{eq : Def : R_1sp} we deduce that if 
$\rho_{\R^{\Fs}_1}(x)=r(x)$ 
(i.e. if $\xi\in\S(\FM)$), then 
\begin{equation}\label{FM=FGEN on S(FM)}
\R^{\Fs}_1(x)\;=\;
\R^{\Fs,\mathrm{sp}}_1(x)\;\leq\;
r(x)\;,
\end{equation}
and if the inequality is strict, then 
$\R_1^{\Fs}(x)=
\R^{\Fs,
\mathrm{sp}}(x)=\R^Y(x)$.

From \eqref{(2)} and \eqref{encadrement rho} one also obtains the 
following often useful expression of $\FM$:
\begin{equation}\label{FM=minFY,rhoFY}
\FM(\xi)\;=\;\min(\FY(\xi),\rho_{\FY}(\xi))\;.
\end{equation}
}\fi
\begin{proposition}[Concavity and transfer theorems]
\label{Prop : (C3) for R_1}
If $\xi_1(f)\leq\xi_2(f)$ for all $f\in\O(X)$, then 
\begin{equation}\label{transfer}
\R^Y(\xi_1)\;\geq\; \R^Y(\xi_2)\quad\textrm{ and }\quad
\R_1^{\Fs}(\xi_1)\;
\geq\; \R_1^{\Fs}(\xi_2)\;.
\end{equation} 
Moreover $\R^Y$ and $\R^{\Fs}_1$ satisfy property $(C3)$ of Section 
\ref{F_t defi} with respect to 
$\Gamma=\Gamma_X$. 
If $I\subseteq [0,R_0]$ is an interval with interior $\stackrel{\circ}{I}$ 
and if the open annulus $\{|T-t_\xi|\in \stackrel{\circ}{I}\}$ is 
contained in $X_{\H(x)}$, 
then $\R^{\Fs}$ and $\R(-,\Fs)$ are $\log$-concave on $I$.
\end{proposition}
\begin{proof}
All the claims for $\R^Y$ immediately follow from \eqref{eq : R^Y} which is 
$\liminf$ of super-harmonic functions (hence $\log$-concaves along $I$). 
For $\R_1^{\Fs}$, the claims follow from the equality 
$\R^{\Fs}_1(x)=\min(\R^Y(x),\rho_{x,X})$. More precisely \eqref{transfer} holds
since one has $\rho_{\xi_1,X}=\rho_{\xi_2,X}$ 
(cf. Remark \ref{rk : xleq x' then rhox=rhox'}). 
\end{proof}

\subsection{Spectral radius and spectral norm of the connection.}
\label{link between Fgen and spectral norm}
Let $(F,|.|_F)\in E(K)$ and let $V$ be a finite dimensional vector space. 
A norm $|.|_{V}$ on $V$ compatible with $|.|_F$ is a map 
$|.|_V:V\to\mathbb{R}_{\geq 0}$ such that 
(i) $|v|_V=0$ if and only if $v=0$;
(ii) $|v-v'|_V\leq\max(|v|_V,|v'|_V)$ for all $v,v'\in V$;
(iii) $|fv|_{V}=|f|_F\cdot|v|_{V}$ for all $f\in F$, $v\in V$.

If $T:V\to V$ is a bounded $\mathbb{Z}$-linear operator, 
we define the \emph{norm} and the \emph{spectral norm} of $T$ by 
\begin{equation}
|T|_V\;:=\;\sup_{v\neq 0}|T(v)|_V/|v|_V\;,\qquad
|T|_{Sp,V}\;:=\;\lim_s|T^s|_V^{1/s}\;.
\end{equation} 

One proves that the limit exists, and that 
$|T|_{Sp,V}$ only depends on $|.|_F$ and not on the choice of $|.|_V$ compatible with $|.|_F$ (cf. \cite[Def. 6.1.3]{Kedlaya-Book}). 

Let $\omega:=\lim_n|n!|^{1/n}$. 
If the restriction of $|.|$ to the sub-field of rational numbers $\mathbb{Q}$ is $p$-adic (resp. trivial), 
then $\omega=|p|^{\frac{1}{p-1}}$ (resp. $\omega=1$). 

If $\xi$ is not of type $1$, $(\H(\xi),\xi)=
(\mathscr{M}(X),\xi)^{\widehat{\phantom{aa}}}$\! 
is the completion of the fraction 
field $\mathscr{M}(X)$ of $\O(X)$
with respect to the norm $\xi$. The following lemma proves that the 
derivation $d/dT$ is continuous, and hence it 
extends by continuity to $\H(\xi)$. Recall that $K=\Ka$.

\begin{lemma}\label{Lemma : norm d = rho gen}
Let $\xi\in \mathbb{A}_K^{\mathrm{1,an}}$ be a point of type 
$2$, $3$, or $4$. The operator norm of $(d/dT)^n$ satisfies
\begin{equation}\label{eq : bound on d^n}
|(d/dT)^n|_{\H(\xi)}\;=\; 
\frac{|n!|}{r(\xi)^n}\;,\qquad |d/dT|_{Sp,\H(x)}\;=\;
\frac{\omega}{r(x)}\;.
\end{equation}
\end{lemma}
\begin{proof}
Let $t\in D(x)$ be a Dwork generic point for $x$ (cf. Section \ref{subsubs : t_x}). 
The Taylor expansion at $t\in X_\Omega$ 
gives an injective isometric map 
of $\H(\xi)$ into the ring $\mathcal{B}(D(x))$ 
of bounded functions over 
$D(x)=D^-(t,r(x))\subset X_\Omega$ commuting with $d/dT$.
The image of $f\in\H(x)$ is $\sum_{i\geq 0}f^{(i)}(t)(T-t)^i/i!$ 
and $x(f)=x_{t,0}(f_\Omega)=x_{t,r(x)}(f_\Omega)=
\sup_{i\geq 0}|f^{(i)}(t)/i!|\cdot r(x)^i$. 
It is well known that $|(d/dT)^n|_{\mathcal{B}(D(x))}=
\frac{|n!|}{r(x)^n}$, and this implies $|(d/dT)^n|_{\H(x)}\leq 
\frac{|n!|}{r(x)^n}$.

Now for all $c\in K$ one has $|n!|=|(d/dT)^n(T-c)^n|(x)
\leq|(d/dT)^n|_{\H(x)}|T-c|(x)^n$. Hence we find 
$|(d/dT)^n|_{\H(x)}\geq \sup_{c\in K}\frac{|n!|}{|t-c|^n}=
\frac{|n!|}{r(x)^n}$, by Lemma \ref{rho gen as inf Kbar} (because 
$K=\Ka$).
\end{proof}

\begin{proposition}\label{Prop. : Rsp}
Let $x\in \mathbb{A}^{1,\mathrm{an}}_K$ be a point of type $2$, 
$3$, or $4$. 
Let $(\Fs,\nabla)$ be a differential module over $\H(x)$ 
endowed with a norm compatible with $|.|(x)$.
Then
\begin{equation}\label{Def. : RMsp}
\omega\cdot
|\nabla|^{-1}_{Sp,\Fs}\;=\;\R_1^{\Fs,\mathrm{sp}}(x)\;.
\end{equation}
\end{proposition}
\begin{proof}
A direct computation gives (cf. \cite[Prop.1.3]{Ch-Dw}, 
\cite[Lemma 6.2.5]{Kedlaya-Book})
\begin{equation}\label{Fgen = norm spectral}
|\nabla|_{Sp,\Fs}\;=\;\max(\limsup_n|G_n|(\xi)^{1/n},
|d/dT|_{Sp,\H(x)})\;,
\end{equation}
where $G_n$ is the matrix of  \eqref{eq : R^Y}.
By Lemma \ref{Lemma : norm d = rho gen}, we have 
$|d/dT|_{Sp,\H(x)}=\omega/r(x)$.
\end{proof}

\begin{remark}
If $K$ is not algebraically closed we still have the equalities 
$|d/dT|_{Sp,\H(x)}=
\omega/r(x)$ and $|n!|/r(x,K)^n\leq|(d/dT)^n|_{\H(x)}\leq 
|n!|/r(x)^n$, where $r(x,K):=\min_{c\in X(K)}|t-x-c|_{\Omega}$. 
The proof in this case is more involved, and unnecessary for our 
purposes.
\end{remark}

\subsection{Spectral Newton polygon of a differential module.} 
\label{NP of a module}
Let $x\in X$ be a point of type $2$, $3$, or $4$. 
By Proposition \ref{Prop. : Rsp} it follows that 
$\R_1^{\Fs,\mathrm{sp}}(x)$ only depends on the restriction of $\Fs$ 
to the differential field $(\H(x),d/dT)$. 
We now define higher spectral radii following \cite{Kedlaya-Book}.
Let $\Fs$ be a differential module of rank $r$ over $(\H(\xi),d/dT)$. 
Let 
\begin{equation}
0\;=\;
\M_{0}\;\subset\;\M_{1}\;\subset\;
\cdots\;\subset\;\M_{n}\;=\;\Fs
\end{equation}
be a Jordan-Hölder sequence of $\Fs$. 
This means that for all $k$, $\N_{k}:=\M_{k}/\M_{k-1}$ has no non 
trivial strict differential sub-modules. 

Let $r_k$ be the rank of $\N_{k}$, and let 
$R_{k}:=\R_1^{\N_{k},\mathrm{sp}}(x)$.  
Perform a permutation of the indexes in order to have 
$R_{1}\leq\ldots\leq R_{n}$.
Let $s^{\Fs,\mathrm{sp}}(x):s^{\Fs,\mathrm{sp}}_{1}(x)\leq 
\ldots\leq s^{\Fs,\mathrm{sp}}_{r}(x)$ be the slope 
sequence obtained from 
$\ln(R_{1})\leq\ldots\leq \ln(R_{n})$ by counting the slope $\ln(R_{k})$ with multiplicity $r_k$:
\begin{equation}\label{R_1,...,R_n}
s^{\Fs,\mathrm{sp}}(\xi)\;\;:\;\;\underbrace{\ln(R_{1})=\ldots=\ln(R_{1})}_{r_1-\textrm{times}}\leq\underbrace{\ln(R_{2})=\ldots=
\ln(R_{2})}_{r_2-\textrm{times}}\leq\cdots\leq\underbrace{\ln(R_{n})=\ldots=\ln(R_{n})}_{r_n-\textrm{times}}\;.
\end{equation}
We set $\R_i^{\Fs,\mathrm{sp}}(x):=
\exp(s_i^{\Fs,\mathrm{sp}}(x))$.
For all $i=1,\ldots,r$ set $h^{\Fs,\mathrm{sp}}_0(x)=0$ 
and $h^{\Fs,
\mathrm{sp}}_i(x):=s^{\Fs,\mathrm{sp}}_1(x)+\cdots+s^{\Fs,
\mathrm{sp}}_i(x)$. We call \emph{spectral Newton polygon} the 
polygon $NP^{\mathrm{sp}}(x,\Fs):=NP(h^{\Fs,\mathrm{sp}}(x))$.

If $\Fs$ is a differential equation over $X$, as for 
$\R_1^{\Fs,\mathrm{sp}}(x)$ (cf. Def. \eqref{Def. : RMsp}), we 
extend the definition of $\R_i^{\Fs,\mathrm{sp}}$ to 
the whole $X$ by setting $\R_i^{\Fs,\mathrm{sp}}(x)=0$, for all $x$ 
of type $1$.

\begin{theorem}[\protect{\cite[Thm.11.3.2, Remarks 11.3.4, 11.6.5]{Kedlaya-Book}}]
\label{thm: continuity of spectral radii along a branch}
Let $\Fs$ be a differential module over $X$. 
Let $I\subseteq X$ be an (open/closed/semi-open) segment. The 
following properties hold:\footnote{
The claim of \protect{\cite[11.3.2]{Kedlaya-Book}} is given for 
segments free of points of type $4$, but around a point $x$ of type 
$4$ we can extend the ground field $K$ 
to turn $x$ into a point $\sigma_{\Omega/K}(x)$ of type $2$, and use 
Remark \ref{eq : dtdcgklb brf} to replace $I$ by 
$\sigma_{\Omega/K}(I)$.
Also the claim of \cite{Kedlaya-Book} is given for $I$ being the skeleton of an annulus in $X$. 
The claim however holds for \emph{the closure} of the skeleton of an open annulus, if the matrix $G$ of $\nabla$ 
has bounded coefficients on the annulus. This gives our claims by considering a subdivision of $I$ by segments 
whose interiors are skeletons of open annuli in $X$.} 
\begin{enumerate}
\item The functions $\R_{i}^{\Fs,\mathrm{sp}}$ and 
$H_{i}^{\Fs,\mathrm{sp}}$ 
verify properties (C2) and (C4) of section \ref{Main result} 
along $I$. 

If moreover $I$ is 
the skeleton of a virtual annulus in $X$, then 
$H_{i}^{\Fs,\mathrm{sp}}$ is $\log$-concave on 
along $I$.
\item  Points (ii.a) and (ii.b) of Theorem 
\ref{Theorem : MAIN THM GEN} hold replacing 
$H_i^{\Fs}$ by  $H_i^{\Fs,\mathrm{sp}}$.

\item Assume that $I$ is the skeleton of a virtual \emph{open} annulus 
in $X$. Assume that $i\leq r$  is a spectral non solvable index 
at $x\in I$. Then  
$\partial_bH_{i}^{\Fs,\mathrm{sp}}(x)=0$ for almost, but a finite 
number of germs of segments $b$ out of $x$, and
$H_{i}^{\Fs,\mathrm{sp}}$ is super-harmonic at $x$. 
Moreover if $i$ is a vertex at $x$ 
of the spectral Newton polygon $NP^{\mathrm{sp}}(x,\Fs)$, 
then $H_{i}^{\Fs,\mathrm{sp}}$ is harmonic at $x$.

\item If $I$ is contained in an open virtual disk $D\subset X$, and if 
the index $i$ is spectral non solvable at $x\in I$, then 
$H_{i}^{\Fs,\mathrm{sp}}$ is non increasing along an open 
sub-segment $J$ of $I$ containing $x$.\hfill$\Box$
\end{enumerate}
\end{theorem}
\if{\begin{proof}
 The claim of 
\protect{\cite[11.3.2]{Kedlaya-Book}} is given for 
segments free of points of type $4$, but around a point $x$ of type 
$4$ we can extend the ground field $K$ 
to turn $x$ into a point $\sigma_{\Omega/K}(x)$ of type $2$, and use 
Remark \ref{eq : dtdcgklb brf} to replace $I$ by 
$\sigma_{\Omega/K}(I)$.
Also the claim of \cite{Kedlaya-Book} is given for $I$ being the skeleton of an annulus in $X$. 
The claim however holds for \emph{the closure} of the skeleton of an open annulus if the matrix $G$ of $\nabla$ 
has bounded coefficients on the annulus. This gives our claim. 
\end{proof}
}\fi
\if{\begin{proof}
The spectral Newton polygon of $\M$ only depend on its restriction 
$\M(\xi)$ to $\H(\xi)$ 
(cf. section 
\ref{Bounded functions, analytic functions and analytic elements.}).
We can assume that $X$ equals an annulus 
$\{|T|\in[\rho_1,\rho_{2}]\}$ having possibly some holes placed at 
distance $\rho_1$ and $\rho_{2}$ from $0$. 
In this situation let $\mathcal{H}_{K}(0,]\rho_1,\rho_{2}[)$ be the ring of 
\emph{analytic elements} on the open annulus 
$\{|T|\in ]\rho_1,\rho_{2}[\}$. 
For all $\rho\in[\rho_1,\rho_{2}]$ the morphism 
$\O(X)\to\H(\xi_{0,\rho})$ factorizes through the natural restriction 
map $\O(X)\to\mathcal{H}_K(0,]\rho_1,\rho_{2}[)\to
\H(\xi_{0,\rho})$. 
So we can assume $\M$ to be a differential module over 
$\mathcal{H}_K(0,]\rho_1,\rho_{2}[)$.
The assertions are hence proved in \cite[Thm. 11.3.2]{Kedlaya-Book}. 

\if{
The last assertion is due to the fact that \cite{Kedlaya-Book} has two 
different definitions chosen ah hoc to adjust the difference between $\R_i^{\M}$ and $\R_i^{\M',\mathrm{sp}}$ : for annuli one has 
$\R^{\M',\mathrm{sp}}_i$ (cf. \cite[Def. 9.4.4]{Kedlaya-Book}), and for disks one has $\widetilde{\R}^{\M',\mathrm{sp}}_i$ 
(cf. \cite[Def. 9.3.1]{Kedlaya-Book}). 
}\fi
Notes: the super-harmonicity is stated in \cite{Kedlaya-Book} for a point of type $(2)$ i.e. such that $\Delta(\xi)$ has at least two elements. 
If $\Delta(\xi)$ has a unique element, then the super-harmonicity is just the $\log$-concavity which is the claim i).
directional finiteness is not mentioned in \cite{Kedlaya-Book}, but from the proof of \cite[Thm. 11.3.2]{Kedlaya-Book} 
it follows that the $\log$-slopes 
$\partial_-H_{i,\delta}^{\M,\mathrm{sp}}(\xi_{t,\overline{\rho}})$ are those of a convenient differential polynomial with coefficients in the 
fraction field $\mathscr{F}(X)$ of $\O(X)$ so the directional finiteness (C5) holds applying Proposition \ref{Prop. NP of a operator}, after possibly 
replacing $\M$ with  its restriction to a sub-affinoid $X'$ of $X$ \emph{preserving the directions} at $\xi_{t,\overline{\rho}}$ 
(cf. Def. \ref{Def. preserves the directions}) 
and such that the coefficients of the polynomial have all no poles in $X'$ (cf. section \ref{Reduction to a cyclic}). 
Finally notice that the harmonicity statement of \cite[11.3.2,(c)]{Kedlaya-Book} implicitly assumes $K=\widehat{K^{\mathrm{alg}}}$ 
because Lemma \ref{Lemma : norm d = rho gen} is used to control the slopes along the complete segments $\Lambda(t)$ 
defined by algebraic $t\in K^{\mathrm{alg}}$.
\end{proof}
}\fi
\begin{proposition}\label{Prop : eeetrich}
For all $i=1,\ldots,r$ we have 
$\R^{\Fs,\mathrm{sp}}_i(x)=\min(\R^{\Fs}_i(x),r(x))$.
In particular $\R_i^{\Fs,\mathrm{sp}}=\R_i^{\Fs}$ along 
$\Gamma(\R_i^{\Fs})$ by Remark 
\ref{Rk : radius is spectral over controlling graph}.
\end{proposition}
\begin{proof}
If $x$ is a point of type $1$, there is nothing to prove. 
If $x$ is a point of type $2$, $3$, or $4$, $\Fs$ admits 
a decomposition separating the spectral radii $\{\R_i^{\Fs,\mathrm{sp}}(x)\}_i$
(cf. \cite{Robba-Hensel} or \cite[10.6.2]{Kedlaya-Book}). 
Now there is also a decomposition of $\Fs$ 
separating the radii of convergence of the Taylor solutions 
at a Dwork generic point $t$ of $x$ that are smaller than or equal 
to $r(x)$ (cf. \cite{RoI}).\footnote{The classical 
proofs of these decomposition results are given for a point 
of type $2$, but they extends smoothly to all points of type 
$2$, $3$, or $4$ (cf. \cite{NP-III} for more details). The key 
ingredient is the fact that $|d/dT|_{Sp,\H(x)}=\omega/r(x)$.} 
Both these decompositions behave well by exact sequences.
So we can assume that $\Fs$ verifies 
$\R^{\Fs,\mathrm{sp}}_1(x)=\cdots=\R^{\Fs,\mathrm{sp}}_r(x)$, 
and that its Taylor solutions at $t$ all have the same radius of 
convergence.
The claim then follows from the case $i=1$ 
(cf. \eqref{eq : Def : R_1sp} plus \eqref{Def. : RMsp}).
\end{proof}
\begin{remark}
\label{Rk : (C2)+(C4)}
\emph{(1)} By Remark \ref{Rk : radius is spectral over controlling graph}
for all $\rho\geq \rho_{x,X}$ we have
$(\R_i^{\Fs,\mathrm{sp}}\circ\lambda_{x})(\rho)=
(\R_i^{\Fs}\circ\lambda_{x})(\rho)$.

In general, for all $\rho\geq 0$ we have 
$r(\lambda_{x}(\rho))=\max(\rho,r(x))$, 
so Proposition \ref{Prop : eeetrich} gives
\begin{equation}\label{eq : Comparison between Rad and Radgen}
(\R_i^{\Fs,\mathrm{sp}}\circ\lambda_{x})(\rho)\;=\;
\min\bigl(\;\max(r(x), \rho )\;,\; 
(\R_i^{\Fs}\circ\lambda_{\xi})(\rho)\; \bigr)\;.
\end{equation}

\emph{(2)} In the case $i=1$, 
$\rho\mapsto(\R_1^{\Fs}\circ\lambda_{x})(\rho)$ 
is moreover $\log$-concave and 
$\log$-decreasing for $\rho\in[0,\rho_{x,X}]$, and 
since $\rho\mapsto\max(r(x),\rho)$ is $\log$-convex for all 
$\rho\in[0,R_0]$, then 
\begin{enumerate}
\item If $\R_1^{\Fs}(x)\leq r(x)$, then 
$(\R_1^{\Fs,\mathrm{sp}}\circ\lambda_{x})(\rho)=
(\R_1^{\Fs}\circ\lambda_{x})(\rho)$ 
for all $\rho\in [0,R_0]$;
\item If $\R_1^{\Fs}(x)>r(x)$, then $(\R_1^{\Fs,\mathrm{sp}}\circ\lambda_{x})(\rho)=
(\R_1^{\Fs}\circ\lambda_{x})(\rho)$ for all $\rho\in[\R_1^{\Fs}(x),R_0]$.
\end{enumerate}
In particular, if $(\R_1^{\Fs,\mathrm{sp}}\circ\lambda_{x})(\rho)=
(\R_1^{\Fs}\circ\lambda_{x})(\rho)$ for some $\rho$, 
the same equality holds for all $\rho'\geq \rho$.

\emph{(3)} The index $i$ is spectral at $\lambda_x(\rho)$, 
for all $\rho\geq\R_i^{\Fs}(x)$. Indeed, by point (6) of Remark \ref{Def.: radius on a disk}, 
if $\R_i^{\Fs}(y)>r(y)$ for some $y=\lambda_x(\rho)$, 
then $D^-(t_x,\R_i^{\Fs}(x))=D^-(t_y,\R_i^{\Fs}(y))$, 
so $\R_i^{\Fs}(x)<\rho$.

So for all $x\in X$, and all $\rho\geq 0$, we have
\begin{equation}\label{eq : R_i VS R_isp along a segment}
(\R_{i}^{\Fs}\circ\lambda_x)(\rho)\;=\;
\left\{
\sm{
\R^{\Fs}_i(x)&\textrm{ if }&\;\;\;\rho\;\in\;[0,\R^{\Fs}_i(x)]\\
(\R_{i}^{\Fs,\mathrm{sp}}\circ\lambda_x)(\rho)&
\textrm{ if }&\rho\;\geq\;\R^{\Fs}_i(x)\;.
}
\right.
\end{equation} 
So $\R^{\Fs}_i\circ\lambda_x$ and 
$\R^{\Fs,\mathrm{sp}}_i\circ\lambda_x$ 
differ by at most two slopes over $[0,\R_i^{\Fs}(x)]$ (the slopes of $\max(r(x), \rho )$) 
that can only be $0$ or $1$ for both radii. 

\emph{(4)} This shows that Theorem \ref{thm: continuity of spectral radii along a branch} 
implies (C2) and (C4) for $\R^{\Fs}_i$, and hence  also for $\R_i(-,\Fs)$, 
it implies moreover points ii) and iii) of Theorem \ref{Theorem : MAIN THM GEN}.
\end{remark}

\begin{remark}\label{eq : dtdcgklb brf}
The functions $\R_i^{\Fs_\Omega,\mathrm{sp}}$ 
are not constant over the fiber $\pi_{\Omega/K}^{-1}(x)$, while 
$\R^{\Fs}_i$ is constant on it. It follows from Proposition 
\ref{Prop : eeetrich} that for all 
$y\in \pi_{\Omega/K}^{-1}(x)$ we have 
$\R^{\Fs_\Omega,\mathrm{sp}}_i
(y)=\min(\R^{\Fs,\mathrm{sp}}_i(x),r_\Omega(y))$. In particular
$\R^{\Fs,\mathrm{sp}}_i(x)=
\R^{\Fs_\Omega,\mathrm{sp}}_i(\sigma_{\Omega/K}(x))$.
\end{remark}
We also recall the following fundamental result:
\begin{theorem}[\protect{\cite[12.4.1]{Kedlaya-Book}}]
\label{deco on a disc}
Let $\Fs$ be a differential equation of rank $r$ over a disk 
$\O(D^-(c,\rho))$, $c\in K$, $\rho>0$. 
Assume that for some $i\leq r$ there exists $\varepsilon>0$ 
such that 
$h_{i-1}^{\Fs,\mathrm{sp}}$ is constant along 
$]x_{c,\rho-\varepsilon},x_{c,\rho}[$, and moreover 
$s_{i-1}^{\Fs,\mathrm{sp}}(x_{c,\rho'})<
s_{i}^{\Fs,\mathrm{sp}}(x_{c,\rho'})$, $\forall$ 
$\rho'\in]\rho-\varepsilon,\rho[$. 
Then $\Fs=\Fs_{\geq i}\oplus\Fs_{<i}$, where: 
\begin{enumerate}
\item The ranks of $\Fs_{<i}$ and $\Fs_{\geq i}$ are $i-1$ and 
$r-i+1$ respectively;
\item For all $k=1,\ldots,i-1$ one has 
$s_{k}^{\Fs_{<i},\mathrm{sp}}(x_{c,\rho'})=
s_{k}^{\Fs,\mathrm{sp}}(x_{c,\rho'})$ for all 
$\rho'\in]\rho-\varepsilon,\rho[$;
\item For all $k=i,\ldots,r$ one has 
$s_{k-i+1}^{\Fs_{\geq i},\mathrm{sp}}(x_{c,\rho'})=
s_{k}^{\Fs,\mathrm{sp}}(x_{c,\rho'})$, 
for all $\rho'\in]\rho-\varepsilon,\rho[$.\hfill $\Box$
\end{enumerate}
\end{theorem}

\subsection{Spectral Newton polygon of a differential operator.} 
\label{Spectral Newton polygon of a differential operator.}
Let $\L:=\sum_{i=0}^rg_{r-i}(T)\cdot(d/dT)^i$ 
be a differential operator 
with $g_0=1$ and $g_i\in \O(X)$. We set 
%$v^{\L,\mathrm{sp}}_i:= -\ln(\omega^{-i}\cdot|g_i(x)|).$
\begin{equation}
v^{\L,\mathrm{sp}}_i\;:=\; -\ln(\omega^{-i}\cdot|g_i(x)|)\;.
\end{equation}
We define the \emph{spectral Newton polygon of $\L$} as $NP(\L,x):=NP(v^{\L,\mathrm{sp}})$. 

Let $s^{\L,\mathrm{sp}}(x):s_1^{\L,\mathrm{sp}}(x)\leq\ldots
\leq s_r^{\L,\mathrm{sp}}(x)$ be its slope sequence. For $i=1$ we have
\begin{equation}\label{FIRST-SLOPE of oper}
s_1^{\L,\mathrm{sp}}(x)\;=\;
\ln\Bigl(\omega\cdot\min_{i=1,\ldots,r}
|g_i(x)|^{-\frac{1}{i}}\Bigr)\;.
\end{equation}
We define as usual 
$\R_i^{\L,\mathrm{sp}}(x):=\exp(s_i^{\L,\mathrm{sp}}(x))$ and 
$H^{\L,\mathrm{sp}}_i(x):=\exp(h_i^{\L,\mathrm{sp}}(x))$. 
%Let (cf. \cite[1.5.4]{BGR})
%\begin{equation}
%|\L|_{Sp,\xi}\;:=\;\max_{1\leq i\leq r}|g_i(\xi)|^{\frac{1}{i}}\;.
%\end{equation} 
\begin{proposition}[Small radii, cf. 
\protect{\cite{Young}, 
\cite[Section 6]{Kedlaya-Book}, \cite[Thm.6.2]{Astx}}]
\label{Prop. : small radius}
\label{Corollary : slopes Rsp = slopes Lsp}
Let $x\in X$ be a point of type $2$, $3$, or $4$, and 
let $(\Fs,\nabla)$ be the differential module over $(\H(x),d/dT)$
attached to $\L$. Then 
\if{
$|\L|_{Sp,x}> r(x)^{-1}$ if and only if 
$|\nabla|_{Sp,\Fs}>r(x)^{-1}$.
%\begin{equation}
%|\L|_{Sp,x}\;>\; r(x)^{-1}\;\quad\Longleftrightarrow\quad
%|\nabla|_{Sp,\Fs}\;>\;r(x)^{-1}\;.
%\end{equation}
If these equivalent conditions hold one has 
$|\L|_{Sp,x}=|\nabla|_{Sp,\H(x)}=\omega/ \R_1^{\Fs,
\mathrm{sp}}(x)$.
%\begin{equation}
%\qquad\qquad
%|\L|_{Sp,x}\;=\;|\nabla|_{Sp,\H(x)}\;=\;\omega/ \R^{\M,
%\mathrm{sp}}(x)\;.
%\end{equation}
More generally
}\fi 
$\exp(s_i^{\L,\mathrm{sp}}(x))< \omega\cdot r(x)$ 
if and only if $\R_i^{\Fs,\mathrm{sp}}(x)< \omega\cdot r(x)$, and in 
this case we have 
%$\R_i^{\Fs,\mathrm{sp}}(x)=\exp(s_i^{\L,\mathrm{sp}}(x))$.
%\hfill$\Box$
\begin{equation}
\qquad\qquad\qquad\qquad\qquad\qquad\qquad
\R_i^{\Fs,\mathrm{sp}}(x)\;=\;\exp(s_i^{\L,\mathrm{sp}}(x))\;.
\qquad\qquad\qquad\qquad\qquad\Box
\end{equation}
\end{proposition}

\begin{remark}
If  $g_r\neq 0$, then 
$s_i^{\L,\mathrm{sp}}(\xi),h_i^{\L,\mathrm{sp}}(\xi)<\infty$. 
This will be the case of major interest, indeed 
the case where $g_r=0$ reduces to a lower degree, since 
we have a factorization $\L=\L_1\cdot(d/dT)$.
\end{remark}

\begin{proposition}\label{Prop. NP of a operator}
Assume that $g_0=1$, $g_r\neq 0$, and that 
%$g_1,\ldots,g_r$ lies in $\O(X)$ and that 
for all $i$, the function $g_i$
is either equal to $0$, or it has no zeros on $X$. Then :
\begin{enumerate}
\item For all $i=0,\ldots,r$ the function $\xi\mapsto 
H^{\L,\mathrm{sp}}_i(\xi)\in\mathbb{R}$ 
verifies the six properties (C1)--(C6) with respect to $\S:=\S_X$ and 
$\C(H_i^{\L,\mathrm{sp}}):=\partial X$. 
It is hence finite by Theorem \ref{th : finiteness theorem};  
\item For all $i=0,\ldots,r$ one has $\S(h_i^{\L,\mathrm{sp}})=
\Gamma(H_i^{\L,\mathrm{sp}})=
\Gamma(s_i^{\L,\mathrm{sp}})=
\Gamma(\R_i^{\L,\mathrm{sp}})=
\Gamma_X$;
%\begin{equation}
%\S(h_i^{\L,\mathrm{sp}})\;=\;
%\S(H_i^{\L,\mathrm{sp}})\;=\;
%\S(s_i^{\L,\mathrm{sp}})\;=\;
%\S(\R_i^{\L,\mathrm{sp}})\;=\;
%\S_X\;;
%\end{equation}
\item Assume that $\xi \in X$ is a point of type $2$, $3$, or $4$, 
and that $i$ is a vertex of 
$NP(\L,\xi)$ (i.e. $i=r$ or 
$s^{\L,\mathrm{sp}}_i(\xi)<s^{\L,\mathrm{sp}}_{i+1}(\xi)$). 
Then $\partial_bH_{i}^{\L,\mathrm{sp}}(x)\in\mathbb{Z}$, 
and $H_i^{\L,\mathrm{sp}}$ is harmonic outside $\partial X$.
\if{
Then: 
\begin{itemize}
\item[$(\mathrm{iii}$-$\mathrm{a})$] Over all germ of segment 
$b:=[x,y[$ out of $x$ one has
$H_i^{\L,\mathrm{sp}}=\omega^{i}|g_i|^{-1}$.
\item[$\mathrm{(iii}$-$\mathrm{b})$] 
$\partial_bH_{i}^{\L,\mathrm{sp}}(x)=
\partial_b(x\mapsto |g_i(\xi)|^{-1})\in\mathbb{Z}$. 
\item[$\mathrm{(iii}$-$\mathrm{c})$] If $x\in X-\partial X$,
then $H_i^{\L,\mathrm{sp}}$ is harmonic at $\xi$. 
\end{itemize}  
}\fi
\item For all $i=1,\ldots,r$ the slopes of $h_i^{\L,\mathrm{sp}}$ and $s_i^{\L,\mathrm{sp}}$ belong to 
$\mathbb{Z}\cup\frac{1}{2}\mathbb{Z}\cup\cdots\cup\frac{1}{r}\mathbb{Z}$.
\end{enumerate}
\end{proposition}
\begin{proof}
Since every $g_i$ has no zeros on $X$ the functions 
$\xi\mapsto|g_i(\xi)|$ are constant on every maximal disk 
$\mathrm{D}(x,X)$. 
Hence ii) holds. If $i$ is a vertex, then
\begin{itemize}
\item[$(\mathrm{iii}$-$\mathrm{a})$] Over all germ of segment 
$b:=[x,y[$ out of $x$ one has
$H_i^{\L,\mathrm{sp}}=\omega^{i}|g_i|^{-1}$.
\item[$\mathrm{(iii}$-$\mathrm{b})$] 
$\partial_bH_{i}^{\L,\mathrm{sp}}(x)=
\partial_b(x\mapsto |g_i(\xi)|^{-1})\in\mathbb{Z}$. 
\item[$\mathrm{(iii}$-$\mathrm{c})$] If $x\in X-\partial X$,
then $H_i^{\L,\mathrm{sp}}$ is harmonic at $\xi$. 
\end{itemize} 
The rest is straightforward 
(see for example \cite[Thm.11.2.1]{Kedlaya-Book}).
Namely iv) is deduced by iii), by interpolation. 
This means that if $i_1< i< i_2$ are the vertexes of the polygon 
that are closest to $i$ at $x$, we define 
$G(y):=h_{i_1}^{\L,\mathrm{sp}}(y)+(i-i_1)
\frac{h^{\L,\mathrm{sp}}_{i_2}(y)-
h^{\L,\mathrm{sp}}_{i_1}(y)}{(i_2-i_1)}$. 
Then $G$ is super-harmonic at $x$, and 
$G\geq h^{\L,\mathrm{sp}}_i$ around $x$. So 
$h_i^{\L,\mathrm{sp}}$ is super-harmonic by Lemma 
\ref{Lemma: G dominates F, so F super-H}.
\end{proof}
\begin{remark}
To deal with the case in which  some $g_i$ is not invertible, 
it is enough to replace $X$ by a sub-affinoid on which each 
$g_i$ is either zero, or it is invertible.
\end{remark}
\begin{remark}\label{trspNP}
Let $s_i'(x):=\min(s_i^{\L,\mathrm{sp}}(x),\ln(\omega\cdot r(x)))$
be the truncated slope sequence. 

The partial heights of the corresponding polygon verify (C2), (C4),  and 
(C3) with respect to $\Gamma:=\Gamma_X$. 
Of course (as for $\R^{\Fs,\mathrm{sp}}_1(x)$) the constancy skeleton 
of each partial height is equal to $X$. 

If $i$ is a vertex of the truncated polygon, 
then the slopes of the $i$-th partial height belong to $\mathbb{Z}$, 
and property $iv)$ of Proposition \ref{Prop. NP of a operator} holds. 
While (iii-c) and super-harmonicity only hold for $i$-th partial heights 
corresponding to indexes $i$ satisfying $s_{i}'(x)<\ln(\omega\cdot r(x))$.
\end{remark}
\if{
\begin{proposition} \label{trspNP}
We maintain the assumptions of Proposition 
\ref{Prop. NP of a operator}.
Let $\alpha>0$, be a constant and let 
$C(\xi):=\ln(\alpha\cdot r(\xi))\in \{-\infty\}\cup\mathbb{R}$. 
Let 
$s_i'(\xi):=\min(s_i^{\L,\mathrm{sp}}(\xi),C(\xi))$, 
$i=1,\ldots,r$ be the truncated slope sequence. 
In order to avoid to work with $-\infty$ let 
$R_i'(x):=\exp(s_i'(x))$, and let $H'_i(\xi):=\exp(h_i'(\xi))$ 
where as usual $h_i'(\xi):=s_1'(\xi)+\cdots+s_i'(\xi)$, and if $t\in X(\widehat{K^{\mathrm{alg}}})$ we extend the above definition by 
$R_i'(\xi_t):=0$ and $H_i'(\xi_t)=0$. Then :
\begin{enumerate}
\item For all $i=0,\ldots,r$ the function 
$\xi\mapsto H'_i(\xi)\in\mathbb{R}$ 
verifies (C2),(C4), and (C3) with $\S=\S_X$;  
\item  $\S(R_i')=\S(H_i')=X$, because $R_i'(\xi_t)=0$ for all 
$t\in X( K^{\mathrm{alg}})$;
\item Assume $\xi\in X_{\mathrm{int}}$. If $(i,h_i'(\xi))$ is a vertex of the truncated Newton polygon at $\xi$ (i.e. $i=r$ or $s_i'(\xi)<s_{i+1}'(\xi)$) 
then for all $\delta\in\Delta(\xi)$ the slopes $\partial_-H'_{i,\delta}(\xi)$ and $\partial_+H_{i}'(\xi)$ lies in $\mathbb{Z}$. 
This implies, by interpolation, that for all $i=1,\ldots,r$ the $\log$-slopes of $H_i'$ always belong to 
$\mathbb{Z}\cup\frac{1}{2}\mathbb{Z}\cup\cdots\cup\frac{1}{r}\mathbb{Z}$.
\item Assume that $\xi \in X_{\mathrm{int}}$ does not belong to the 
Shilov boundary of $X$. Let $i_0\in\{1,\ldots,r\}$ be the largest integer 
such that $s_{i_0}'(\xi)<C(\xi)$. 
If $K=\widehat{K^{\mathrm{alg}}}$, then for all $i=1,\ldots,i_0$, the 
function $H_i'$ is super-harmonic at $\xi$, and if moreover 
$(i,h_i'(\xi))$ 
is a vertex of the trucated Newton polygon as in iii), then $H_i'$ is 
harmonic at $\xi$. \hfill$\Box$
\end{enumerate}
\end{proposition}
}\fi

\subsection{Localization to a sub-affinoid}
\label{Restriction to a sub-affinoid}
Let $\Fs$ be a differential module over $X$, 
and let $X'\subseteq X$ be a sub-affinoid domain. 
The polygon $NP^{\mathrm{sp}}(x,\Fs)$ only depends on the 
restricted module $\Fs(x)=\Fs\widehat{\otimes}\H(x)$, so it is 
invariant by restriction to $X'$. 
Conversely $NP^{\mathrm{conv}}(x,\Fs)$ is not: the radii change by 
localization. Hence the following proposition is not a direct 
consequence of Remark \ref{Rk : restr to X'}. 
The claim is given for the function $\R^{\Fs}_i$, and an immediate 
translation gives the analogous statement for $\R_i(-,\Fs)$ 
(cf. Remark \ref{Def.: radius on a disk}).

\begin{proposition}\label{Prop. restriction to a sub-affinoid}
Let $X'\subseteq X$ be a sub-affinoid. Then
\begin{enumerate}
\item For all $i=1,\ldots,r$ and all $\xi'\in X'$ one has 
$\R_i^{\Fs_{|X'}}(x')=
\min(\;\R_i^{\Fs}(x')\;,\;\rho_{x',X'}\;)$, 
and
\begin{equation}\label{eq : sk X sk X'}
\Gamma(X',\R_i^{\Fs_{|X'}})\;=\; 
\Bigl(\Gamma(X,\R_i^{\Fs})\bigcap X'\Bigr)\bigcup
\Gamma_{X'}\;.
\end{equation}
%$\R_i(x',\Fs_{|X'})=
%\min(\;\frac{\rho_{x',X}}{\rho_{x',X'}}\cdot\R_i(x',\Fs)\;,\;1\;)$, 
%and
%\begin{equation}\label{eq : sk X sk X'}
%\Gamma(X',\R_i(-,\Fs_{|X'}))\;=\; 
%\Bigl(X'\bigcap \Gamma(X,\R_i(-,\Fs))\Bigr)\bigcup
%\Gamma_{X'}\;.
%\end{equation}
\item $\R^{\Fs}_i$ is directionally finite at $x'\in X'$ (cf. (C5)) if and 
only if $\R_i^{\Fs_{|X'}}$ is directionally finite at $x'$.
\item If $\Gamma_{X'}\subseteq\Gamma(X,\R_i^{\Fs})$, then for all 
$x'\in X'$ one has 
\begin{equation}
\R_i^{\Fs_{|X'}}(x')\;=\;\R_i^{\Fs}(x')\;,\quad \textrm{ and }\quad  
H_i^{\Fs_{|X'}}(x')\;=\;H_i^{\Fs}(x')\;.
\end{equation}
In particular, if $X'$ is an affinoid neighborhood of $x'$ in $X$, 
then $H_i^{\Fs}$ is super-harmonic (resp. harmonic) at $x'$ 
if and only if so is $H_i^{\Fs_{|X'}}$.

\item 
Assume that $X'$ is an affinoid neighborhood of $x'$ in $X'$, and that 
$\R^{\Fs}_i(x')<\rho_{x',X'}$. Then for all $j=1,\ldots,i$ and all 
$b\in\Delta(\xi')$ one has 
$\partial_b\R^{\Fs}_{j}(x')=\partial_b\R_{j}^{\Fs_{|X'}}(x')$, and 
$\partial_bH^{\Fs}_{j}(x')=\partial_bH_{j}^{\Fs_{|X'}}(x')$.
\if{
\begin{equation}
\partial_b\R^{\Fs}_{j}(x')=\partial_b\R_{j}^{\Fs_{|X'}}(x')\;,\quad 
\textrm{ and }\quad\partial_bH^{\Fs}_{j}(x')=
\partial_bH_{j}^{\Fs_{|X'}}(x')\;.
\end{equation}
}\fi
Hence $H_j^{\Fs}$ is super-harmonic (resp. harmonic) at $x'$ 
if and only if so is $H_j^{\Fs_{|X'}}$.
\end{enumerate}
\end{proposition}
\begin{proof}
i)+ii). The relation 
$\R_i^{\Fs_{|X'}}(x')=\min(\R^{\Fs}_i(x'),\rho_{x',X'})$ follows 
from Def. \ref{Def. CRAD}. This, together with \eqref{(2)}, 
%point (6) of Remark \ref{Def.: radius on a disk}, 
gives $\rho_{\R^{\Fs_{|X'}}_i}(x')=\min(\rho_{\R^{\Fs}_i}(x'),
\rho_{x',X'})$. This implies \eqref{eq : sk X sk X'}, and hence  ii) 
follows.

iii). Assume that $\Gamma_{X'}\subseteq\Gamma(X,\R_i^{\Fs})$, 
then by point iii) of Prop. \ref{prof propp}, for all $j\leq i$ we have 
\begin{equation}
\R^{\Fs}_j(x')\;\leq\;\R^{\Fs}_i(x')
\;\leq\; \rho_{\R^{\Fs}_i}(x')
\;=\;\rho_{\Gamma(X,\R_i^{\Fs})}(x')
\;\leq\; \rho_{\Gamma_{X'}}(x')
\;=\; \rho_{x',X'}\;,
\end{equation}  
so $\R_j^{\Fs_{|X'}}(x')=\R^{\Fs}_j(x')$ for all $x'\in X'$. 
So  $H_j^{\Fs_{|X'}}(x')=H_j^{\Fs}(x')$ for all $x'\in X'$.
%The assertion about the (super-)harmonicity follows since the 
%equality of the functions implies the equality of their slopes. 

iv). We have  
$\R_j^{\Fs_{|X'}}(x')=\min(\R^{\Fs}_j(x'),\rho_{x',X'})=
\R^{\Fs}_j(x')$ 
since $\R^{\Fs}_j(x')\leq\R^{\Fs}_i(x')<\rho_{x',X'}$. Moreover 
this remains true by continuity over each germ of segment out of $x'$ 
(cf. Remark \ref{Rk : (C2)+(C4)}).
\if{We now prove  iv). 
Since $X'$ preserves the directions at $\xi''$, then $\Delta(X',\xi'')=\Delta(X,\xi'')$ (with evident meaning of notations). 
On the other hand, for all $\xi'\in X'$, $r(\xi')$ 
is independent on $X$ nor on $X'$, and $r(\xi')\leq \rho_{\xi',X'}\leq\rho_{\xi',X}$. We distinguish the radii verifying $\R_i^{\M}(\xi'')<r(\xi'')$ 
from those verifying $\R_i(\xi'')>\xi''$ (the case $\R_i(\xi'')=r(\xi'')$ is excluded by (WSH) of Thm. \ref{Theorem : MAIN THM GEN}). 
These inequalities remain strict along a conveniently small open segment containing $\xi''$ of each complete segment containing $\xi''$. 
If $\R_i^{\M}(\xi'')<r(\xi'')$ for all $\xi'$ in that interval one has 
$\R_i^{\M}(\xi')<r(\xi')\leq\rho_{\xi',X'}$ and so $\R_i^{\M'}(\xi')=\R_i^{\M}(\xi')$. 
The slopes at $\xi$ of $H_i^{\M}$ are then preserved by restriction to $X'$.
Conversely for all $i$ such that $\R_i^{\M}(\xi'')>r(\xi'')$, the function $\R_i^{\M}$ is constant around $\xi''$ by \eqref{...E.E.}. 
So the contribution to the super-harmonicity of $\R_i^{\M}$ is equal to zero over $X$ and over $X'$: the super-
harmonicity at $\xi''$ is controlled by those $\R_i^{\M}$ verifying $\R_i(\xi'')<r(\xi'')$. 
}\fi
\end{proof}

\subsection{Base change by a matrix in the fraction field 
$\mathscr{M}(X)$ of $\O(X)$}\label{base change in F(X).}
Let $\mathscr{M}(X)$ denotes the fraction field of $\O(X)$, and let 
$H\in GL_r(\mathscr{M}(X))$. 
Replacing $X$ by a sub-affinoid $X'$ having conveniently small 
holes around the zeros and poles of $H(T)$ and of $H(T)^{-1}$ we 
obtain $H,H^{-1}\in GL_r(\O(X'))$. 
If $x\in X$, is a given point of type $2$, $3$, $4$, 
then $X'$ can be chosen as an affinoid neighborhood of $x$ in $X$, 
because the zeros and poles are $K$-rational (recall that $K=\Ka$). 

\subsubsection{Reduction to a cyclic module.}
\label{Reduction to a cyclic}
Let $r:=\mathrm{rk}(\Fs)$ be the rank of $\Fs$.
By the cyclic vector theorem (cf. \cite{Katz-cyclic-vect}) 
one finds a cyclic basis of $\Fs(X)\otimes_{\O(X)}\mathscr{M}(X)$ in 
which $\Fs$ is represented by an operator 
$\L:=\sum_{i=0}^rg_{r-i}(T)(d/dT)^i$, with $g_i\in \mathscr{M}(X)$ 
for all $i$, and $g_0=1$. 

The operator $\L$ represents simultaneously the connection of all differential modules
$\Fs(x)=\Fs\otimes_{\O(X)}\H(x)$ for all 
$x\in X$ of type $2$, $3$, or $4$. 
If $H(T)\in\mathscr{M}(X)$ is the base change matrix, one can chose 
$X'\subseteq X$ as indicated in section \ref{base change in F(X).}. 
In order to fulfill Prop. \ref{Prop. NP of a operator} we can further restrict $X'$ in order that none of the $g_i$ has poles nor 
zeros on it. 

By Proposition \ref{Prop. restriction to a sub-affinoid} the restriction of 
$\Fs$ to $X'$ does not affect the finiteness. 
If moreover $\S_{X'}\subseteq\S(\R^{\Fs}_i)$, the super-harmonicity of 
$H_i^{\Fs}$ is also preserved.

\section{Push-forward by Frobenius}
\label{Pull-back and push-forward by Frobenius}
We here recall and slightly generalize some result about Frobenius 
coming from \cite{Kedlaya-Book} (cf. also 
\cite{Christol-GEAU},
\cite{Ch-Dw}, \cite{Pons}, 
\cite{Balda-Inventiones}). We study the behavior of 
$dd^cH_i^{\Fs}(x)$ by Frobenius descent. In \cite{Kedlaya-Book} this is done for an 
annulus, here we generalize it to an affinoid domain of $\mathbb{A}^{1,\mathrm{an}}_K$.
Along Section \ref{Pull-back and push-forward by Frobenius}, 
we assume that
$K$ is of mixed characteristic $(0,p)$, with $p>0$. 
Recall that $K=\Ka$. 
\subsection{Frobenius map}
Let $T,\widetilde{T}$ be two variables. 
The ring morphism $\varphi^\#:K[T]\to K[\widetilde{T}]$
%\begin{equation}
%\varphi^\#\;:\;K[T]\to K[\widetilde{T}]
%\end{equation} 
sending $f(T)$ into $f(\widetilde{T}^p)$, defines a morphism  
$\varphi:\mathbb{A}^{1,\mathrm{an}}_K\to
\mathbb{A}^{1,\mathrm{an}}_K$. 
If $t\in\mathbb{A}^{1,\mathrm{an}}_K(\Omega)$ is a Dwork generic point for $\xi\in\mathbb{A}^{1,\mathrm{an}}_K$, 
then $t^p$ is a Dwork generic point for $\varphi(\xi)$. Indeed for all $f\in K[T]$ one has 
$\varphi(\xi)(f)=\xi(f(\widetilde{T}^p))=|f(t^p)|_\Omega$. 

We now describe the image of a point of type $\xi_{t,\rho}$. 
For all $\sigma>0$ and $\rho,\rho'\geq 0$ we set
\begin{eqnarray}
\phi(\sigma,\rho)\;:=\;
\max(\rho^p,|p|\sigma^{p-1}\rho)&\;=\;&\left\{\sm{
\rho^p&\textrm{ if }& \rho\geq \omega\cdot \sigma\\
|p|\sigma^{p-1}\rho&\textrm{ if }&\rho\leq \omega\cdot \sigma
}\right.\;,\\
\psi(\sigma,\rho')\;:=\;
\min\Bigl((\rho')^{1/p},\frac{\rho'}{|p|\sigma^{p-1}}\Bigr)
&\;=\;&\left\{\sm{
(\rho')^{1/p}&\textrm{ if }& \rho'\geq \omega^p\cdot \sigma^p\\
\frac{\rho'}{|p|\sigma^{p-1}}&\textrm{ if }&\rho'\leq \omega^p\cdot 
\sigma^p
}\right.\;.
\end{eqnarray}
For $\sigma$ fixed, $\phi$ and $\psi$ are increasing functions of 
$\rho$ such that $\phi(\sigma,\psi(\sigma,\rho'))=\rho'$ and 
$\psi(\sigma,\phi(\sigma,\rho))=\rho$.
In the sequel of this section by convention of notations we set
%$\rho'=\phi(h,\rho)$ and $\rho=\psi(h,\rho')$. 
\begin{equation}
\rho'\;=\;\phi(\sigma,\rho)\;,\quad and\quad \rho\;=\;\psi(\sigma,\rho')\;.
\end{equation}
\begin{proposition}\label{Prop : varphi xi}
Let $c\in K$, $\rho>0$. Then
\begin{equation}
\varphi(x_{c,\rho})\;=\;x_{c^p,\phi(|c|,\rho)}\;,\quad
\varphi^{-1}(x_{c^p,\rho'})\;=\;
\{x_{\alpha c,\psi(|c|,\rho')}\}_{\alpha^p=1}\;.
\end{equation}
In particular if $\rho\geq\omega|c|$, 
$\varphi^{-1}(x_{c^p,\rho'})$ has an individual point, 
otherwise it has $p$ distinct points.\hfill$\Box$
\end{proposition}
\if{
\begin{proof}
By density and by multiplicativity it is enough to prove that for all 
$a\in K$ one has 
$\varphi(\xi_{c,\rho})(T-a)=\xi_{c^p,\phi(|t|,\rho)}(T-a)$. 
Write $\varphi(\xi_{c,\rho})(T-a)=\xi_{c,\rho}(
\widetilde{T}^p-a)=\xi_{c,\rho}((\widetilde{T}-c+c)^p-a)=
\xi_{c,\rho}(\sum_{k=0}^p\tbinom{p}{k}(
\widetilde{T}-c)^kc^{p-k}-a)=
\max(|c^p-a|,|p||c|^{p-1}\rho,\rho^p)$, 
in fact the terms corresponding to $k=1,\ldots,p-1$ form either a non 
decreasing or a non increasing sequence. 
On the other hand $\xi_{c^p,\phi(|c|,\rho)}(T-a)=
\xi_{c^p,\phi(|c|,\rho)}(T-c^p+c^p-a)=
\max(\phi(|c|,\rho),|c^p-a|)=\varphi(\xi_{c,\rho})(T-a)$. 
\end{proof}
\begin{remark}
Proposition \ref{Prop : varphi xi} also applies to points in 
$\mathbb{A}^{1,\mathrm{an}}_\Omega$, 
for all $\Omega\in E(K)$. In particular it applies to 
points of generic disks.
\end{remark}
}\fi

\if{\begin{remark}
If $t$ is a Dwork generic point for $\xi\in X^{1/p}$, then 
\begin{equation}
|t|\;=\;|\widetilde{T}|(x)
\end{equation}
is independent on $t$.
\end{remark}
\begin{remark}
If $\rho=\psi(|t|,\rho')\geq\omega|t|=|\alpha-\alpha'||t|$, then  
$\xi_{\alpha t,\psi(|t|,\rho')}=\xi_{\alpha' t,\psi(|t|,\rho')}$ for all  
$\alpha,\alpha'\in\bs{\mu}_p(K)$. Hence $\varphi^{-1}(\xi_{t^p,\rho'})$ has a single element. 
Conversely if $\rho< \omega|t|$, then $\varphi^{-1}(\xi_{t^p,\rho'})$ has $p$ distinct elements.
\end{remark}

\begin{lemma}\label{Lemma : radius of roots of T-t^p}
Let $\Omega\in E(K)$, $0\neq t\in\Omega$, 
$\alpha\in\bs{\mu}_p(K)$. For all $k=1,\ldots,p-1$ the power series  
\begin{equation}\label{eq : y'=y/p}
T^{k/p}-\alpha t\;=\;
(T-t^p+t^p)^{k/p}-\alpha t\;=\;
(t^k-\alpha t)+t^k\cdot\sum_{s\geq 1}\tbinom{k/p}{s}(\frac{T-t^p}{t^p})^s
\end{equation}	
has radius of convergence $\omega^p|t|^p$ at $t^p$. 
While its image by $\varphi^{\#}$ has infinite radius.\hfill$\Box$
\end{lemma}
}\fi
\if{\begin{proof}
Since $|\tbinom{k/p}{s}|=|k/p|^s/|s!|$, then 
$\liminf_s|\tbinom{k/p}{s}t^{-ps}|^{-1/s}_\Omega=|p||t|^p\liminf_s|s!|^{1/s}=\omega^p|t|^p$.
\end{proof}
}\fi
\if{\begin{proposition}\label{prop : action of phi on disks}
Let $t\in\Omega-\{0\}$, $0\leq\rho\leq|t|$, $0\leq\rho'\leq |t^p|$, be such that $\rho=\psi(|t|,\rho')$ and $\rho'=\phi(|t|,\rho)$. 
Let $\mathrm{D}^-(t,\rho)$ and $\mathrm{D}^-(t^p,\rho')$ be the open disks 
with algebras $\a_\Omega(t,\rho)$ and $\a_\Omega(t^p,\rho')$ respectively (cf. section \ref{Disks.}). Then: 
\begin{enumerate}
\item One has the following equalities 
\begin{equation}
\varphi(\mathrm{D}^-(t,\rho))\;=\;\mathrm{D}^-(t^p,\phi(|t|,\rho))\;,\quad
\varphi^{-1}(\mathrm{D}^-(t^p,\rho'))\;=\;\cup_{\alpha^p=1}\mathrm{D}^-(\alpha t,\psi(|t|,\rho'))
\end{equation}
\item For all $\alpha\in\bs{\mu}_p(K)$ the corresponding morphism 
\begin{equation}
\varphi_{\alpha,\rho}^*:\a_\Omega(t^p,\rho')\to\a_\Omega(\alpha t,\psi(|t|,\rho'))
\end{equation} 
is injective and isometric in the following sense. For all $f\in\a_\Omega(t^p,\rho')$ and all $\eta<\rho'$ one has 
\begin{equation}\label{eq : phi = isometric}
|f|_{t^p,\eta}\;=\;|\varphi_{\alpha,\rho}^*(f)|_{\alpha t,\psi(|t|,\eta)}\;.
\end{equation}
\item If $\rho'\leq\omega^p|t|^p$, then for all $\alpha\in\bs{\mu}_p(K)$, $\varphi_{\alpha,\rho}^*$ is an isomorphism of rings (which is 
isometric in the above sense).
\item If $\omega^p|t|^p<\rho'\leq|t|^p$, then $\bs{\mu}_p(K)$ acts on $\a_\Omega(t,\rho)$ by $\alpha(f)(T):=f(\alpha T)$, and one has 
\begin{equation}
\varphi_\rho^*(\a_\Omega(t^p,\rho'))\;=\;\a_\Omega(t,\rho)^{\bs{\mu}_p(K)}\;.
\end{equation}
where $\varphi_\rho^*:=\varphi_{1,\rho}^*$.
\item For all $0\leq \rho\leq|t|$ we denote by
\begin{equation}
\psi_{\alpha,\rho}^*:\a_\Omega(\alpha t,\rho)\to\a_{\Omega}(t^p,\rho')
\end{equation}
the $\Omega$-linear map defined as follows
\begin{enumerate}
\item  if $\omega|t|<\rho\leq|t|$, then $\psi_{\rho}^*(f(T)):=\frac{1}{p}\sum_{\alpha^p=1}f(\alpha T)$. 
The notation $\psi_\rho^*=\psi_{1,\rho}^*$ is justified by the fact that $\psi_\rho^*$ is 
independent on $\alpha$ since $\a_\Omega(\alpha t,\rho)=\a_\Omega(t,\rho)$ for all $\alpha$. 
\item if $0\leq \rho\leq \omega|t|$, then $\psi_{\alpha,\rho}^*:=(\varphi_{\alpha,\rho}^{*})^{-1}$.
For all $\alpha\in\bs{\mu}_p(K)$ the maps $\{\psi_{\alpha,\rho}\}_{\rho\in[0,|t|[}$ enjoy the following properties:
\begin{enumerate}
\item[(c)] For all $0\leq \rho \leq|t|$ one has $\psi^*_{\alpha,\rho}\circ\varphi_{\alpha,\rho}^*=\mathrm{Id}_{\a_\Omega(t^d,\rho')}$. 
\item[(d)] Let $0\leq\rho_1,\rho_2\leq|t|$. If $\omega|t|\notin [\rho_1,\rho_2[$, and if $\rho_i':=\phi(|t|,\rho_i)$, $i=1,2$, then the following diagram is commutative
\begin{equation}
\xymatrix{
\ar@{}[dr]|{\circlearrowleft}\a_\Omega(\alpha t,\rho_1)\ar[d]_{\psi_{\alpha,\rho_1}^*}&\ar[l]\a_\Omega(\alpha t,\rho_2)\ar[d]^{\psi_{\alpha,\rho_2}^*}\\
\a_\Omega(t^p,\rho_1')&\ar[l]\a_\Omega(t^p,\rho_2')
}
\end{equation}
where the horizontal maps are the restrictions. If $\omega|t|\in [\rho_1,\rho_2[$, then the diagram does not commute. 
\end{enumerate} 
\end{enumerate}
\end{enumerate}
\end{proposition}
}\fi
The following proposition describes the image and the inverse image 
by $\varphi$ of a disk. 
We mainly apply this to generic disks, so the 
center of the disk will be denoted by $t$, and all disks are 
$\Omega$-rational.

\begin{proposition}\label{prop : action of phi on disks}
Let $t\in\Omega$ and let $\rho,\rho'\geq 0$ be such that 
$\rho=\psi(|t|,\rho')$ and $\rho'=\phi(|t|,\rho)$. Then:
\begin{enumerate}
\item One has the following equalities 
\begin{equation}
\varphi(D^-(t,\rho))\;=\;
D^-(t^p,\rho')\;,\quad
\varphi^{-1}(D^-(t^p,\rho'))\;=\;
\cup_{\alpha^p=1}D^-(\alpha t,\rho)\;.
\end{equation}
\item For all $\alpha\in\bs{\mu}_p(K)$ the morphism 
$\varphi_{\alpha,\rho}^\#:\O(D^-(t^p,\rho'))\to
\O(D^-(\alpha t,\rho))$ %
is injective and isometric in the following sense. 
For all $f\in \O(D^-(t^p,\rho'))$ and all $\eta<\rho'$ one has 
\begin{equation}\label{eq : phi = isometric}
|f|_{t^p,\eta}\;=\;|\varphi_{\alpha,\rho}^\#(f)|_{\alpha t,\psi(|t|,\eta)}\;.
\end{equation}
\item If $\rho'\leq\omega^p|t|^p$, then for all $\alpha\in\bs{\mu}_p(K)$, $\varphi_{\alpha,\rho}^\#$ is an isomorphism of rings 
(satisfying \eqref{eq : phi = isometric}).
\item If $\omega^p|t|^p<\rho'$, then $\varphi_\rho^\#:=\varphi_{\alpha,\rho}^\#$ is independent on $\alpha$. 
Moreover $\bs{\mu}_p(K)$ acts on $\O(D^-(t,\rho))$ by $\alpha(f)(\widetilde{T}):=
f(\alpha \widetilde{T})$, and 
$\varphi_\rho^\#(\O(D^-(t^p,\rho')))=
\O(D^-(t,\rho))^{\bs{\mu}_p(K)}$.
%\begin{equation}
%\varphi_\rho^\#(\O(D^-(t^p,\rho')))\;=\;
%\O(D^-(t,\rho))^{\bs{\mu}_p(K)}\;.
%\end{equation}
%
%
%
\if{
\item For all $ \rho\geq 0$ we denote by 
$\psi_{\alpha,\rho}^\#:
\O(D^-(\alpha t,\rho))\to\O(D^-(t^p,\rho'))$
the $\Omega$-linear map defined as 
\begin{equation}
\psi_{\alpha,\rho}^\#:=\Bigl\{
\sm{
(\varphi_{\alpha,\rho}^{\#})^{-1}&\textrm{ if }&0\leq\rho\leq \omega|t|\\
\frac{1}{p}\sum_{\alpha^p=1}f(\alpha \widetilde{T})&\textrm{ if }&\omega|t|<\rho\\
}\Bigr.\;.
\end{equation}
\if{
follows. If $ 0\leq\rho\leq \omega|t| $, then $\psi_{\alpha,\rho}^\#:=
(\varphi_{\alpha,\rho}^{\#})^{-1}$, and if $\omega|t|<\rho$, then $\psi_{\rho}^\#(f(\widetilde{T})):=
\psi_{\alpha,\rho}^\#(f(\widetilde{T})):=\frac{1}{p}\sum_{\alpha^p=1}f(\alpha \widetilde{T})$.\footnote{The notation 
$\psi_\rho^\#=\psi_{\alpha,\rho}^\#$ is justified by the fact that $\psi_\rho^\#$ is 
independent on $\alpha$ since $\a_\Omega(\alpha t,\rho)=\a_\Omega(t,\rho)$ for all $\alpha$.} 
}\fi
Then for all $\alpha\in\bs{\mu}_p(K)$ the family of 
maps $\{\psi_{\alpha,\rho}\}_{\rho\geq 0}$ satisfy 
$\psi^\#_{\alpha,\rho}\circ\varphi_{\alpha,\rho}^\#=\mathrm{Id}_{\O(D^-(t^p,\rho'))}$, for all $\rho\geq 0$. 
Moreover if $\omega|t|\notin [\rho_1,\rho_2[$, and if $\rho_i':=\phi(|t|,\rho_i)$, then the following 
diagram is commutative where the horizontal maps are the restrictions:
\begin{equation}\label{diagram : psi and res}
\xymatrix{
\ar@{}[dr]|{\circlearrowleft}
\O(D^-(\alpha t,\rho_1))
\ar[d]_{\psi_{\alpha,\rho_1}^\#}&\ar[l]\O(D^-(\alpha t,\rho_2))
\ar[d]^{\psi_{\alpha,\rho_2}^\#}\\
\O(D^-(t^p,\rho_1'))&\ar[l]\O(D^-(t^p,\rho_2')).
}
\end{equation}
If $\omega|t|\in [\rho_1,\rho_2[$, then the diagram does not 
commute. %(e.g. Lemma \ref{Lemma : radius of roots of T-t^p}).

}\fi
\hfill$\Box$
\end{enumerate} 
\end{proposition}
\if{\begin{proof}
To prove i) one needs to evaluate $|a^p-t^p|$ for all $a\in\mathrm{D}^-_{\Omega'}(t,\rho)$, and arbitrary $\Omega'\in E(\Omega)$. 
The skeleton of $f(\widetilde{T})=\widetilde{T}^p-t^p=\prod_{\alpha^p=1}(\widetilde{T}-\alpha t)$ is  
$\S(f)=\mathrm{Sat}(\{\alpha t\}_{\alpha^p=1})$. Hence for all $a\in \Omega$ one has $|a^p-t^p|=
|\widetilde{T}^p-t^p|_{a,|a-\alpha_a t|}
|\widetilde{T}^p-t^p|_{\alpha_a t,|a-\alpha_a t|}=\phi(|t|,|a-\alpha_a t|)$, 
where $\alpha_a\in\bs{\mu}_p(K)$ satisfies $|a-\alpha_a t|=\min_{\alpha^p=1}|a-\alpha t|$.
In fact by Prop. \ref{Prop : varphi xi} one has $|\widetilde{T}^p-t^p|_{\alpha t,\rho}=\phi(|t|,\rho)$. 
Since $\rho\mapsto\varphi(|t|,\rho)$ is strictly increasing, and since $|a-\alpha_a t|\leq|a-t|$, then $|a^p-t^p|\leq\phi(|t|,|a-t|)$. 
This proves $\varphi(\mathrm{D}^-(t,\rho))\subseteq\mathrm{D}^-(t^p,\phi(|t|,\rho))$. Conversely applying $\psi(|t|,-)$ to 
$|a^p-t^p|=\varphi(|t|,|a-\alpha_at|)$ with $b:=a^p$ one finds $\psi(|t|,|b-t^p|)=|\alpha' b^{1/p}-\alpha'\alpha_{b^{1/p}}t|$ for all 
$\alpha'\in\bs{\mu}_{p}(K)$. So i) holds. 

To prove ii), by density one can assume that $f$ is a polynomial in $\Omega[T]\subset\a_\Omega(t^p,\rho')$, 
and by multiplicativity up to enlarge $\Omega$ one can assume $f=(T-a)$ of degree $1$. 
In this case the assertion is easy. The injectivity follows from \ref{eq : phi = isometric}.

To prove iii) we can assume $t\neq 0$ because $\rho\leq\omega^p|t|^p$. We now define $(\varphi_{\alpha,\rho}^\#)^{-1}$.  
By Lemma \ref{Lemma : radius of roots of T-t^p}, 
$g:=T^{1/p}-\alpha t\in\a_\Omega(t^p,\rho')$ because 
$\rho'\leq \omega^p|t|^p$. Hence $(\varphi^\#_{\alpha,\rho})^{-1}(\sum_{i=0}^n a_i (\widetilde{T}-\alpha t)^i)=\sum_{i=0}^n a_ig^i$. 
This defines a map $\Omega[\widetilde{T}-t]\to\a_\Omega(t^p,\rho')$ satisfying \eqref{eq : phi = isometric}, that 
extends by continuity to $\a_\Omega(t,\rho)$ and coincides with $(\varphi_{\alpha,\rho}^\#)^{-1}$. This proves iii). 
Assume $\omega|t|<\rho$. Since $\Omega[\widetilde{T}]^{\bs{\mu}_p}=\Omega[\widetilde{T}^p]=
\varphi^\#_{\rho}(\Omega[T])$, then by density one has iv) together with 
$\psi^\#_{\alpha,\rho}\circ\varphi_{\alpha,\rho}^\#=\mathrm{Id}_{\a_\Omega(t^p,\rho')}$. 
The commutativity of the diagram is evident. 
If $\omega|t|\in[\rho_1,\rho_2[$ it does not commute since for 
example 
$\psi_{\alpha,\rho_2}^\#(\widetilde{T})=0
\neq\psi_{\alpha,\rho_1}^\#(\widetilde{T})$, because 
$\psi_{\alpha,\rho_1}^\#=(\varphi_{\alpha,\rho_1}^\#)^{-1}$.
\end{proof}
}\fi
The morphism $\varphi$ induces a $K$-linear isometric inclusion 
$\varphi^\#:\H(\varphi(\xi))\to\H(\xi)$.
%\begin{equation}
%\varphi^\#\;:\;\H(\varphi(\xi))\to\H(\xi)\;.
%\end{equation}
\begin{corollary}\label{Cor : equality of Delta(x)}
Let $c\in K$, $\rho\geq 0$, $x=x_{c,\rho}$. 

If $\rho\neq \omega|c|$, the morphism $\varphi:
\mathbb{A}^{1,\mathrm{an}}_K\to
\mathbb{A}^{1,\mathrm{an}}_K$ provides a bijection
\begin{equation}\label{eq : phi:delta(x)->delta(phi(x))}
\varphi\;:\;\Delta(x)\;\simto\;\Delta(\varphi(x))\;.
\end{equation}

If $\rho= \omega|c|$, then \eqref{eq : phi:delta(x)->delta(phi(x))} is 
surjective. The inverse image of the germ of segment out of 
$\varphi(x)$ directed toward $+\infty$ has a single element, while the 
inverse image of each other germ of segment out of 
$\varphi(x)$ is formed by $p$ distinct germs of segments out of $x$.
\hfill$\Box$
\end{corollary}
\begin{proposition}\label{Prop : deg p}
Let $c\in K$, and let 
$x=x_{c,\rho}$, $\rho\geq 0$. Then :
\begin{enumerate}
\item If $\rho<\omega|c|$, 
then $[\H(\xi_{c,\rho}):\H(\varphi(\xi_{c,\rho}))]=1$.
\item If $\rho>\omega|c|$, then 
$[\H(\xi_{c,\rho}):\H(\varphi(\xi_{c,\rho}))]=p$. 
\end{enumerate} 
If $X$ is an affinoid domain of $\mathbb{A}^{1,\mathrm{an}}_K$, 
the same relations hold by density 
replacing $\H(x)$ and $\H(\varphi(x))$ by the 
local rings $\O_{X,x}$ 
and $\O_{X^p,\varphi(x)}$ respectively.\hfill$\Box$
\end{proposition}
\if{\begin{proof}
$K(\widetilde{T})$ (resp. $K(T)$) is dense on $\H(\xi)$ 
(resp. $\H(\varphi(\xi))$). 
The map $\varphi^\#:
K(T)\to K(\widetilde{T}),\;T\mapsto\widetilde{T}^p$ 
is an extension of degree $p$. 
By density the degree of $\H(\xi)/\H(\varphi(\xi))$ is equal to $1$ or 
$p$, and it is $1$ if and only if 
$\widetilde{T}=T^{1/p}\in\H(\varphi(\xi))$.

One has $r(x_{c,\rho})=\max(r(x_c),\rho)=\rho$, and $r(\varphi(\xi_{c,\rho}))=r(x_{c^p,\phi(|c|,\rho)})=
\min(r(\xi_{c^p}),\phi(|c|,\rho))
%=\min(\phi(|c|,r(\xi_c)),\phi(|c|,\rho))
=\phi(|c|,\rho)$.
Let $\Omega\in E(K)$, $t_{c,\rho}\in\Omega$ 
(resp. $t^p_{c,\rho}\in\Omega$) be a Dwork generic point for 
$\xi_{c,\rho}$ 
(resp. $\xi_{c^p,\phi(|c|,\rho)}$). One has a diagram
\begin{equation}\label{comm diag}
\xymatrix@=10pt{
\H(\xi)\ar[rr]\ar@{}[rrd]|{\circlearrowleft}&& \a_\Omega(t_{c,\rho},\rho)\\
\H(\varphi(\xi))\ar[u]^{\varphi^\#}\ar[rr]&&\a_\Omega(t^p_{c,\rho},\phi(|c|,\rho))\ar[u]_{\varphi_{1,\rho}^\#}.
}
\end{equation}

If $\rho>\omega|c|$, then 
$|t_{c,\rho}|=|\widetilde{T}-c+c|_{c,\rho}=\max(|c|,\rho)
\geq \rho>\max(\omega|c|,\omega\rho)=
\omega|t_{c,\rho}|$. 
And hence $\phi(|c|,\rho)=\phi(|t_{c,\rho}|,\rho)=\rho^p$. 

We proceed by contrapositive : if
$T^{1/p}\in \H(\varphi(\xi))$, then $T^{1/p}-t_{c,\rho}\in\a_\Omega(t_{c,\rho}^p,\rho^p)$ which is absurd by 
Lemma \ref{Lemma : radius of roots of T-t^p}. 
So $[\H(\xi):\H(\varphi(\xi))]=p$.

To prove i) write 
$T^{1/p}-c = \lim_sc\sum_{k=1}^s
\tbinom{1/p}{k}(\frac{T-c^p}{c^p})^k$. 
This limit converges in $\H(\varphi(\xi))$ with respect to 
$\varphi(\xi)=\xi_{c^p,|p||c|^{p-1}\rho}$. 
Indeed 
$|\tbinom{1/p}{k}(\frac{T-c^p}{c^p})^k|_{c^p,|p|
|c|^{p-1}\rho}=\frac{|1/p|^k}{|k!|}\frac{|p|^k\rho^k}{|c|^k}=
\frac{\rho^k}{|c|^k|k!|}\to 0$ since $\rho<\omega|c|$. 
Hence $T^{1/p}\in\H(\varphi(\xi))$.

The assertions about the local rings are deduced by density.
\end{proof}
}\fi
\begin{remark}\label{Rk : decrease the condition}
If $x=x_{c,\rho}$ fulfills the assumptions of condition i) of Proposition 
\ref{Prop : deg p}, then also does $\varphi(x)$ with respect to 
$\varphi^2(x)$. 
This is no longer true if we are in the situation ii). 

Namely if $x=x_{c,\rho}$ satisfies 
$\rho>\omega^{\frac{1}{p^n}}|c|$, then for all $k=1,\ldots,n$, 
$\varphi^k(x)=x_{c^{p^k},\rho^{p^k}}$ satisfies 
$\rho^{p^k}>\omega|c|^{p^k}$, and 
$[\H(\varphi^k(x)):\H(\varphi^{k+1}(x))]=p$. 
While $[\H(\varphi^{n+1}(x)):\H(\varphi^{n+2}(x))]=1$.
\end{remark}
\if{\subsection{Push-forward by Frobenius}
\label{Push-forward by Frob}
Let $X$ be an affinoid domain of $\mathbb{A}^{1,\mathrm{an}}_K$.
For all derivation $d:\O(X)\to\O(X)$ we denote by 
$d-\mathrm{Mod}(X)$ the 
category of finite free $\O(X)$-modules $\M$ with a connection 
$\nabla:\M\to\M$ verifying the Leibnitz rule with respect to $d$.
The map $\varphi^{\#}:
\O(X^p)\to
\O(X)$ verifies 
$(\frac{d/d\widetilde{T}}{p\widetilde{T}^{p-1}})\circ\varphi^\#
=\varphi^\#\circ d/dT$.
Denote by $C:(d/d\widetilde{T})-\mathrm{Mod}(X)
\xrightarrow[]{\;\sim\;} 
(\frac{d/d\widetilde{T}}{p\widetilde{T}^{p-1}})-
\mathrm{Mod}(X)$
the equivalence associating to 
$(\widetilde{\Fs},\widetilde{\nabla})$ the module 
$(\widetilde{\Fs},\frac{1}{p\widetilde{T}^{p-1}}\widetilde{\nabla})$,  
and which is the identity on the morphisms between objects.

\begin{definition}
The pull-back functor 
\begin{equation}
(\N,\nabla)\;\mapsto\;(\varphi^*(\N),\varphi^*(\nabla))
\;=\;
\Bigl(\N\otimes_{\O(X^p),\varphi^{\#}}\O(X), 
p\widetilde{T}^{p-1}(\nabla\otimes 1+
1\otimes \frac{d/d\widetilde{T}}{p\widetilde{T}^{p-1}})\Bigr)
\end{equation}
is the scalar extension functor by the map $\varphi^{\#}:
(\O(X^p),d/dT)\to
(\O(X),\frac{d/d\widetilde{T}}{p\widetilde{T}^{p-1}})$,
%\begin{equation}\label{eq: varphi ring dif morph}
%\varphi^{\#}:(\O(X^p),d/dT)
%\to
%(\O(X),\frac{d/d\widetilde{T}}{p\widetilde{T}^{p-1}})\;,
%\end{equation} 
followed by the functor $C^{-1}$. The push-forward by Frobenius 
\begin{equation}
(\widetilde{\M},\widetilde{\nabla})\;\mapsto\;
(\varphi_*(\widetilde{\M}),\varphi_*(\widetilde{\nabla}))\;=\;
(\widetilde{\M},\frac{1}{p\widetilde{T}^{p-1}}\widetilde{\nabla})\;.
\end{equation}
is the composite of $C$ followed by the restriction of scalars functor 
by $\varphi^\#$:
\begin{equation}
(d/d\widetilde{T})-
\mathrm{Mod}(\O(X))
\;\xrightarrow[C]{\;\sim\;} \;
(\frac{d/d\widetilde{T}}{p\widetilde{T}^{p-1}})-
\mathrm{Mod}(\O(X))\;\xrightarrow[]{\quad}\;
(d/dT)-\mathrm{Mod}(\O(X^p))\;.
\end{equation}
\end{definition}
}\fi
\subsection{Behavior of spectral non solvable 
radii by Frobenius push-forward}
\label{Section : push-forward radii}
The map $\varphi^{\#}:
\O(X^p)\to
\O(X)$ verifies 
$(\frac{d/d\widetilde{T}}{p\widetilde{T}^{p-1}})\circ\varphi^\#
=\varphi^\#\circ d/dT$.
Let $\Fs$ be a finite free module of rank $r$ over $\O(X)$, 
and let $\nabla:\Fs\to\Fs$ be a connection 
with respect to $d/d\widetilde{T}$.
The push-forward of $(\Fs,\nabla)$ is 
the $(\O(X^p),\frac{d}{dT})$-differential module 
$(\Fs,\frac{1}{p\widetilde{T}^{p-1}}\nabla)$, obtained by considering 
$\Fs$ as an $\O(X^p)$-module via $\varphi^{\#}$, so that 
$\frac{1}{p\widetilde{T}^{p-1}}\nabla:\Fs\to\Fs$ is a connection 
with respect to $d/dT$. 
We will denote it by $(\varphi_*\Fs,\varphi_*\nabla):=
(\Fs,\frac{1}{p\widetilde{T}^{p-1}}\nabla)$.

Let $x\in X$ be a point of type 
$2$, $3$, or $4$.
Spectral non solvable radii of $\Fs$ at $x$ only depend on its restriction 
to $\H(x)$. 
We study separately the two cases of Proposition \ref{Prop : deg p}.

We firstly consider the situation i) of Proposition \ref{Prop : deg p}, where  $x=x_{c,\rho}$, with $\rho<\omega|c|$.
In this case $\varphi^\#:\H(\varphi(x))\to\H(x)$ is an isomorphism of fields, 
and hence the scalar extension, and the restriction of scalars, functors are equivalences of categories. 
By Proposition \ref{prop : action of phi on disks} the radii of all 
sub-disks of the generic disk $D(x)$ are multiplied by $|p||t|^{p-1}$, 
where $t$ is a Dwork generic point for $x$.\footnote{Note that $x=x_{c,\rho}$, hence 
$|t|=|\widetilde{T}|(x)=\max(|c|,\rho)=\max(|c|,r(x))$, since $r(x)=\rho$.} 
So for all $i=1,\ldots,r$ we have 
\begin{equation}\label{eq : esdpsps}
\R_i^{\varphi_*(\Fs),\mathrm{sp}}(\xi)
\;=\;|p||t|^{p-1}
\R_i^{\Fs,\mathrm{sp}}(\xi)\;. 
%If $\rho=\omega|c|$, equality \eqref{eq : esdpsps} is false 
%in general.
%This is essentially due to the fact that the diagram 
%\ref{diagram : psi and res} does not commute. 
%A counterexamples is given by the equation $y'=(1/pT)y$, 
%(cf. Lemma \ref{Lemma : radius of roots of T-t^p}.
\end{equation}

In the situation ii) of Proposition \ref{Prop : deg p}, where  $x=x_{c,\rho}$, with $\rho>\omega|c|$,
the map $\varphi^\#:\H(\varphi(x))\to\H(x)$ 
is a field extension of degree $p$. 
The situation is then regulated by the following results:

\begin{proposition}\label{prop: rad of Frob. push-forward}
Let $x\in \mathbb{A}^{1,\mathrm{an}}_K$ 
be a point of type $2$, $3$, or $4$ of the form 
$x=x_{c,\rho}$, with $c\in K$, and $\rho>\omega|c|$.
Let $\Fs$ be a differential module over 
$\H(x)$ of rank $r$. 

Define $0\leq i_1(x)\leq r$ as the index satisfying\footnote{It is understood that $i_1(x)=0$ if and only 
if  $\R_i^{\Fs,\mathrm{sp}}(\xi)>\omega|t|$ for all $i$.} 
\begin{equation}\label{eq : i_1 def}
\R^{\Fs,\mathrm{sp}}_{i_1(x)}(\xi)\;\leq \;
\omega|t|\;<\;\R_{i_1(x)+1}^{\Fs,\mathrm{sp}}(\xi)\;.
\end{equation}
Then, up to a permutation, the list (with multiplicities) of the spectral radii of $\varphi_*(\Fs)$ is given by 
\begin{equation}
\bigcup_{1\leq i\leq i_1(x)}\Bigl\{\underbrace{|p||t|^{p-1}\R_i^{\Fs,\mathrm{sp}}(\xi),
\ldots,|p||t|^{p-1}\R_i^{\Fs,\mathrm{sp}}(\xi)}_{p\textrm{-times}}\Bigr\}
\bigcup_{i_1(x)<i\leq r}\Bigl\{\R_i^{\Fs,\mathrm{sp}}(\xi)^p,
\underbrace{\omega^p|t|^p,\ldots,\omega^p|t|^p}_{p-1\textrm{-times}}\Bigr\}\;.
\end{equation}
\end{proposition}
\begin{proof}
The proof follows 
\cite[Thm. 10.5.1]{Kedlaya-Book}, with slide modifications. 
\if{
We reproduce it for the convenience of the reader.
We can assume that $\widetilde{\M}$ has no non trivial sub-objects, 
so $\R_1^{\widetilde{\M},\mathrm{sp}}(\xi)=\cdots=\R_r^{\widetilde{\M},\mathrm{sp}}(\xi)=R$. 
Assume $R>\omega|t|$. Let $\M$ be such that $\widetilde{\M}=\varphi^*(\M)$ (cf. Prop.\ref{Prop. antec.frob}). 
Then $\M$ is simple, since a sub-object $\N\subset\M$ produces a sub-object $\varphi^*\N$ of $\widetilde{\M}$. 
Hence $\R_1^{\M,\mathrm{sp}}(\varphi(\xi))=\cdots=\R_r^{\M,\mathrm{sp}}(\varphi(\xi))=R^p$. 
By Lemma \ref{Lemma: standard properties} one has $\varphi_*\widetilde{\M}=\varphi_*\varphi^*\M=\oplus_{k=0}^{p-1}\M[k]$. 
Each $\M[k]:=\M\otimes_{\H(\varphi(\xi))}\H(\varphi(\xi))\cdot T^{k/p}$ is irreducible since $\H(\varphi(\xi))\cdot T^{k/p}$ has rank one.
Since $\R^{\H(\varphi(\xi))\cdot T^{k/p}}_1(\varphi(\xi))=\left\{\sm{
\omega^p|t|^p&\textrm{ if }&k\neq 0\\
r(\varphi(\xi))&\textrm{ if }&k=0}\right.$, and since $\omega^p|t|^p<R^p$, 
then $\R^{\M[k],\mathrm{sp}}(\varphi(\xi))=\left\{\sm{
\omega^p|t|^p&\textrm{ if }&k\neq 0\\
R^p&\textrm{ if }&k=0}\right.$, because the radius of a tensor product is the minimum of the radii if they are different 
(cf. \cite[Lemma 9.4.6]{Kedlaya-Book}). This proves the result. 
Assume now that $R\leq\omega|t|$. Let $R':=|p||t|^{p-1}R=\phi(|t|,R)$. 
One has a commutative diagram compatible with the derivations like \eqref{comm diag} with 
$\a_\Omega(t_{c,\rho},\rho)$ (resp. $\a_\Omega(t_{c,\rho}^p,\phi(|c|,\rho))$) replaced by 
$\a_\Omega(t,R)$ (resp. $\a_\Omega(t^p,R')$). The restriction $\H(\varphi(\xi))\to\a_\Omega(t^p,R')$ then factorizes 
through $\H(\varphi(\xi))\to\H(\xi)\to\a_\Omega(t,R)\stackrel{(\varphi_{1,R}^\#)^{-1}}{\to}\a_\Omega(t^p,R')$ 
(cf. Prop. \ref{prop : action of phi on disks}, iii)). Hence 
$\varphi_*(\widetilde{\M})\otimes_{\H(\varphi(\xi))}\a_\Omega(t^p,R')=\varphi^*\varphi_*(\widetilde{\M})\otimes_{\H(\xi)}\a_\Omega(t^p,R')=
\widetilde{\M}^p\otimes_{\H(\xi)}\a_\Omega(t^p,R')$ (cf. Lemma \ref{Lemma: standard properties}).
By assumption $\widetilde{\M}$ is trivialized by $\a_\Omega(t,R)$ and hence by $\a_\Omega(t^p,R')$, 
so $\R^{\varphi_*\widetilde{\M},\mathrm{sp}}(\varphi(\xi))\geq R'$. %
\if{
Concretely ... by sending the basis $\widetilde{T}$ ... $f(\widetilde{T})=\sum_{k=0}^{p-1}\widetilde{T}^{k}f_k(\widetilde{T}^p)\mapsto
\sum_{k=0}^{p-1}\widetilde{T}^{k}f_k(T)$.

Recall that $\varphi_*(\H(\xi))\cong \oplus_{k=0}^{p-1}\H(\varphi(\xi))\cdot T^{k/p}$. 
\begin{lemma}\label{Lemma: 6.17}
For all $R\leq \omega|t|$, $R':=|p||t|^{p-1}R$, the canonical map 
$\H(\xi_{c,\rho})\widehat{\otimes}_{\H(\varphi(\xi_{c,\rho}))}\a_\Omega(t^p,R')\simto\oplus_{\alpha^p=1}\a_\Omega(\alpha t,R)$  
(cf. \eqref{comm diag}) is an $\H(\varphi(\xi))$-linear isomorphism commuting with $d/dT$ and $d/d\widetilde{T}$. 
\end{lemma}
\begin{proof}
An element in $\H(\xi)\widehat{\otimes}_{\H(\varphi(\xi))}\a_\Omega(t^p,R')$ can be uniquely written as 
$\sum_{\alpha^p=1}\alpha\widetilde{T}\widehat{\otimes}g_\alpha(T)$, and is sent to 
%its image is 
$\oplus_{\alpha^p=1}\alpha\widetilde{T}\cdot 
\varphi^\#_{\alpha,R}(g_\alpha(T))$. The inverse map is $\oplus_{\alpha^p=1}f_\alpha(\widetilde{T})\mapsto 
\sum_{\alpha^p=1}\alpha\widetilde{T}\widetilde{\otimes}\psi_{\alpha,R}(f_\alpha(\widetilde{T})/\alpha\widetilde{T})$.
\end{proof}
By Lemma \ref{Lemma: 6.17} 
$\varphi_*(\widetilde{\M})\otimes_{\H(\varphi(\xi))}\a_\Omega(t^p,R')\cong
\widetilde{\M}\otimes_{\H(\xi)}(\H(\xi)\otimes_{\H(\varphi(\xi))}\a_\Omega(t^p,R'))$ 
is isomorphic to
$\varphi_*(\widetilde{\M}\otimes_{\H(\xi)}\oplus_{\alpha^p=1}\a_\Omega(\alpha t,R))=
\oplus_{\alpha^p=1}(\varphi_{\alpha,R})_*(\widetilde{\M}\otimes_{\H(\xi)}\a_\Omega(\alpha t,R))$. 
Now $\widetilde{\M}$ is trivialized by $\a_\Omega(\alpha t,R)$, and $\varphi_{\alpha,R}$ is an isomorphism by 
Prop. \ref{prop : action of phi on disks}, iii). 
Then $\varphi_*\widetilde{\M}$ is trivialized by $\a_\Omega(t^p,R')$, and $\R^{\varphi_*\widetilde{\M},\mathrm{sp}}(\varphi(\xi))\geq R'$.
}\fi
If now $0\neq \N\subseteq\varphi_*\widetilde{\M}$, then $\R^{\N,\mathrm{sp}}(\varphi(\xi))\geq 
\R^{\varphi_*(\M),\mathrm{sp}}(\varphi(\xi))\geq R'$. We claim that this is an equality for all $\N$. Since 
$\varphi^*\N\subseteq\varphi^*\varphi_*(\widetilde{\M})=\widetilde{\M}^p$,  each simple sub-quotient of 
$\varphi^*\N$ appears among those of $\widetilde{\M}$, and its radius is $R$. 
So $\R^{\varphi^*\N,\mathrm{sp}}(\xi)=R\leq \omega|t|$ and by Proposition \ref{prop: 6.8} one has 
$\R^{\N,\mathrm{sp}}(\varphi(\xi))=R'$.
}\fi
\end{proof}

\begin{corollary}\label{Cor : sllls}
We maintain the assumptions of Proposition 
\ref{prop: rad of Frob. push-forward}.
Let $i_x^{\mathrm{sp}}$ (resp. $i_1(x)$) be the largest 
index satisfying 
$\R_{i}^{\Fs}(x)<r(x)$ 
(resp. $\R_{i}^{\Fs}(x)\leq\omega\max(|c|,r(x))$ as in \eqref{eq : i_1 def}). 

For all $i\in\{1,\ldots,r\}$ we define\footnote{It is understood that if $i_1(x)=i_x^{\mathrm{sp}}$, 
then $\phi(i,x)=pi$ and $d_i(x)=i$ for all $i\in\{1,\ldots,i_x^{\mathrm{sp}}\}$.} 
\begin{equation}
\phi(i,x):=\left\{\sm{
pi&\textrm{ if }&1\leq i< i_1(x)\\
(p-1)r+i&\textrm{ if }& i_1(x)\leq i\leq r
}\right.\;,\qquad 
d_i(x):=
\left\{\sm{
i&\textrm{ if }&1\leq i< i_1(x)\\
r&\textrm{ if }&i_1(x)\leq i\leq r
}\right.\;,\qquad 
\ell_{i,x}(\widetilde{T}):=(p\widetilde{T}^{(p-1)})^{d_i(x)}\in\O(X)\;.
\end{equation}   
Let $i\in\{1,\ldots,i_x^{\mathrm{sp}}\}$, then 
\begin{equation}\label{eq : HM_i=HphiM_i}
|\ell_{i,x}|(x)\cdot H_{i}^{\Fs}(x)\;=\;
H_{\phi(i,x)}^{\varphi_*\Fs}(
\varphi(x))^{1/p}\;.
\end{equation}
%The same relation holds for $H_i^{\Fs,sp}$ for 
%all $i=1,\ldots,r$. 
\end{corollary}
\begin{proof}
Write $s_1^{\Fs}(x)\leq \cdots 
\leq s_{i_1(x)}^{\Fs}(x)
\leq \ln(\omega|t|) <
s_{i_1(x)+1}^{\Fs}(x)\leq
\cdots \leq 
s^{\Fs}_{i_x^{\mathrm{sp}}}(x)<\ln(\rho)\leq 
s^{\Fs}_{i_x^{\mathrm{sp}}+1}(x)\leq\cdots 
\leq s^{\Fs}_{r}(x)$. %
The behavior of spectral non solvable radii is given by 
Prop. \ref{prop: rad of Frob. push-forward}, so we have
\begin{eqnarray}
s^{\varphi_*\Fs}(\varphi(x)) &\;:\;& \phantom{\leq}
\overbrace{\ln\Bigl(|p||t|^{p-1}\Bigr)+s_{1}^{\Fs}(\xi)
=\cdots=\ln\Bigl(|p||t|^{p-1}\Bigr)+
s_{1}^{\Fs}(\xi)}^{p\textrm{-times}}
\leq\cdots\\&&
\leq\overbrace{\ln\Bigl(|p||t|^{p-1}\Bigr)+
s_{i_1(x)}^{\Fs}(\xi)=\cdots=
\ln\Bigl(|p||t|^{p-1}\Bigr)+
s_{i_1(x)}^{\Fs}(\xi)}^{p\textrm{-times}}\leq\\&&
\leq\overbrace{\ln
\Bigl(\omega^p|t|^p\Bigr)=\cdots=
\ln\Bigl(\omega^p|t|^p\Bigr)}^{(p-1)(r-i_1(x))\textrm{-times}}
<p s_{i_1(x)+1}^{\Fs}(\xi)\leq\cdots\leq 
p s_{i_x^{\mathrm{sp}}}^{\Fs}(\xi)<p\ln(\rho)\leq\cdots\qquad\quad
\label{eq : (5.15)}\end{eqnarray}
where $t$ is a Dwork generic point for $x$.
Hence for all $i< i_1(x)$ we have $h_{pi}^{\varphi_*\Fs}(\varphi(x)) = 
p\cdot h_{i}^{\Fs}(\xi)+p\cdot i\cdot \ln(|p||t|^{p-1})$. 
And if $i_1(x)\leq i \leq i_x^{\mathrm{sp}}$, then 
\begin{eqnarray}
h_{(p-1)r+i}^{\varphi_*\Fs}(\varphi(x)) 
%&\;=\;& 
%h_{pi_1(x)}^{\varphi_*\Fs}(\varphi(x)) +
%(p-1)(r-i_1(x))\ln(\omega^p|t|^p)+
%ps_{i_1(x)+1}^{\Fs}(x)+\cdots+
%ps_{i}^{\Fs}(x)\qquad\quad\\
&\;=\;& 
p\cdot h_{i}^{\Fs}(x)+p\cdot i_1(x)\cdot 
\ln(|p||t|^{p-1})+(p-1)(r-i_1(x))\ln(\omega^p|t|^p)\qquad\\
&\;=\;&p\cdot h_{i}^{\Fs}(x) + 
p\cdot r\cdot\ln(|p||t|^{p-1})\;.
\end{eqnarray}
This proves \eqref{eq : HM_i=HphiM_i}. 
\end{proof}
\begin{remark}
The proof shows also that if $i$ is a spectral non solvable index, 
then $i$ is a vertex of 
$NP^{\mathrm{conv}}(x,\Fs)$ (i.e. $i=r$ or 
$s^{\Fs}_i(x)<s^{\Fs}_{i+1}(x)$) 
if and only if 
$\phi(i,x)$ 
is a vertex of 
$NP^{\mathrm{conv}}(\varphi(x),\varphi_*(\Fs))$.
% (i.e. $\phi(i)=pr$ or $s^{\varphi_*(\Fs)}_{\phi(i)}
%(\varphi(x))<s^{\varphi_*(\Fs)}_{\phi(i)+1}
%(\varphi(x))$). 
\end{remark}
\subsection{Behavior of the (spectral non solvable) slopes by 
Frobenius push-forward}\label{Harmonicity and Frob-push-frw}

We now study the behavior of the slopes of the radii along the germs 
of segments out of a point. 

We maintain the notations of Section 
\ref{Section : push-forward radii}. 
If $x\in X$ is a point of type 
$2$, $3$, or $4$, the local ring 
$\O_{X,x}$ is a differential field.
For $b\in\Delta(x)$, the slopes $\partial_b\R_1^{\Fs}(x),\cdots,
\partial_b\R_{i_x^{\mathrm{sp}}}^{\Fs}(x)$ 
of the spectral non solvable radii of $\Fs$ 
only depend on its restriction to 
$\O_{X,x}$. 
Indeed spectral radii are stable by localization, and if an index 
$i$ is spectral non solvable at $x$, by continuity 
(cf. Thm. \ref{thm: continuity of spectral radii along a branch}) 
it remains spectral non solvable 
over $b$. As above we distinguish the two situations of Proposition 
\ref{Prop : deg p}.

If we are in the situation i) of Proposition \ref{Prop : deg p}, then 
$\varphi$ is a trivial covering of disks (cf. Prop. \ref{prop : action of phi on disks}). 
So if $b$ is a germ of segment out of $x$, then for all 
$i=1,\ldots,i_x^{\mathrm{sp}}$ we have
\begin{equation}
\partial_b\R^{\Fs}_{i}(x)=
\partial_{\varphi(b)}\R^{\varphi_*(\Fs)}_{i}(\varphi(x))\;,\qquad
\partial_bH^{\Fs}_{i}(x)=
\partial_{\varphi(b)}H^{\varphi_*(\Fs)}_{i}(\varphi(x))\;.
\end{equation}
The Laplacians are then naturally identified.

In the situation of the point ii) of Proposition \ref{Prop : deg p}, 
the map $\varphi^\#:\O_{X^p,\varphi(x)}\to\O_{X,x}$
is a field extension of degree $p$. 
Let $b=]x,y[$ be a germ of segment out of $x$. We now compare the slopes $\partial_bH_i^{\Fs}$ with those of 
$\partial_bH_{\phi(i,x)}^{\varphi_*\Fs}$. 
We can restrict $]x,y[$ in order that the function $z\mapsto i_1(z)$ is constant over 
$b=]x,y[$. We call the corresponding quantities $i_1(b)$, $\phi(i,b)$, $d_i(b)$, $\ell_{i,b}$. 
For $i\leq i^{\mathrm{sp}}_x$ we have
\begin{equation}\label{eq : slopes}
%H_{\phi(i,x)}^{\varphi_*(\Fs)}
%(\varphi(x))\;=\;H_{\phi(i,b)}^{\varphi_*(\Fs)}
%(\varphi(x))\;,\qquad
\partial_{\varphi(b)}H_{\phi(i,b)}^{\varphi_*(\Fs)}
(\varphi(x))\;=\;
\partial_bH_i^{\Fs}(x)+
\partial_b|\ell_{i,b}|(x)\;.\footnote{The natural parametrization \eqref{eq : parametrization} 
of $\varphi(b)$ multiplies by $p$ the 
distances, while the exponent $1/p$ of $H_{\phi(i,b)}^{\varphi_*(\Fs)}
(\varphi(x))$ divides by $p$ the result. So globally we have \eqref{eq : slopes}.}
\end{equation}
Notice that we may have $\phi(i,x)\neq \phi(i,b)$, namely this can only happen if 
$\R_{i_1(x)}^{\Fs}(x)=\omega|t|$.
However it follows from \eqref{eq : (5.15)} that,  if $i\leq i_x^{\mathrm{sp}}$ is a \emph{vertex} at $x$ 
of the convergence newton 
polygon, then $\phi(i,x)=\phi(i,b)$ for all $b\in\Delta(x)$.
The same happens for $\ell_{i,x}$. We then denote them 
by $\phi(i)$ and $\ell_i$. 

Assume that $i\leq i_x^{\mathrm{sp}}$ is a vertex of  $NP^{\mathrm{conv}}(x,\Fs)$. 
Then, for all $b\in\Delta(x)$, we have
\begin{equation}\label{eq : variation of the ph by Fr}
\partial_{\varphi(b)}H_{\phi(i)}^{\varphi_*(\Fs)}
(\varphi(x))\;=\;
\partial_bH_i^{\Fs}(x)+
\partial_b|\ell_i|(x)\;.
\end{equation}
By \eqref{eq : phi:delta(x)->delta(phi(x))}, 
the directions out of $x$ coincide with those out of $\varphi(x)$. Moreover $\ell_i$ is harmonic, 
so the Laplacians are also identified (once we will prove that the sum defining the Laplacian is finite):
\begin{equation}\label{eq : dd^c preserved by Frobenius}
dd^cH_i^{\Fs}(x)\;=\;
dd^cH_{\phi(i)}^{\varphi_*(\Fs)}
(\varphi(x))
\;.
\end{equation}
\begin{proposition}\label{Prop : SH-well done spectral NS}
We allow the case where the valuation of $K$ is trivial on $\mathbb{Z}$. 
Let $x\in X$ be a point of type $2$, $3$, or $4$. 
If the index $i$ is free of solvability at $x$ 
(cf. Def. \ref{Def : index free of solvabl}), then 
\begin{enumerate}
\item the slopes of $\R^{\Fs}_1,\ldots,\R^{\Fs}_i$ and of 
$H_1^{\Fs},\ldots,H_i^{\Fs}$ are zero for almost but a finite number 
of germs of segments out of $x$;
\item $H_1^{\Fs},\ldots,H_i^{\Fs}$ are super-harmonic at $x$ and 
satisfy properties iii) and iv) of Proposition 
\ref{Prop. NP of a operator} around $x$ 
(cf. also Remark \ref{trspNP});
\item In particular if $i$ is a vertex of $NP^{\mathrm{conv}}(x,\Fs)$, 
then $H_i^{\Fs}(x)$ is harmonic at $x$.
\end{enumerate}
\end{proposition}
\begin{proof}
Over-solvable radii are constant around $x$ by \eqref{(2)}, 
so they do not play any role, and we can assume $i\leq i_x^{\mathrm{sp}}$. 
Assume first that $K$ is of mixed characteristic $(0,p)$, with $p>0$.
The radii $\R^{\Fs}_i$ are insensitive to scalar extension, so 
replacing $K$ by a larger field we can assume that $x$ 
is of type $2$, and by a translation we can assume $x=x_{c,\rho}$, 
with $c=0$. 
This guarantee that for all $k\geq 0$, $\varphi^k(x)$ satisfies the 
situation ii) of Proposition \ref{Prop : deg p} 
(cf. Remark \ref{Rk : decrease the condition}). 
We apply Frobenius push-forward 
several times in order that $\R_{\phi^n(i)}^{\varphi_*^n\Fs}(\varphi^n(x))<\omega|t^{p^n}|$ (which is the 
assumption of  Proposition \ref{Prop. : small radius}). By continuity (cf. Remark \ref{Rk : (C2)+(C4)}) 
this assumption remains verified along all directions out of $\varphi^n(x)$.  
Now, by point (iv) of Proposition \ref{Prop. restriction to a sub-affinoid}, and by 
Section \ref{Reduction to a cyclic}, we can localize and pass to a cyclic basis without affecting the 
super-harmonicity, nor the directional finiteness.
In a cyclic basis the radii are explicitly intelligible by Propositions \ref{Prop. : small radius} and \ref{Prop. NP of a operator}. 
Now, by \eqref{eq : slopes}, for all $b\in\Delta(x)=\Delta(\varphi^n(x))$ 
the slope $\partial_bH_i^{\Fs}(x)$ appears among those in the family 
$\{\partial_{\varphi^n(b)}H_{j}^{\varphi^n_*\Fs}(\varphi^n(x))\}_j$. 
For almost all $b\in\Delta(x)$ these slopes are all zero by Propositions 
\ref{Prop. : small radius} and \ref{Prop. NP of a operator}, so i) holds. 
Moreover to prove the other statements we can assume, by interpolation, that $i$ is a vertex at $x$ of 
$NP^{\mathrm{conv}}(\Fs)$ 
(cf. proof of Proposition \ref{Prop. NP of a operator}), so we can use \eqref{eq : dd^c preserved by Frobenius} to reduce to 
Propositions \ref{Prop. : small radius} and \ref{Prop. NP of a operator}.

The case where the valuation is trivial on $\mathbb{Z}$ is much more 
easier. Indeed $\omega=1$, and we can immediately apply Propositions 
\ref{Prop. : small radius} and \ref{Prop. NP of a operator}, without involving any Frobenius machinery.
\end{proof}
%\begin{remark}
%The cases where some radii are solvable will be controlled by using Lemma 
%\ref{Lemma : KEY . .!}.
%\end{remark}
\section{Proof of the main Theorem 
\ref{Theorem : MAIN THM GEN}}\label{Proof of the MTH}

The properties of Theorem \ref{Theorem : MAIN THM GEN} 
are invariant by scalar extension of the 
ground field $K$. 
So, from now on we assume that $K$ is algebraically 
closed, and spherically complete.

By Remark \ref{Def.: radius on a disk}, $\R^{\Fs}_i$ 
and $\R_i(-,\Fs)$ are closely related. Indeed the function 
$x\mapsto\rho_{x,X}$ is continuous, locally constant outside 
$\Gamma_X$, and with slope $+1$ on each segment in $\Gamma_X$ 
oriented as towards $+\infty$. As it is clearly stated in Theorem 
\ref{Theorem : MAIN THM GEN} 
each assertion about $\R_{i}(-,\Fs)$ and $H_i(-,\Fs)$ 
is equivalent to an assertion bout $\R_i^{\Fs}$ and $H_i^{\Fs}$. In 
the following we prove those assertions for $\R_i^{\Fs}$ and 
$H_i^{\Fs}$ since the super-harmonicity and localization 
properties are more easy.

We begin by describing the link between the graphs of the partial
heights and those of the radii.

\begin{remark}\label{Rk : defof RRM_i}
For $i\leq r$, let $\RR_i^{\Fs}:X\to\mathbb{R}^i$ be the function defined by 
\begin{equation}
\RR_i^{\Fs}(x)\;:=\;
(\R_1^{\Fs}(x),\ldots,\R_i^{\Fs}(x))\;.
\end{equation}
Defines analogously $\HH_i^{\Fs},\ss_i^{\Fs},\hh_i^{\Fs}$. Clearly 
$\rho_{\RR_i^{\Fs}}(x)=
\min_{j=1,\ldots,i}\rho_{\R_i^{\Fs}}(x)$, 
so that 
\begin{equation}\label{eq : Gamma_ikjhg}
\S(\RR_i^{\Fs})\;=\;
\bigcup_{j=1,\ldots,i}\S(\R_j^{\Fs}(x))\;=\;
\Gamma_i\;.
\end{equation}
Hence the finiteness of $\RR_r^{\Fs}$ is equivalent to the finiteness of 
all $\R_i^{\Fs}$. 
The same holds for $\HH_i^{\Fs},\ss_i^{\Fs},\hh_i^{\Fs}$. 

The maps $\RR_i^{\Fs}$ and $\HH_i^{\Fs}$ are the exponential of 
$\ss_i^{\Fs}$ and 
$\hh_i^{\Fs}$ respectively, and the exponential map is injective,
so we are reduced to prove the finiteness of 
$\ss_i^{\Fs}$ and $\hh_i^{\Fs}$. 
The functions $\ss_i^{\Fs}$ and $\hh_i^{\Fs}$ are related by the 
bijective map $\hh_i^{\Fs}(x)=U\cdot \ss_i^{\Fs}(x)$, 
where $U\in GL_r(\mathbb{Z})$ is the matrix $U=(u_{i,j})$ with 
$u_{i,j}=1$ if $i\geq j$ and $u_{i,j}=0$ otherwise. This proves that 
for all $i=1,\ldots,r$ (cf. Def. \ref{Def : Gamma_i})
\begin{equation}
\Gamma_i\;=\;\Gamma(\RR^{\Fs}_i)\;=\;
\Gamma(\HH^{\Fs}_i)\;=\;
\Gamma(\hh^{\Fs}_i)\;=\;
\Gamma(\ss^{\Fs}_i)\;.
\end{equation}
In particular the finiteness of all the partial heights is equivalent to that 
of all the radii.
\end{remark}

The proof of Theorem 
\ref{Theorem : MAIN THM GEN} consists in showing that 
$H^{\Fs}_i$ satisfy the properties (C1),\ldots,(C6) 
of section \ref{F_t defi}, plus the other claims of the theorem. 
We prove them by induction on $i$.

We already know, by Remark \ref{Rk : (C2)+(C4)}, 
that (C1), (C2), (C4) hold for $\R_i^{\Fs}$ and $H_i^{\Fs}$, and 
moreover that points ii) and iii) of 
Theorem \ref{Theorem : MAIN THM GEN} 
follow from the analogous claims for spectral radii (cf.
Thm. \ref{thm: continuity of spectral radii along a branch}).
Finally point v) of Theorem \ref{Theorem : MAIN THM GEN} 
is proved in point iii) 
of Proposition  \ref{Prop : SH-well done spectral NS}.

It remains to prove the weak super-harmonicity 
(i.e. point iv) of Theorem \ref{Theorem : MAIN THM GEN}), 
together with (C3) and (C5). This will implies the finiteness by Theorem 
\ref{th : finiteness theorem}. 

More precisely we prove, by induction on $i$, that $H_i^{\Fs}$ verifies  
(C3), (C5), (C6), with respect to $\Gamma:=\Gamma_{i-1}$  (where $\Gamma_0:=\Gamma_X$), 
and $\C(H_i^{\Fs}):=\C_i$ (which is finite by induction, by definition (b) and (c) of point iv) of Theorem 
\ref{Theorem : MAIN THM GEN}). 

By Proposition \ref{Prop : (C3) for R_1} we know that  
$H^{\Fs}_1=\R^{\Fs}_1$ satisfies (C3) with respect to $\Gamma:=\Gamma_X$.

\subsection{Property (C3) for $H_i^{\Fs}$}\label{perrohgfhxfd}
Let $D\subset X$ be an open disk on which 
$\R^{\Fs}_1,\ldots,\R^{\Fs}_{i-1}$ are constant with value 
$R_1,\ldots R_{i-1}$ i.e. $D\cap\Gamma_{i-1}=\emptyset$. 
Let $b_0:=1$, and if $1\leq k\leq i-1$ let 
$b_{k}:=\prod_{j=1}^{k}R_j$. 
Then $H_i^{\Fs}=b_{i-1}\cdot \R^{\Fs}_i$ over $D$. 
The functions $H_i^{\Fs}$ and $\R^{\Fs}_i$ then have the 
same properties over $D$. In particular
\begin{equation}\label{eq : GammaH_i=GammaR_i}
\Gamma(H_i^{\Fs})\cap D\;=\;\Gamma(\R_i^{\Fs})\cap D\;.
\end{equation}
The following proposition asserts that $\R^{\Fs}_i$ 
coincide over $D$ with the first radius $\R^{\Fs_{\geq i}}_1$ of a 
certain sub-module $\Fs_{\geq i}\subseteq\Fs_{|D}$ coming from 
Theorem \ref{deco on a disc}. 

This is the crucial property for the 
induction in the proof of Theorem \ref{Theorem : MAIN THM GEN}. It 
constitutes a generalization to higher radii of the Transfer principle 
(cf. Proposition \ref{Prop : (C3) for R_1}).
\begin{proposition}[Transfer principle]\label{Prop. : (C3)}
With the above setting two situations are possible over $D$:
\begin{enumerate}
\item The function $\R^{\Fs}_{i}$ is also constant on $D$;
\item $\R^{\Fs}_{i-1}(x)<\R^{\Fs}_i(x)$ for all $x\in D$, 
and one has a decomposition $\Fs_{|D}=\Fs_{\geq i}\oplus \Fs_{<i}$ 
as in Theorem \ref{deco on a disc} satisfying moreover 
$\R^{\Fs}_i(x)=\R^{\Fs_{\geq i}}_1(x)$ for all $x\in D$.
\end{enumerate}
In particular $\R^{\Fs}_i$ and $H_i^{\Fs}$ verify (C3) with respect to 
$\Gamma:=\Gamma_{i-1}$, and they both enjoy all the properties of 
a first radius of convergence outside $\Gamma_{i-1}$. 
\end{proposition}
\begin{proof}
Assume that $\R^{\Fs}_i$ is not constant over $D$. In this case 
we have $\R^{\Fs}_k=\R^{\Fs_{|D}}_k$ over $D$ for all $k\leq i$. 
Indeed, by \eqref{(2)}, the non constancy gives for all $x\in D$, 
$\R^{\Fs}_k(x)\leq \R_{i}^{\Fs}(x)\leq 
\rho_{\R_{i}^{\Fs}}(x)<\rho$, where $\rho$ is the radius of $D$. 
Hence $\R_k^{\Fs_{|D}}(x)=\min(\R_k^{\Fs}(x),\rho)=
\R_k^{\Fs}(x)$.

The functions 
$H_i^{\Fs}$ and $\R^{\Fs}_i$ have the same slopes over $D$ 
because $H_i^{\Fs}=b_{i-1}\R^{\Fs}_i$. 
So $\R^{\Fs}_i$ 
verifies the concavity property (b) of point iii) of Theorem 
\ref{Theorem : MAIN THM GEN} which we have already proved. 

Hence, if $c\in D$ is a $K$-rational point, and if $R:=\R_i^{\Fs}(c)$, 
then along the segment 
$]x_{c,R},x_{c,\rho}[$ we must have 
$\R^{\Fs}_i>\R^{\Fs}_{i-1}$, because $\R_{i}^{\Fs}$ is concave 
 on it, while $\R^{\Fs}_{i-1}$ is constant on $D$.

Now,  by \eqref{eq : R_i VS R_isp along a segment}, 
$\R_i^{\Fs}$ is spectral 
along $]x_{c,R},x_{c,\rho}[$, 
hence $\R_{i-1}^{\Fs}$ is spectral non solvable on it.
So, by Theorem \ref{deco on a disc}, there exists a unique direct sum 
decomposition 
$\Fs_{|D}=\Fs_{\geq i}\oplus \Fs_{<i}$ 
such that for all $x\in]x_{c,R},x_{c,\rho}[$ one has 
$\R^{\Fs_{\geq i}}_{k-i+1}(x)=
\R^{\Fs}_{k}(x)$, 
for $k=i,\ldots,r$, and $\R^{\Fs_{<i}}_{k}(x)=
\R^{\Fs}_{k}(x)$, 
for $k=1,\ldots,i-1=\mathrm{rank}(\Fs_{<i})$. 

We now prove that these equalities hold for all $x\in D$, 
the claim  will then follow.
%$\R^{\Fs}_i(x)=\R^{\Fs_{\geq i}}_1(x)$ and 
%$\R^{\Fs}_{i-1}(x)=\R^{\Fs_{< i}}_{i-1}(x)$, for all $x\in D$.

By Proposition \ref{Prop. direct sum then ok with radii} the convergence 
radii of $\Fs$ at $x$ are the union (with multiplicities) of those of 
$\Fs_{\geq i}$ and of 
$\Fs_{<i}$. So it is enough to prove that for all $x\in D$ 
one has $\R_{i-1}^{\Fs_{<i}}(x)< \R^{\Fs_{\geq i}}_{1}(x)$. 
\begin{lemma}
Let $x\in D$. If $\R_{i-1}^{\Fs}(x)< \R^{\Fs_{\geq i}}_{1}(x)$, 
then 
$\R_{i-1}^{\Fs_{<i}}(x)< \R^{\Fs_{\geq i}}_{1}(x)$.
\end{lemma}
\begin{proof}
Since $\Fs_{\geq i}\subseteq\Fs$, the assumption implies  
$\{\R^{\Fs_{\geq i}}_{1}(x),\ldots,
\R^{\Fs_{\geq i}}_{r-i+1}(x)\}\subset
\{\R^{\Fs}_{i}(x),\ldots,\R^{\Fs}_{r}(x)\}$. 
Moreover by point (7) of Remark \ref{Def.: radius on a disk},
the two multi-sets must coincide since they are 
equipotent. By difference, this implies that 
$\R^{\Fs_{<i}}_{k}(x)=\R^{\Fs}_k(x)$, for all $k\leq i-1$. 
\end{proof}
\if{More precisely it is enough to prove that for all $x\in D$ one has 
$\R_{i-1}^{\M}(x)< \R^{\N}_{1}(x)$. Indeed
we must have 
$\{\R^{\N}_{1}(x),\ldots,\R^{\N}_{r-i+1}(x)\}\subset
\{\R^{\M}_{i}(x),\ldots,\R^{\M}_{r}(x)\}$. 
Since the radii of $\N$ appears in the list at most with the multiplicity 
of those of $\M$ the two multi-sets must coincide. This also implies, 
by difference, that 
$\R^{\mathrm{Q}}_{j}(x)=\RM_j(x)$, for all $j\leq i-1$. 
 }\fi
We now prove that $\R_{i-1}^{\Fs}(x)< \R^{\Fs_{\geq i}}_{1}(x)$, 
for all $x\in D$. 
Since the radii are insensitive to scalar extensions of 
$K$, we can assume that $x$ is $K$-rational. 

Now $\R^{\Fs_{\geq i}}_{1}$ is $\log$-concave with non positive 
$\log$-slopes along $[x,x_{c,\rho}[$. 
Then for all $\rho'$ close enough to $\rho$ one has 
$\R^{\Fs_{\geq i}}_1(x)\geq \R^{\Fs_{\geq i}}_1(x_{c,\rho'})=
\R^{\Fs}_{i}(x_{c,\rho'})>
\R^{\Fs}_{i-1}(x_{c,\rho'})=\R^{\Fs}_{i-1}(x)$ as desired. 
\if{
On the other hand since $H^{\M_2}_{j}$ verifies (WC3), for all $j=1,\ldots,i-1$, and induction on $j$ shows that $H^{\M_2}_j$
then for all $\rho'$ close enough to $\rho$ one has $\R^{\M_2}_{i-1}(t)\leq\R^{\M_2}_{i-1,t}(\rho')=\RM_{i-1,t}(\rho')=\RM_{i-1}(t)$. 
....

Let $\rho_1:=\max(\rho'\leq\rho,\;\textrm{ such that }\RM_{i,t}(\rho')\geq\rho')$. 
Since $\RM_i$ is non constant, then $\RM_{i-1}(t)<\rho_1$. 
Since $\RM_{i,t}=\R_{1,t}^{\M_1}$ over the segment $]\RM_{i-1}(t),\rho_1[$ this proves that $\RM_i$ is concave at $\rho_1$ 
and hence that $\rho_1=\RM_i(t)$. So $\RM_{t,i}$ is constant on $[0,\rho_1]$ by \eqref{...E.E.} 
as well as $\R^{\M_1}_{1,t}$, and hence $\RM_{i,t}=\R^{\M_1}_{1,t}$ all along $[0,\rho]$.

 $\R^{\M_1}_{k,t}(\rho')=\RM_{i+k-1,t}(\rho')$ for all $\rho'\in]\RM_{i-1}(t),\rho_1[$, and all $k=1,\ldots,r_1$, 
this proves that $\RM_{i}(t)=\rho_1$ and by \eqref{...E.E.} there is no intervals $J\subseteq ]0,\rho[$ on which $\RM_{i,t}(\rho')=\rho'$ for all 
$\rho'\in J$. 
Now by Proposition \ref{Prop. direct sum then ok with radii} the radii of $\M$ at all $\xi'$ 
are the union of those of $\M_1$ and of $\M_2$, for 
}\fi
\end{proof}

\begin{remark}
By Prop. \ref{Prop : negative slope}
we have
$\Gamma(\R^{\Fs}_i)\cap D\neq\emptyset$ 
if and only if 
$\partial_b\R^{\Fs}_{i}(x)>0$, where $x$ is the point at the boundary 
of $D$, and $b$ is the germ of segment out of $x$ lying in $D$ 
oriented as inside $D$.
\end{remark}

\subsection{Finiteness (C5) and super-harmonicity (C6)}
The following crucial lemma describes the locus of points 
where $\R^{\Fs}_i$ is solvable or over-solvable:
\begin{lemma}\label{Lemma : KEY . .!}
If $\R^{\Fs}_i(x)\geq r(x)$, then either  
$x\notin \Gamma(\R^{\Fs}_i)$ or, if $x\in\Gamma(\R_i^{\Fs})$, then:
\begin{enumerate}
\item If $x\in\Gamma(\R^{\Fs}_i)-\Gamma_{i-1}$, then $x$ 
is a boundary point of $\Gamma(\R^{\Fs}_i)$;
\item If $x\in\Gamma_{i-1}\cap\Gamma(\R^{\Fs}_i)$, then 
$\Delta(x,\Gamma(\R^{\Fs}_i))\subseteq\Delta(x,\Gamma_{i-1})$.
\end{enumerate}
\end{lemma}
\begin{proof}
Assume $x\in\Gamma(\R_i^{\Fs})$. It is enough to prove that $\R^{\Fs}_i$ is constant on each open disk 
$D\subset X$ with boundary $x$ such that 
$D\cap \Gamma_{i-1}=\emptyset$. Let $c\in D$ be a rational point.
By Proposition \ref{Prop. : (C3)} the function $\R^{\Fs}_i$ enjoys 
concavity properties on $D$, so 
$\R^{\Fs}_i(c)\geq\R^{\Fs}_i(x)\geq r(x)$. 
Since $r(x)$ coincides with the radius of $D$, this means that 
$\R^{\Fs}_i$ is constant on $D$ by 
\eqref{(2)}.
\end{proof}
\if{
An easy induction proves the following claim
\begin{lemma}
Let $D\subset X$ be an open disk with boundary $x\in X$, and let $b$ 
be the germ of segment out of $x$ inside $D$. 
Let $\RM_1,\ldots,\RM_i$ be the list of the radii that are spectral non 
solvable at $x$. If $\partial_b\RM_j(x)=0$ for all $j=1,\ldots,i$, 
then $\RM_1,\ldots,\RM_r$ are all constant on $D$.\hfill $\Box$
\end{lemma}
}\fi
The following statement will be the base case of our induction in the 
proof of Theorem \ref{Theorem : MAIN THM GEN}.
\begin{proposition}\label{Case i=1}
Theorem \ref{Theorem : MAIN THM GEN} holds for $\R^{\Fs}_1$.
\end{proposition}
\begin{proof}
By Proposition \ref{Prop : negative slope}, to prove directional 
finiteness (C5) we shall prove 
that $\partial_b\R^{\Fs}_1(x)=0$ for almost but a finite number of germ of 
segments out of $x$. This follows from Propositions 
\ref{Prop : SH-well done spectral NS} %and \ref{Prop. : (C3)} 
if the index $i=1$ is not solvable at $x$, and from
Lemma \ref{Lemma : KEY . .!} %and Proposition \ref{Prop. : (C3)} 
if $i=1$ is solvable at $x$. 

The super-harmonicity properties iii) and iv) of Theorem 
\ref{Theorem : MAIN THM GEN} for $\R^{\Fs}_1$, 
follow again from Proposition 
\ref{Prop : SH-well done spectral NS} if the index $i=1$ is not solvable 
at $x$. 

Otherwise, if $i=1$ is solvable at $x$, then we have three cases: 

If 
$x\notin\Gamma(\R^{\Fs}_1)$ there is nothing to prove; 

If $x\in\Gamma(\R^{\Fs}_1)-\Gamma_X$, then $x$ is a boundary point 
of $\Gamma(\R^{\Fs}_1)$ by Lemma \ref{Lemma : KEY . .!}, 
in this case the super-harmonicity is just the concavity property (C3) 
of Proposition \ref{Prop. : (C3)};

If $x\in\Gamma_X$, then $\Gamma(\R^{\Fs}_1)=\Gamma_X$ around 
$x$ by Lemma \ref{Lemma : KEY . .!}. 
Moreover the function $y\mapsto\R^{\Fs}_1(y)$ is bounded by 
$y\mapsto\rho_{y,X}$ around $x$, and the two functions are equal at 
$x$. So $\R^{\Fs}_1$ is super-harmonic at $x$, by 
Lemma \ref{Lemma: G dominates F, so F super-H}.
\end{proof}

The following two propositions conclude the proof of 
Theorem \ref{Theorem : MAIN THM GEN}.
\begin{proposition}
If $H_1^{\Fs},\ldots,H_{i-1}^{\Fs}$ are finite, then $H_i^{\Fs}$ is 
directionally finite (C5).
\end{proposition}
\begin{proof}
Let $x\in \Gamma(H_{i}^{\Fs})$ be a bifurcation point. 
We have to prove that 
there are a finite number of open disks 
$D\subset X$ with boundary $x$ such that 
$D\cap\Gamma(H_i^{\Fs})\neq\emptyset$. Since $\Gamma_{i-1}$ is 
finite, there are a finite number of such disks intersecting it, so we 
can neglect them.

By \eqref{eq : GammaH_i=GammaR_i} (cf. also Remark \ref{Rk : defof RRM_i}), we have
$\Gamma(\R^{\Fs}_i)-\Gamma_{i-1}=\Gamma(H^{\Fs}_i)-\Gamma_{i-1}$. 
So we can replace $H^{\Fs}_i$ by $\R^{\Fs}_i$, and apply Proposition 
\ref{Prop. : (C3)} to have the properties of a first radius over each 
open disk $D$ with boundary $x$ such that 
$\Gamma_{i-1}\cap D=\emptyset$. In particular by  
Proposition \ref{Prop : negative slope}, $H^{\Fs}_i$ is constant over 
such a disk $D$ if and only if $\partial_bH^{\Fs}_i(x)= 0$, where $b$ 
is the germ of segment out of $x$ inside $D$.

As in the proof of Proposition \ref{Case i=1}, directional finiteness (C5) is then consequence of Proposition 
\ref{Prop : SH-well done spectral NS} (if $i$ is  not solvable 
at $x$) and Lemma \ref{Lemma : KEY . .!} (if $i$ is solvable at $x$).
\if{directional finiteness (C5) then follows from Propositions 
\ref{Prop : SH-well done spectral NS} and \ref{Prop. : (C3)} 
if the index $i$ is spectral at $x$, or from 
Lemma \ref{Lemma : KEY . .!} and Proposition \ref{Prop. : (C3)} 
if $i$ is solvable or over-solvable at $x$.

By assumption $\S_{i-1}$ is finite.
If $\xi\notin\S_{i-1}$, then, $\RM_1,\ldots,\RM_{i-1}$ being constant outside $\S_{i-1}$, 
the function $H_i^{\M}$ is directionally finite at $\xi$ 
if and only if so does $\RM_i$. By Prop. \ref{Prop. : (C3)} the $\RM_i$ acquires the properties of a genuine radius outside $\S_{i-1}$, 
so it is finite and super-harmonic on each non generic disk $\mathrm{D}^{-}(t,\rho)$ tangent to $\S_{i-1}$. 
Let $\xi\in\S_{i-1}$. If $\RM_i(\xi)\geq r(\xi)$ then one applies Lemma \ref{Lemma : KEY . .!} to prove the directional finiteness of $H_i^{\M}$ 
since $\S(H_{i}^{\M})-\S_{i-1}=\S(\RM_{i})-\S_{i-1}$ (cf. Remark \ref{Rk : defof RRM_i}). 
If $\RM_i(\xi)<r(\xi)$, by Prop. \ref{Prop. : (C3)} one has $\partial_-H^{\M}_{i,\delta}(\xi)<0$, 
for all $\delta\in\Delta(\xi,\S(H_i^{\M}))-\Delta(\xi,\S_{i-1})$. 
Now by Prop. \ref{Prop : R(xi) <r(xi) sup-harm} there are a finite number of directions $\delta\in\Delta(\xi)$ such that 
$\partial_-H^{\M}_{i,\delta}(\xi)\neq 0$. Hence $\Delta(\xi,\S(H_i^{\M}))$ is finite.
}\fi
\end{proof}

\begin{proposition}\label{Prop. : fin and SH of H_i}
If $H_1^{\Fs},\ldots,H_{i-1}^{\Fs}$ satisfy Theorem 
\ref{Theorem : MAIN THM GEN}, 
then so does $H_i^{\Fs}$.
\end{proposition}
\begin{proof}
It remains to prove that $H_i^{\Fs}$ verifies the super-harmonicity 
property iv) of Theorem \ref{Theorem : MAIN THM GEN}.
This will guarantee that $H_i^{\Fs}$ fulfill the assumptions (C1)--(C6) 
of Thm. \ref{th : finiteness theorem} 
with respect to $\Gamma:=\Gamma_{i-1}$ and 
$\C(H_i^{\Fs}):=\C_i$. 
Notice that $\mathscr{C}_i\subseteq\Gamma_{i-1}$ is finite by 
(b) and (c) of point iv) Theorem \ref{Theorem : MAIN THM GEN}.

If $i$ is free of solvability then we deduce the super-harmonicity from Proposition 
\ref{Prop : SH-well done spectral NS}. 

It remains to prove that if $x\in X-(\mathscr{C}_i\cup\partial X)$, 
and if some index $j\leq i$ is solvable at $x$, then $H_i^{\Fs}$ is super-harmonic at $x$. 

Since $\C_{i-1}\subset\C_i$, by induction 
$H_{i-1}^{\Fs}$ is super-harmonic at $x\notin\C_i\cup\partial X$. We 
then write
\begin{equation}\label{eq : H_i=H_i-1R_i}
H_{i}^{\Fs}\;=\;
H_{i-1}^{\Fs}\cdot\R^{\Fs}_i\;.
\end{equation}

If $x\notin\Gamma_{i-1}$, then $H_i^{\Fs}$
enjoys the properties of a first radius of convergence outside 
$\Gamma_{i-1}$ by Proposition \ref{Prop. : (C3)}, 
so it is super harmonic at $x$ by Proposition \ref{Case i=1}. 

If $x\notin\Gamma(H^{\Fs}_i)$, then $H_i^{\Fs}$  
is constant around $x$ (hence harmonic at $x$). 

If $x\in\Gamma(H_i^{\Fs})\cap\Gamma_{i-1}$, we now prove that 
$\R^{\Fs}_i$ is super-harmonic at $x$. By \eqref{eq : H_i=H_i-1R_i} this will imply that   
$H_i^{\Fs}$ is super-harmonic at $x$.

If $i$ is over-solvable at $x$, or if $x\notin\Gamma(\R^{\Fs}_i)$, then 
$\R^{\Fs}_i$ is constant around $x$, and hence it is super-harmonic at $x$. 

It remains to check the case  
where $i$ is solvable at $x$ (i.e. $\R^{\Fs}_i(x)=r(x)$) and 
\begin{equation}
x\;\in\;
\Gamma(H_i^{\Fs})\cap\Gamma_{i-1}\cap\Gamma(\R^{\Fs}_i)\;.
\end{equation}

We have to prove that $\R^{\Fs}_i$ is super-harmonic at the 
points of that graph that are not in $\C_{i-1}$ nor in the boundary of 
$\Gamma(\R^{\Fs}_i)$. As observed, these points are finite in number because this admissible graph is finite by induction.

By Lemma \ref{Lemma : KEY . .!} we have the inclusion 
$\Delta(x,\Gamma(\R^{\Fs}_i))\subseteq\Delta(x,\Gamma_{i-1})$. So $\Gamma(\R^{\Fs}_i)$ is finite around $x$. 
Now since $x$ is not a 
boundary point of $\Gamma(\R^{\Fs}_i)$, the function 
$\rho_{\Gamma(\R^{\Fs}_i)}:X\to\mathbb{R}$ is super-harmonic at 
$x$ (as in point $1.$ of Example \ref{ex: exercice ttt}). 
Moreover, by \eqref{(2)}, and by point i) of Proposition 
\ref{prof propp}, for all  $b\in\Delta(x)$ we have respectively
\begin{equation}
\partial_b\R^{\Fs}_i(x)\;\leq\;\partial_b\rho_{\Gamma(\R^{\Fs}_i)}(x)\;,
\qquad\R^{\Fs}_i(x)=\rho_{\Gamma(\R^{\Fs}_i)}(x)=r(x)\;.
\end{equation}
The function $\R^{\Fs}_i$ is then super-harmonic by Lemma 
\ref{Lemma: G dominates F, so F super-H}.
\end{proof}
\if{
\subsection{Proof of Corollary \ref{Corollary - after thm}}
\begin{proof}[Proof of Corollary \ref{Corollary - after thm}]
By translation we can assume $t=0$. Let $\O$ be equal to one of $\mathcal{H}_K(0,I)$, $\mathcal{B}_K(0,I)$, $\a_K(0,I)$.
Almost all assertions can be proved from Thm. \ref{Theorem : MAIN THM GEN} 
by restriction to a sub-annulus (resp. sub-disk) $\{|T|\in J\}$, with $J$ compact (resp. $0\in J$). 
The unique assertion that remains to prove are the global finiteness of each $H^{\M}_i$, and the fact that along the segment 
$\lambda_0(I):=\{\xi_{0,\rho}\}_{\rho\in I}$ each $H^{\M}_i$ has a finite number of breaks. 
By Cor. \ref{cor : finiteness on a disk}, these two assertions are in fact equivalent, and this proves iii). This also proves the last assertion. 
Namely assume that, for $i\leq r$, all $H^{\M}_1,\ldots,H_i^{\M}$ have a finite number of breaks along $\lambda_0(I)$. 
Then $\RM_1=H^{\M}_1$ is finite by Cor. \ref{cor : finiteness on a disk}, 
and an induction on $k\leq i$ shows that $\S(H^{\M}_k)$ is contained in $C(0,J_k)\cup\lambda_0(I)$ for some compact $J_k$. 
Indeed each new complete segment generates by super-harmonicity a break along $\lambda_0(I)$.
We now prove i) and ii). If $\O=\mathcal{H}_K(0,I)$ or if $\O=\mathcal{B}_K(0,I)$ with $K$ discretely valued, 
then $\RM_i$ and $H_i^{\M}$ always have a finite number of breaks along $\lambda_0(I)$ 
by \cite[Thm. 11.3.2, Remark 11.3.4]{Kedlaya-Book}. The proof ends here, but for the convenience of the reader 
we now show why this is true.
We assume that $I$ is open, since otherwise the result is Thm. \ref{Theorem : MAIN THM GEN}.
For all $i$ one has $\RM_i=\R_i^{\M,\mathrm{sp}}$ along $\lambda_0(I)$.
Let $L_i:=\lim_{\rho\to s^-}\R^{\M,\mathrm{sp}}_{i}(\xi_{0,\rho})$, where $s:=\sup(I)$. 
The limit exists in $[0,s]$ since  $H_i^{\M,\mathrm{sp}}$ is concave along $I$, for all $i$.
 If $L_1=s$,\footnote{This condition is called solvability in \cite{Astx}.} 
then the sequence of slopes of $\R^{\M,\mathrm{sp}}_1$ is decreasing (by concavity), contained in 
$\mathbb{Z}\cup\frac{1}{2}\mathbb{Z}\cup\cdots\cup\frac{1}{r}\mathbb{Z}$, and lower bounded by $1$ (by definition of spectral radius). 
So the sequence of slopes is constant for $\rho\to s^-$ : there is a ``last slope'' (this argument is due to Christol-Mebkhout \cite{Astx}).
So $\RM_1$ is $\log$-linear outside some compact $J_1\subseteq I$, and hence finite by Cor. \ref{cor : finiteness on a disk}. 
Since $H_2^{\M}$ is $\log$-concave, then $\RM_2$ is concave outside $J_1$, so the same argument proves that $\RM_2$ is $\log$-linear 
outside some $J_2\supseteq J_1$, and hence finite by Cor. \ref{cor : finiteness on a disk}. By induction all $\RM_i$ are finite.
Assume now that $L_1<s$. Let $i_0$ be the larger value of $i$ such that $L_{i}<s$.
We perform Frobenius push-forward to have $L_{i_0}<\omega s$. By continuity there exists $\varepsilon>0$ such that 
$\RM_i(\xi_{0,\rho})<\omega\rho$ for all $\rho\in[s-\varepsilon,s[$.
To prove that the number of breaks of $\RM_1$ is finite on $[s-\varepsilon,s[$ one has to perform 
a global push-forward over $[s-\varepsilon,s[$ in order to control simultaneously all the $\{\RM_i(\xi_{0,\rho})\}_{\rho\in[s-\varepsilon,s[}$. 
So one argues as in section \ref{Push-forward by Frob} replacing $\H(\xi)$ by $\mathcal{H}_K([s-\varepsilon,s[)$ or 
$\mathcal{B}_K([s-\varepsilon,s[)$. One has the same results as in section \ref{Beh push}.
One sees from \eqref{eq : phi_*(G)}, that if the coefficients of $G(\widetilde{T})$ are bounded (resp. 
analytic elements), then so does $\varphi_*(G)(T)$. Indeed the sequence $\{a_i\}_i$ of the Taylor coefficients of the entries 
$h_{i,j}(T)=\sum_{i\in\mathbb{Z}}a_iT^i$ of $\varphi_*(G)(T)$ are obtained as sub-sequences of those 
of $G(\widetilde{T})$. Now we perform a base change in the fraction field of $\O$, and possibly restrict the annulus, 
to find a differential operator that again has bounded coefficients (resp. analytic elements). Now bounded functions (resp. 
analytic elements) have a finite number of zeros\footnote{Bounded function have a finite number of zeros if and only if $K$ 
has a discrete valuation \cite{Christol-Book}.} so the Newton polygon of the operator has a finite number of slopes. This
proves that $\RM_1,\ldots,\RM_{i_0}$ have a finite number of breaks along $\lambda_{0}(I)$, 
and hence that they are finite by Cor. \ref{cor : finiteness on a disk}. 
For $i\geq i_0+1$ one proceeds inductively using the above argument of Christol-Mebkhout, 
to prove the finiteness of $\RM_i$.
\end{proof}
}\fi
\section{Notes.}
\label{Notes final}
%\textbf{\textsc{Notes.}} 
A first proof of the harmonicity properties is due to P.Robba \cite{RoIII} and \cite{RoIV} 
for rank one differential equations with rational coefficients. 
He obtained the harmonicity of the radius function by expressing its slopes by means of the index 
(cf. \cite[Thm. 4.2, p.201]{RoIII}), and then deducing the harmonicity from the additivity of the indexes 
(cf. \cite[Prop.4.5, p.207]{RoIII}).\footnote{The principle is 
used in \cite[top of p.50]{RoIV} to construct 
the so called Robba's exponentials.} 

In a recent paper of Kedlaya \cite[Section 5]{Kedlaya-turritin2var} 
there is a proof of the finiteness of a certain function related to the 
radii of a differential equation over a surface. 
It is a function on a Berkovich closed unit disk over $K:=k(\!(z)\!)$, 
where $k$ is a trivially valued field.   
\if{The function considered by \cite{Kedlaya-turritin2var} is equal to 
the \emph{irregularity} of the module 
i.e. the height of the Newton polygon of the module depending on the 
Berkovich point 
(in the spirit of section \ref{NP of a module} and 
\cite{Kedlaya-Book}).}\fi 
The definitions of \cite{Kedlaya-turritin2var} are given \emph{ad hoc} 
to deal with a closed 
disk and there are discrepancies with those of this paper, especially for 
the definition of the skeleton 
of a function (which is defined in \cite{Kedlaya-turritin2var} in term of 
the slopes). 
It turns out that the two definitions eventually coincide over a closed 
disk (by Lemma \ref{Lemma : break at rho_F}), and in fact certain techniques of this paper 
are not far from those of \cite{Kedlaya-turritin2var} and 
\cite{Kedlaya-Book}.

A proof of the finiteness of the first radius function have been obtained by 
G.Christol \cite{Ch-Formula} 
for differential equations of rank one with polynomial coefficients. 
The proof uses an explicit formula for $\R_1^{\Fs}$ that we 
have contributed to realize (cf. the introduction of \cite{Ch-Formula}). 
The generalization of such a formula to rank one 
differential equation with arbitrary coefficients is the object of a 
forthcoming paper. 
This have been the starting point of the present paper. 

After the first version of the present paper is appeared 
(cf. \cite{NP-I},\cite{NP-II}), 
another proof of the finiteness of the controlling graphs 
have been obtained in \cite{Kedlaya-draft}. 
K.S.Kedlaya obtains there a shorter derivation 
of our proof, based on the same methods. 
No description of the super-harmonicity locus is given\footnote{Of 
course the finiteness of the locus of points where the 
super-harmonicity fails is a consequence of the finiteness of 
the controlling graphs and (C2).}. 
On the other hand he obtains a deep result showing that 
the end points of the controlling graphs can not be of type 
$4$. 

We also quote the remarkable result of Y.André 
 about the semi-continuity of the irregularity for 
meromorphic connections in a relative context \cite{Andre-cont}. 
With our notations the irregularity is (related to) the slope $\partial_bH_r^{\Fs}(x)$: 
it is the derivative of the height of the convergence Newton polygon.

\bibliographystyle{amsalpha}
\bibliography{bib}

\providecommand{\bysame}{\leavevmode\hbox to3em{\hrulefill}\thinspace}
\providecommand{\MR}{\relax\ifhmode\unskip\space\fi MR }
\providecommand{\MRhref}[2]{%
  \href{http://www.ams.org/mathscinet-getitem?mr=#1}{#2}
}
\providecommand{\href}[2]{#2}
\begin{thebibliography}{DMR07}

\bibitem[And02]{An}
Y.~Andr{\'e}, \emph{Filtrations de type {H}asse-{A}rf et monodromie
  {$p$}-adique}, Invent. Math. \textbf{148} (2002), no.~2, 285--317.

\bibitem[And07]{Andre-cont}
Yves Andr{\'e}, \emph{Structure des connexions m\'eromorphes formelles de
  plusieurs variables et semi-continuit\'e de l'irr\'egularit\'e}, Invent.
  Math. \textbf{170} (2007), no.~1, 147--198. \MR{2336081 (2008m:32049)}

\bibitem[Bal10]{Balda-Inventiones}
Francesco Baldassarri, \emph{Continuity of the radius of convergence of
  differential equations on {$p$}-adic analytic curves}, Invent. Math.
  \textbf{182} (2010), no.~3, 513--584. \MR{2737705 (2011m:12015)}

\bibitem[Ber90]{Ber}
Vladimir~G. Berkovich, \emph{Spectral theory and analytic geometry over
  non-{A}rchimedean fields}, Mathematical Surveys and Monographs, vol.~33,
  American Mathematical Society, Providence, RI, 1990.

\bibitem[BGR84]{BGR}
S.~Bosch, U.~G{\"u}ntzer, and R.~Remmert, \emph{Non-{A}rchimedean analysis},
  Grundlehren der Mathematischen Wissenschaften [Fundamental Principles of
  Mathematical Sciences], vol. 261, Springer-Verlag, Berlin, 1984, A systematic
  approach to rigid analytic geometry. \MR{MR746961 (86b:32031)}

\bibitem[BR10]{Baker-Book}
Matthew Baker and Robert Rumely, \emph{Potential theory and dynamics on the
  {B}erkovich projective line}, Mathematical Surveys and Monographs, vol. 159,
  American Mathematical Society, Providence, RI, 2010. \MR{2599526
  (2012d:37213)}

\bibitem[BV07]{DV-Balda}
F.~Baldassarri and L.~Di Vizio, \emph{Continuity of the radius of convergence
  of $p$-adic differential equations on berkovich spaces}, arXiv, 2007,
  \url{http://arxiv.org/abs/0709.2008}, pp.~1--22.

\bibitem[CD94]{Ch-Dw}
G.~Christol and B.~Dwork, \emph{Modules diff\'erentiels sur des couronnes},
  Ann. Inst. Fourier (Grenoble) \textbf{44} (1994), no.~3, 663--701.
  \MR{MR1303881 (96f:12008)}

\bibitem[Chr77]{Christol-GEAU}
Gilles Christol, \emph{Structure de {F}rob\'enius des \'equations
  diff\'erentielles {$p$}-adiques}, Groupe d'\'{E}tude d'{A}nalyse
  {U}ltram\'etrique, 3e ann\'ee (1975/76), {F}asc. 2 ({M}arseille-{L}uminy,
  1976), {E}xp. {N}o. {J}5, Secr\'etariat Math., Paris, 1977, p.~7. \MR{0498578
  (58 \#16673)}

\bibitem[Chr11]{Ch-Formula}
\bysame, \emph{The radius of convergence function for first order differential
  equations}, Advances in non-{A}rchimedean analysis, Contemp. Math., vol. 551,
  Amer. Math. Soc., Providence, RI, 2011, pp.~71--89. \MR{2882390}

\bibitem[CM93]{Ch-Me-I}
G.~Christol and Z.~Mebkhout, \emph{Sur le th\'eor\`eme de l'indice des
  \'equations diff\'erentielles {$p$}-adiques. {I}}, Ann. Inst. Fourier
  (Grenoble) \textbf{43} (1993), no.~5, 1545--1574. \MR{1275209 (95j:12009)}

\bibitem[CM97]{Ch-Me-II}
\bysame, \emph{Sur le th\'eor\`eme de l'indice des \'equations
  diff\'erentielles {$p$}-adiques. {II}}, Ann. of Math. (2) \textbf{146}
  (1997), no.~2, 345--410. \MR{1477761 (99a:12009)}

\bibitem[CM00]{Ch-Me-III}
\bysame, \emph{Sur le th\'eor\`eme de l'indice des \'equations
  diff\'erentielles {$p$}-adiques. {III}}, Ann. of Math. (2) \textbf{151}
  (2000), no.~2, 385--457. \MR{1765703 (2001k:12014)}

\bibitem[CM01]{Ch-Me-IV}
\bysame, \emph{Sur le th\'eor\`eme de l'indice des \'equations
  diff\'erentielles {$p$}-adiques. {IV}}, Invent. Math. \textbf{143} (2001),
  no.~3, 629--672. \MR{1817646 (2002d:12005)}

\bibitem[CM02]{Astx}
Gilles Christol and Zoghman Mebkhout, \emph{\'{E}quations diff\'erentielles
  {$p$}-adiques et coefficients {$p$}-adiques sur les courbes}, Ast\'erisque
  (2002), no.~279, 125--183, Cohomologies $p$-adiques et applications
  arithm{\'e}tiques, II. \MR{1922830 (2003i:12014)}

\bibitem[CP09]{Rk1-non-perf}
Bruno Chiarellotto and Andrea Pulita, \emph{Arithmetic and differential {S}wan
  conductors of rank one representations with finite local monodromy}, Amer. J.
  Math. \textbf{131} (2009), no.~6, 1743--1794. \MR{2567506 (2011a:14036)}

\bibitem[Cre00]{Crew-Can-ext}
Richard Crew, \emph{Canonical extensions, irregularities, and the {S}wan
  conductor}, Math. Ann. \textbf{316} (2000), no.~1, 19--37. \MR{MR1735077
  (2001c:14039)}

\bibitem[DGS94]{DGS}
B.~Dwork, G.~Gerotto, and F.~J. Sullivan, \emph{An introduction to
  {$G$}-functions}, Annals of Mathematics Studies, vol. 133, Princeton
  University Press, Princeton, NJ, 1994.

\bibitem[DMR07]{Correspondance-Malgrange-Ramis}
Pierre Deligne, Bernard Malgrange, and Jean-Pierre Ramis, \emph{Singularit\'es
  irr\'eguli\`eres}, Documents Math\'ematiques (Paris) [Mathematical Documents
  (Paris)], 5, Soci\'et\'e Math\'ematique de France, Paris, 2007,
  Correspondance et documents. [Correspondence and documents]. \MR{2387754
  (2009d:32033)}

\bibitem[DR77]{Dw-Robba}
B.~Dwork and P.~Robba, \emph{On ordinary linear {$p$}-adic differential
  equations}, Trans. Amer. Math. Soc. \textbf{231} (1977), no.~1, 1--46.
  \MR{0447247 (56 \#5562)}

\bibitem[DR80]{Dwork-Robba-Growth}
\bysame, \emph{Effective {$p$}-adic bounds for solutions of homogeneous linear
  differential equations}, Trans. Amer. Math. Soc. \textbf{259} (1980), no.~2,
  559--577. \MR{567097 (81k:12022)}

\bibitem[Duc]{Duc}
Antoine Ducros, \emph{La structure des courbes analytiques},
  \url{http://www.math.jussieu.fr/~ducros/livre.html}.

\bibitem[Dwo73]{DW-II}
B.~Dwork, \emph{On {$p$}-adic differential equations. {II}. {T}he {$p$}-adic
  asymptotic behavior of solutions of ordinary linear differential equations
  with rational function coefficients}, Ann. of Math. (2) \textbf{98} (1973),
  366--376. \MR{0572253 (58 \#27987b)}

\bibitem[FJ04]{Favre-Jonsson}
Charles Favre and Mattias Jonsson, \emph{The valuative tree}, Lecture Notes in
  Mathematics, vol. 1853, Springer-Verlag, Berlin, 2004. \MR{2097722
  (2006a:13008)}

\bibitem[HL10]{loeser}
Ehud Hrushovski and Francois Loeser, \emph{Non-archimedean tame topology and
  stably dominated types}, to appear in Annalso of Math., 2010,
  \url{http://arxiv.org/abs/1009.0252}.

\bibitem[Hub96]{Huber-Book}
Roland Huber, \emph{\'{E}tale cohomology of rigid analytic varieties and adic
  spaces}, Aspects of Mathematics, E30, Friedr. Vieweg \& Sohn, Braunschweig,
  1996. \MR{1734903 (2001c:14046)}

\bibitem[Kat87]{Katz-cyclic-vect}
Nicholas~M. Katz, \emph{A simple algorithm for cyclic vectors}, Amer. J. Math.
  \textbf{109} (1987), no.~1, 65--70. \MR{878198 (88b:13001)}

\bibitem[Ked04]{Ked}
Kiran~S. Kedlaya, \emph{A {$p$}-adic local monodromy theorem}, Ann. of Math.
  (2) \textbf{160} (2004), no.~1, 93--184. \MR{MR2119719 (2005k:14038)}

\bibitem[Ked10a]{Kedlaya-turritin2var}
\bysame, \emph{Good formal structures for flat meromorphic connections, {I}:
  surfaces}, Duke Math. J. \textbf{154} (2010), no.~2, 343--418. \MR{2682186
  (2011i:14041)}

\bibitem[Ked10b]{Kedlaya-Book}
\bysame, \emph{p-adic differential equations}, Cambridge Studies in Advanced
  Mathematics, vol. 125, Cambridge Univ. Press, 2010.

\bibitem[Ked13]{Kedlaya-draft}
\bysame, \emph{Local and global structure of connections on nonarchimedean
  curves}, arxiv, 2013, \url{http://arxiv.org/abs/1301.6309}, pp.~1--76.

\bibitem[Laz62]{Lazard}
Michel Lazard, \emph{Les z\'eros des fonctions analytiques d'une variable sur
  un corps valu\'e complet}, Inst. Hautes \'Etudes Sci. Publ. Math. (1962),
  no.~14, 47--75. \MR{0152519 (27 \#2497)}

\bibitem[Mar04]{Adri-Swan}
Adriano Marmora, \emph{Irr\'egularit\'e et conducteur de {S}wan {$p$}-adiques},
  Doc. Math. \textbf{9} (2004), 413--433 (electronic). \MR{MR2117421
  (2005i:11169)}

\bibitem[Mat95]{Ma}
Shigeki Matsuda, \emph{Local indices of {$p$}-adic differential operators
  corresponding to {A}rtin-{S}chreier-{W}itt coverings}, Duke Math. J.
  \textbf{77} (1995), no.~3, 607--625. \MR{MR1324636 (97a:14019)}

\bibitem[Meb02]{Me}
Z.~Mebkhout, \emph{Analogue {$p$}-adique du th\'eor\`eme de {T}urrittin et le
  th\'eor\`eme de la monodromie {$p$}-adique}, Invent. Math. \textbf{148}
  (2002), no.~2, 319--351.

\bibitem[MR83]{Reversat}
Michel Matignon and Marc Reversat, \emph{Sur les automorphismes continus
  d'extensions transcendantes valu\'ees}, J. Reine Angew. Math. \textbf{338}
  (1983), 195--215. \MR{684023 (85f:12013)}

\bibitem[Pon00]{Pons}
Emilie Pons, \emph{Modules différentiels non solubles. rayons de convergence et
  indices}, Rend.Sem.Mat.Padova \textbf{103} (2000), 21--45.

\bibitem[Poo93]{Poonen}
Bjorn Poonen, \emph{Maximally complete fields}, Enseign. Math. (2) \textbf{39}
  (1993), no.~1-2, 87--106. \MR{1225257 (94h:12005)}

\bibitem[PP12a]{Potentiel}
Andrea Pulita and Jérôme Poineau, \emph{Continuity and finiteness of the radius
  of convergence of a $p$-adic differential equation \emph{via} potential
  theory}, To appear in Crelle's Journal, 2012,
  \url{http://arxiv.org/abs/1209.6276}, pp.~1--20.

\bibitem[PP12b]{NP-II}
\bysame, \emph{The convergence newton polygon of a $p$-adic differential
  equation ii : Continuity and finiteness on berkovich curves}, arxiv, 2012,
  \url{http://arxiv.org/abs/1209.3663}, pp.~1--16.

\bibitem[PP13a]{NP-III}
\bysame, \emph{The convergence newton polygon of a $p$-adic differential
  equation iii : global decomposition and controlling graphs}, arxiv, 2013,
  \url{http://arxiv.org/abs/1308.0859}, pp.~1--81.

\bibitem[PP13b]{NP-IV}
\bysame, \emph{The convergence newton polygon of a $p$-adic differential
  equation iv : local and global index theorems}, arxiv, 2013,
  \url{http://arxiv.org/abs/1309.3940}, pp.~1--44.

\bibitem[Pul12]{NP-I}
Andrea Pulita, \emph{The convergence newton polygon of a $p$-adic differential
  equation i : Affinoid domains of the berkovich affine line}, arxiv, 2012,
  \url{http://arxiv.org/abs/1208.5850v1}, pp.~1--52.

\bibitem[Ram78]{Ramis-Gevrey}
J.-P. Ramis, \emph{D\'evissage {G}evrey}, Journ\'ees {S}inguli\`eres de {D}ijon
  ({U}niv. {D}ijon, {D}ijon, 1978), Ast\'erisque, vol.~59, Soc. Math. France,
  Paris, 1978, pp.~4, 173--204. \MR{542737 (81g:34010)}

\bibitem[Rob75]{RoI}
P.~Robba, \emph{On the index of {$p$}-adic differential operators. {I}}, Ann.
  of Math. (2) \textbf{101} (1975), 280--316. \MR{0364243 (51 \#498)}

\bibitem[Rob80]{Robba-Hensel}
\bysame, \emph{Lemmes de {H}ensel pour les op\'erateurs diff\'erentiels.
  {A}pplication \`a la r\'eduction formelle des \'equations diff\'erentielles},
  Enseign. Math. (2) \textbf{26} (1980), no.~3-4, 279--311 (1981). \MR{610528
  (82k:12022)}

\bibitem[Rob84]{RoIII}
P.~Robba, \emph{Indice d'un op\'erateur diff\'erentiel {$p$}-adique. {III}.
  application to twisted exponential sums}, Asterisque \textbf{119-120} (1984),
  191--266.

\bibitem[Rob85]{RoIV}
\bysame, \emph{Indice d'un op\'erateur diff\'erentiel {$p$}-adique. {IV}. {C}as
  des syst\`emes. {M}esure de l'irr\'egularit\'e dans un disque}, Ann. Inst.
  Fourier (Grenoble) \textbf{35} (1985), no.~2, 13--55.

\bibitem[Thu05]{Amaury-These}
Amaury Thuillier, \emph{Théorie du potentiel sur les courbes en géométrie
  analytique non archimédienne. applications à la théorie d'arakelov}, Thèse de
  l'Université de Rennes 1 \textbf{N.ordre 3231} (2005), viii+185.

\bibitem[Tsu98]{Tsu-Swan}
Nobuo Tsuzuki, \emph{The local index and the {S}wan conductor}, Compositio
  Math. \textbf{111} (1998), no.~3, 245--288. \MR{MR1617130 (99g:14021)}

\bibitem[You92]{Young}
Paul~Thomas Young, \emph{Radii of convergence and index for {$p$}-adic
  differential operators}, Trans. Amer. Math. Soc. \textbf{333} (1992), no.~2,
  769--785. \MR{1066451 (92m:12015)}

\end{thebibliography}

\end{document}